\documentclass[11pt, reqno]{amsart}

\usepackage{amssymb, amscd, mathrsfs, wasysym,color} 
\usepackage{a4wide}
\usepackage{esint}

\usepackage{hyperref}
\usepackage{graphicx}
\numberwithin{equation}{section}
\theoremstyle{plain}
        \newtheorem{theorem}{Theorem}[section]
        \newtheorem{proposition}[theorem]{Proposition}
        \newtheorem{lemma}[theorem]{Lemma}
        \newtheorem{corollary}[theorem]{Corollary} 
        \newtheorem{conjecture}[theorem]{Conjecture}
        \newtheorem{definition}[theorem]{Definition} 
        \newtheorem{remark}[theorem]{Remark}  
\newtheorem*{theorem*}{Theorem}
\newtheorem*{proposition*}{Proposition}

\newcommand{ \black  }{\color{black} }

\let\oldmarginpar\marginpar
\renewcommand\marginpar[1]{\-\oldmarginpar[\raggedleft\footnotesize #1]%
{\raggedright\footnotesize #1}}
\newcommand \be {\begin{equation}}
\newcommand \ee {\end{equation}}
\newcommand \la \langle
\newcommand \ra \rangle

\newcommand \lan {\langle}
\newcommand \ran {\rangle}

\newcommand{\R}{\mathbb{R}}
\newcommand{\C}{\mathbb{C}}
\newcommand{\T}{\mathbb{T}}
\newcommand{\Z}{\mathbb{Z}}
\newcommand{\N}{\mathbb{N}}
\newcommand{\GG}{\mathfrak G}

\newcommand{\BB}{\mathcal B}
\newcommand{\TT}{\mathcal T}
\newcommand{\RR}{\mathcal R}

\begin{document}
\title
[Weakly nonlinear large box limit for 2D NLS] 
{The Weakly nonlinear large box limit of the 2D cubic nonlinear Schr\"odinger
equation}



\thanks{E. Faou is supported by the ERC Starting Grant project GEOPARDI}
\thanks{P. Germain is partially supported by NSF grant DMS-1101269, a start-up grant from the Courant Institute, and a Sloan fellowship.}
\thanks{Z. Hani is supported by a Simons Postdoctoral Fellowship and NSF Grant DMS-1301647.}

\author
    [E. Faou, P. Germain, and Z. Hani]
    {Erwan Faou, Pierre Germain, and Zaher Hani}

\address{Erwan Faou \\ INRIA \& ENS Cachan Bretagne, Campus de Ker Lann, Avenue Robert Schumann, 35170 Bruz, France. }
\email{Erwan.Faou@inria.fr}

\address{Pierre Germain \\ Courant Institute of Mathematical Sciences, 251 Mercer Street, New York 10012-1185
NY, USA}
\email{pgermain@cims.nyu.edu}
\address{Zaher Hani\\ Courant Institute of Mathematical Sciences, 251 Mercer Street, New York 10012-1185
NY, USA}
\email{hani@cims.nyu.edu}

\date{\today}

\subjclass[2000]{}

\begin{abstract} 
We consider the cubic nonlinear Schr\"odinger (NLS) equation set on a two dimensional box of size $L$ with periodic boundary conditions. By taking the large box limit $L \to \infty$ in the weakly nonlinear regime (characterized by smallness in the critical space), we derive a new equation set on $\R^2$ that approximates the dynamics of the frequency modes. This nonlinear equation turns out to be Hamiltonian and enjoys interesting symmetries, such as its invariance under the Fourier transform, as well as several families of explicit solutions. A large part of this work is devoted to a rigorous approximation result that allows to project the long-time dynamics of the limit equation into that of the cubic NLS equation on a box of finite size.
\end{abstract}
\maketitle

\tableofcontents


\section{Introduction}

The long-time behavior of nonlinear dispersive equations on compact domains is a subject that is remarkably rich yet very poorly understood. This complexity arises since dispersion does not translate into decay, neither at the linear nor at the nonlinear level, in contrast to the case of non-compact domains like $\R^d$. As a consequence, one loses all asymptotic stability results around equilibrium solutions (like the zero solution); such results are often the starting point in the study of asymptotic behavior in that setting. In fact, these equilibria are not expected to be stable in the long run, and completely out-of-equilibrium dynamics is anticipated. 

The purpose of the present work is to uncover new, and structured or coherent, aspects of this out-of-equilibrium behavior. Such coherent dynamics are best seen in frequency space and will be revealed in the \emph{large-volume limit} and \emph{weakly nonlinear} regime. The latter is determined by how the size of the data compares to that of the box, a factor that plays a key role in determining the limiting system as we argue below.

\subsection{Presentation of the equation}

Our model equation will be the 2D cubic nonlinear Schr\"odinger equation given by
\begin{equation}\label{cubicNLS}
\begin{cases}
-i\partial_t v +\Delta v&=\pm  |v|^2v \tag{NLS}\\
v(0,x)&= v_0(x)\qquad x\in \T^2_L=[0,L]\times[0, L].
\end{cases}
\end{equation}
where $v=v(t,x)$ is a complex-valued field, $t\in \R$ represents time, and the spatial domain is taken to be the box $\T^2_L$ of size $L$ with periodic boundary conditions\footnote{ Since the equation respects the parity of the initial data, the periodic boundary condition include both Dirichlet and Neumann conditions corresponding respectively to odd and even solutions.}. Recall that this equation conserves mass and the Hamiltonian given respectively by:
\begin{equation}\label{conservation laws}
M[v(t)] \overset{def}{=} \int_{\T^2_L}|v(t,x)|^2\, dx\quad\text{and}\quad H[v(t)] \overset{def}{=} \frac{1}{2}\int_{\T^2_L}|\nabla v(t,x)|^2\,dx\pm \frac{1}{4}\int_{\T^2_L}|v(t,x)|^4\, dx.
\end{equation}

The sign of the non-linearity ($+$ for defocusing equation and $-$ for focusing equation) will not play a central part in the analysis; this is partly due to the fact that we focus on the long-time behavior of small initial data, i.e. ones that are close to the zero equilibrium $v=0$ in the scale-invariant topology (namely $L^2(\T^2_L)$). 

The size of the data is measured in comparison to the size $L$ of the box, and their relationship will be a defining factor for various regimes of long-time behavior. Consequently, it will be both convenient and more transparent to adopt the ansatz $v(t,x)=\epsilon u(t,x)$, where $\epsilon>0$. The field $u$ satisfies 
\begin{equation}\label{NLSe}
\begin{cases}
-i\partial_t u +\Delta u&=\pm \epsilon^2 |u|^2u\tag{NLS$_\epsilon$}\\
u(0,x)&= u_0(x).
\end{cases}
\end{equation}
When working with~\eqref{NLSe}, it makes most sense to consider that $\|u_0\|_{L^2}\sim 1$, since the $\epsilon$ parameter already accounts for the size of the data.

\subsection{Cubic NLS on $\mathbb{T}^2$ vs.~$\mathbb{R}^2$}
Before starting our analysis, let us review known results on the large-time behavior of solutions to \eqref{cubicNLS} on $\R^2$ and $\T^2$ (of course, it would be equivalent, but less convenient, to discuss \eqref{NLSe}) . 

The equation is mass-critical in the sense that the scaling symmetry given by $v \mapsto v_\lambda \overset{def}{=} \lambda v (\lambda^2 t,\lambda x)$ leaves both the equation and the mass invariant\footnote{On the torus, scaling is not a perfectly honest symmetry: $v_\lambda$ is a solution of \eqref{cubicNLS} on $\T^2_{\lambda^{-1}}$ rather than $\T^2$. Nevertheless, this symmetry keeps the mass $M[u]$ invariant, which extends the crucial notion of criticality to the compact setting.}.  This leads to global well-posedness in $H^s(\R^2)$ for all $s\geq 0$ and small mass (\cite{CW} or see \cite{TaoBook} for a textbook treatment). This  restriction on the mass can be removed completely in the defocusing case and relaxed in the focusing case up to the mass of the ground state \cite{KTV, Dodson1, Dodson2}. More importantly, the asymptotic behavior near the zero equilibrium is completely understood in this Euclidean setting: any small mass solution scatters to a linear solution as $t\to \pm \infty$. This means that if $v_0\in H^s(\R^2)$ for some $s\geq 0$, there exists $\varphi_{\pm \infty}\in H^s(\R^2)$ such that 
$$
\|v(t)-e^{-it\Delta_{\R^2}}\varphi_{\pm \infty}\|_{H^s(\R^2)}\to 0 \qquad \text{as }t\to\pm \infty.
$$
Moreover, one has the following asymptotic stability statement which is a direct consequence of the small data theory: If $v_0\in H^s(\R^2)$ for some $s\geq 0$ is sufficiently small, then
\begin{equation}\label{Euclidean stability}
\|v(t)-e^{-it\Delta}v_0\|_{H^s(\R^2)}=O(\|v_0\|_{H^s}^3) \qquad\text{ for all }t\in \R.
\end{equation}

On~$\mathbb{T}^2$, local well-posedness was proved by Bourgain~\cite{Bourgain} for data in $H^s$, with $s>0$. This leads when combined with \eqref{conservation laws} to global well-posedness for data with finite energy (and small mass in the focusing case). We remark that since we will be only interested in the dynamics of \emph{small} solutions of \eqref{cubicNLS} with finite energy, or equivalently solutions of \eqref{NLSe} with $\epsilon \ll 1$, all solutions exist globally and conserve mass and energy. It is also worth pointing out that local well-posedness in the scale-invariant norm $L^2$ remains an open question on $\T^2$ due to the failure of the scale-invariant $L^4_{t,x}$ Strichartz estimate (see \eqref{Strichartz}). Rather surprisingly, although we will only be interested in sufficiently regular solutions that exist globally by elementary arguments, this endpoint Strichartz estimate and the sharp constants associated to it will play a central role in our analysis (see Section \ref{tri est section}). 

When it comes to long-time behavior of \eqref{cubicNLS} on $\T^2$, we observe a number of phenomena rather than the single-scenario situation of the Euclidean setting. Indeed, the expectation is that \eqref{cubicNLS} is a \emph{non-integrable} infinite dimensional Hamiltonian system, and as a consequence it should sustain, just as in the finite dimensional case, orbits exhibiting strikingly different asymptotic behaviors. Known results include KAM tori and quasi-periodic solutions \cite{Geng, Procesi, BourgainQuasi, EliaKuk} and heteroclinic and unbounded Sobolev orbits (see \cite{CKSTT, Hani, Guardia} for partial results in this direction).  In particular, no long-time stability result anywhere close to \eqref{Euclidean stability} can hold for general data. For instance, one of the byproducts of the constructions in this paper is a whole regime of initial data violating \eqref{Euclidean stability} at time scales of roughly $\|v_0\|_{L^2}^{-2}$.

\subsection{The weakly nonlinear large-box limit} 
Equation \eqref{NLSe} has a particularly transparent form in Fourier space: Denoting by $\Z^2_L \overset{def}{=} L^{-1}\Z^2$ and expanding
$$
v(t, x)=L^{-2}\sum_{K \in \Z^2_L} a_K(t) e^{2\pi iK\cdot x},\qquad \text{where}\qquad a_K(t)=\int_{\T^2_L}v(t,x)e^{-2\pi iK\cdot x}\,dx
$$
one obtains that $a_K(t)$ solves the infinite dimensional system of ODE:
\begin{equation*}
-i\partial_t a_K(t)-4\pi^2 |K|^2a_K(t)=\pm \frac{\epsilon^2}{L^4}\sum_{(K_1, K_2, K_3)\in\mathcal S(K)} a_{K_1}(t) \overline{a_{K_2}(t)}a_{K_3}(t)
\end{equation*}
where $\mathcal S(K)=\{(K_1, K_2, K_3)\in \Z^2_L: K_1-K_2+K_3=K\}$. Obviously, the linear flow evolves much faster than the non-linear one, which makes it convenient to pass to the so-called ``interaction representation picture" by defining $\widetilde a_K(t)=e^{-4\pi^2i|K|^2t}a_K(t)$. The equation satisfied by $\widetilde a_K(t)$ is:
\begin{equation}\label{FNLSe}
-i\partial_t \widetilde a_K(t)=\pm\frac{\epsilon^2}{L^4}\sum_{\mathcal S(K)} \widetilde a_{K_1}(t) \overline{\widetilde a_{K_2}(t)}\widetilde a_{K_3}(t)e^{4\pi^2i\Omega t}
\end{equation}
where $\Omega \overset{def}{=} |K_1|^2-|K_2|^2+|K_3|^2-|K|^2$. This equation describes how the frequency mode $a_K$ is excited by other modes through the nonlinear interactions included in $\mathcal S(K)$.

Our aim is now to consider the weakly nonlinear ($\epsilon \to 0$) large-box ($L \to \infty$) limit, while keeping a balance between $\epsilon$ and $L$. We elaborate on this below.

\begin{itemize}
\item \underline{The weakly nonlinear limit} $\epsilon \to 0$ enables to restrict the above sum $\sum_{\mathcal S(K)}$ to \emph{resonant interactions}, which correspond to those $(K_1, K_2, K_3) \in \mathcal S(K)$ such that $\Omega=0$. This can be achieved by a normal form transformation. Restricting nonlinear interactions to resonant ones, one obtains the related \emph{resonant system} 
\begin{equation}\label{RS}
-i\partial_t  r_K(t)=\pm\frac{\epsilon^2}{L^4}\sum_{\mathcal R(K)}  r_{K_1}(t) \overline{ r_{K_2}(t)} r_{K_3}(t)\end{equation}
where $\mathcal R(K)=\{(K_1, K_2, K_3)\in \mathcal S(K): \Omega\overset{def}{=}|K_1|^2-|K_2|^2+|K_3|^2-|K|^2=0\}$. This system is also the first Birkhoff normal form approximation of \eqref{FNLSe} and will play a central role in our analysis. Its fundamental importance in describing the long-time behavior of \eqref{cubicNLS} has been realized and exploited in previous works (\cite{CKSTT, CF, Hani, HaPau}).
\item \underline{The large box limit} $L \to \infty$ is the thermodynamic or infinite volume limit: we are not concerned too much about the behavior of each single Fourier mode, but rather by ``macroscopic" interactions and variables. This perspective can be compared to that taken in statistical mechanics to study many-particle systems. Passing to the large-box limit is also one of the main ingredients of \emph{kinetic weak turbulence theory}, which seeks a statistical description of the out-of-equilibrium frequency dynamics of equations like \eqref{NLSe}. Taking the large-box limit, one obtains from \eqref{RS} a simpler system now set on $\R^2$. A very pedestrian explanation is the following: as $L$ goes to infinity, the lattice $\mathbb{Z}^2_L$ becomes more and more refined, allowing (in some sense which will be explained) to consider the right-hand side of \eqref{RS} as a Riemann sum, approaching a certain integral as $L \rightarrow \infty$.
\item \underline{The relative sizes of $\epsilon$ and $L$:} as measured by the quantity $\frac{(-\log \epsilon)}{\log L}$ are crucial in determining which limiting system best describes the dynamics in the corresponding range of $\epsilon$ and $L$. The main focus of this work is the regime where $\epsilon L \ll 1$. As we argue below, this is exactly the regime where the resonant dynamics dictated by \eqref{RS} dominate non-resonant interactions in \eqref{FNLSe}.
\end{itemize}
Taking the large-box limit of \eqref{RS} in a rigorous fashion entails several difficulties due to the non-linear lattice relation ``$\Omega=0$" defining the resonant manifold $\mathcal R(K)$. This is the reason why tools from analytic and geometric number theory will play a crucial role in the analysis, first at the level of deriving the limit equation and then at the level of proving the needed uniform estimates on resonant sums like that in \eqref{RS}. Without going into the details of the derivation for now, the offspring of this analysis is a nonlinear integro-differential equation, now set on $\R^2$, given by
\begin{equation}\label{star}\tag{CR}
\boxed{\begin{split}
-i\partial_t g(\xi,t)=& \mathcal T(g,g,g)(\xi,t);\qquad \xi \in \R^2\\
\mathcal T(g,g,g)(\xi,t)=& \int_{-1}^1\int_{\R^2} g(\xi+\lambda z, t) \overline{g}(\xi+\lambda z+z^\perp)g(\xi+z^\perp) \,d\lambda\,dz.
\end{split}
}
\end{equation}
where we used $z^\perp$ to denote the rotation of $z\in \R^2$ by the angle of $+\pi/2$. The name \eqref{star} stands for \emph{continuous resonant} as it corresponds to a continuous limit of \eqref{RS}. Morally speaking, one should think of the frequency modes $a_K$ of \eqref{FNLSe} as being approximated by the trace\footnote{It is customary in the physics literature not to distinguish between the notation for $a_K(t)$ and its ``large-box limit" $g(t, K)$ and denote both by $a_K(t)$ with the understanding that $K\in \R^2$ after the large-box limit is taken.} of a \emph{rescaled version} of $g(t,\xi)$ on the lattice $\Z^2_L$. To the best of our knowledge, equation \eqref{star} is new. Upon analyzing it, one soon realizes that it is somewhat special in the sense that it enjoys rather unusual properties, symmetries, and even explicit solutions that we postpone elaborating on till the next section.

How does the new equation inform on the long-time behavior of \eqref{NLSe} on a box of finite size? In other words, can we prove that the nonlinear dynamics of \eqref{star} is a reflection or a ``subset" of the dynamics of \eqref{NLSe}? The positive answer to these questions is a major part of the analysis and is one of the novelties of this work. It is contained in Theorems \ref{approx thm1} and \ref{result on T^2}, whose precise statements we delay until we have set up all our notations and parameters in the next section. Roughly speaking, Theorem \ref{approx thm1} allows to project the dynamics of \eqref{star} onto that of \eqref{NLSe} posed on a box of size $L$, whereas Theorem \ref{result on T^2} shows how the dynamics of \eqref{star} embed into that of \eqref{cubicNLS} on the torus $\T^2$.

The weakly nonlinear, large-box regime is the usual set up for the theory of weak turbulence, also called wave turbulence. We present this theory briefly in the next section to draw the similarities and differences with what is done here. Equation \eqref{star} can be viewed as a deterministic version of the famous Kolmogorov-Zakharov (KZ)\footnote{Also called sometime \emph{wave-wave kinetic equation}.} equation of weak turbulence theory. The (KZ) equation is widely used in the physics and applied sciences (oceanography and atmospheric sciences) to understand several aspects of the frequency dynamics of equations like \eqref{NLSe}. Despite many similarities, the most notable difference between this work and the vast physics literature on (kinetic) wave turbulence is that our derivation does not involve any randomization of the data. This has the effect that the limiting equation \eqref{star} is time-reversible (even Hamiltonian), as opposed to (KZ) (cf.~Section \ref{WT subsection}). It should be noted that the proper and rigorous justification of the Kolmogorov-Zakharov equation is a very important, and difficult, open question both from the physical and mathematical viewpoint. We hope that this work is a step forward in this direction.

\subsection{Organization of the paper} The paper is organized as follows. We end this introduction by listing the notation used in the rest of the manuscript. In the next section, we sketch briefly the derivation of \eqref{star}, state the main approximation theorems \ref{approx thm1} and \ref{result on T^2} mentioned above, and summarize the properties of \eqref{star}. We also include in that section some heuristic interpretations of \eqref{star} and elaborate on its relation to weak turbulence theory. In section \ref{tri est section} we prove the crucial trilinear estimates on resonant sums like that in \eqref{RS}, which are of interest in their own right as they are intimately related to some sharp Strichartz estimates on $\T^2$. In Section \ref{T_L to T section}, we prove that the right-hand-side of \eqref{RS} converges in the limit of large $L$ to the right-hand-side of \eqref{star} up to important rescaling factors and obtain quantitative estimates on the errors incurred. In Section \ref{proof of approx theorem section} we give the proof of the approximation theorems. Afterwards, we embark on analyzing equation \eqref{star} in Sections \ref{sectionhamiltonian}-\ref{sectionvariational}, starting from its Hamiltonian properties, symmetries and invariances, explicit solutions, global well-posedness results, as well as variational properties and characterizations of some of the explicit solutions. 

\subsection{Notations}

The following notations are used. Recall that we denote $\Z^2_L=L^{-1}\Z^2$ and for any subset $\Lambda$ of $\Z^2_L$, $\Lambda^*= \Lambda \setminus \{0\}$.

\subsubsection{Fourier transform} \underline{On $\R^2$:} The Fourier transform of a function $f$ on $\mathbb{R}^2$ is denoted $\mathcal{F} f$ or $\widehat{f}$, and given by 
$$
\mathcal{F} f (\xi) = \widehat{f} (\xi) = \frac{1}{2\pi} \int_{\mathbb{R}^2} e^{-ix\cdot \xi}
f(x)\,dx \quad \mbox{so that} \quad f(x) = \frac{1}{2\pi} \int_{\mathbb{R}^2} e^{ix\cdot \xi}
\widehat{f}(\xi)\,d\xi.
$$
The map $f \mapsto \widehat{f}$ is an isometry on $L^2(\mathbb{R}^d)$. 

\medskip

\noindent \underline{On $\T^2_L$:} The Fourier transform of a function $u$ on $\mathbb{T}^2_L$ is the sequence $a_K=\widehat u(K)$ with $K \in \Z^2_L=L^{-1}\Z^2$ given by 
$$
a_K= \int_{\T_L^2} u(x)e^{-2\pi iK.x} \,dx \quad \mbox{so that} \quad u(x)=\frac{1}{L^2}\sum_{K \in \Z_L^2} a_K(t) e^{2\pi i K.x}.
$$
The map $u \mapsto \{a_K\}$ is an isometry from $L^2(\T_L^2)$ to $\ell^2_L$ defined below.

\subsubsection{Function spaces}
\noindent \underline{On $\Z^2_L$:} We will consider sequences $(a_K)_{K \in \Z^2_L}$ and use the normalized counting measure $L^{-2} \sum_{K \in \Z^2_L}\delta_K$ to define $\ell^p(\Z^2_L)$ norms for $p < \infty$. Therefore we set for $1 \leq p < +\infty,$
$$
\|a_K\|_{\ell^p_L}\overset{def}{=}\left(L^{-2}\sum_{K \in \Z^2_L} |a_K|^p\right)^{1/p}
\quad \mbox{and} \quad 
\| a_K \|_{\ell_L^\infty} \overset{def}{=} \sup_{K \in \Z^2_L} |a_K|, \quad \mbox{for} \quad p = + \infty. 
$$
The weighted $\ell_L^p(\Z^2_L)$ spaces will be denoted $\ell^{p,\sigma}_L$ with norms
$$
\|a_K\|_{\ell^{p,\sigma}_L}\overset{def}{=} \|\langle K \rangle^\sigma a_K\|_{\ell^{p,}_L}
$$
Of particular importance in the analysis will be the space $X^\sigma_L \overset{def}{=} \ell^{\infty,\sigma}_L$ with $\sigma \in \R$, the space associated with the norm 
\begin{equation}
\label{eq:defXsigmaL}
\|a_K\|_{X^\sigma_L}\overset{def}{=}\|\langle K \rangle^\sigma a_K\|_{\ell^\infty_L}.
\end{equation}
Notice that with our normalization of the $\ell^{2,s}_L(\Z^2_L)$ spaces we have that 
$$
\|a_K\|_{\ell^{2,s}_L(\Z^2_L)}\lesssim \|a_K\|_{X^\sigma_L(\Z^2_L)}\quad \mbox{ whenever }\sigma>s+1.
$$
 
\medskip

\noindent \underline{On $\R^2$ and $\T^2_L$:} We use the standard definitions and notations for Lebesgue, weighted Lebesgue (in their homogeneous and inhomogeneous versions), and Sobolev spaces on $\R^2$ and $\T^2_L$. For example, with $M$ standing for either $\R^2$ or $\T^2_L$ we denote the $L^p$, homogeneous weighted $L^p$, inhomogeneous weighted $L^p$, and Sobolev norms respectively by
\begin{align*}
& \|g\|_{L^p(M)}\overset{def}{=}\left(\int_{M} |g(x)|^p\,dx\right)^{1/p} \\
& \|g\|_{\dot L^{p,\sigma}(M)}\overset{def}{=}\| | x |^\sigma g(x)\|_{L^p(M)} \\
& \|g\|_{L^{p,\sigma}(M)}\overset{def}{=}\|\langle x \rangle^\sigma g(x)\|_{L^p(M)} \\
& \|g\|_{H^m(M)}\overset{def}{=}\|\langle \nabla \rangle^{m}g(x)\|_{L^2(M)},
\end{align*}
for $p \in [1,\infty]$ (with the usual modification for $p=\infty$), and $m,\sigma \in \R$.  Notice that the Sobolev norm of a function $\phi \in H^s(\T^2_L)$ is equivalent to $\|\widehat \phi\|_{\ell^{2,s}_L}$.
The inner product on $L^2(M)$ is denoted by 
$$\langle f,g \rangle \overset{def}{=} \int_{M} f(x) \overline{g(x)}\,dx.$$
Finally, because of their particular importance in our analysis, we will denote by $\dot X^\sigma=\dot L^{\infty,\sigma}(\R^2)$ and $X^\sigma:=L^{\infty,\sigma}(\R^2)$ with norm
\begin{equation}
\label{eq:defXsigma}
\|g\|_{X^\sigma}\overset{def}{=}\|\langle x \rangle^\sigma g\|_{L^\infty(\R^2)}.
\end{equation}

Note that for a given function in $g \in X^\sigma(\R^2)$, we can associate the trace or projection sequence $\{ g_K\}_{K \in \Z^2_L} \in X^\sigma_L(\Z^2_L)$ by the formula $g_K = g(K)$.

\subsubsection{Miscellaneous}

\begin{itemize}
\item The counter-clockwise rotation of center $0$ and angle $\frac{\pi}{2}$ on $\mathbb{R}^2$ will be denoted by the superscript ${\cdot}^\perp$.
\item The ``Japanese brackets" stand for $\lan x\ran \overset{def}{=} \sqrt{1+|x|^2}$.
\item The normalized Gaussian with $L^2$ mass one reads $G(x) \overset{def}{=} \frac{1}{\sqrt \pi} e^{-\frac{|x|^2}{2}}$.
\item $A \lesssim B$ if $A \leq C B$ for some implicit, universal constant $C$. $A\lesssim_\delta B$ means that the implicit constant $C$ depends on $\delta$.
\item $A \sim B$ if $A \lesssim B$ and $B \lesssim A$.
\item For any $m \in \R_+$, we denote by $\mathcal B(m)$ the box $[-\frac{m}{2}, \frac{m}{2})\times  [-\frac{m}{2}, \frac{m}{2})$ and by $B(x,m)$ the Euclidean ball with center $x$ and radius $m$.
\end{itemize}

\section{Statement of the results}\label{obtained results}
In this section, we present a very formal sketch of the derivation of \eqref{star} from \eqref{FNLSe}. This is done rigorously in the following sections. We then state the main results in this work starting with the main approximation theorems that allow to compare the dynamics of \eqref{star} and that of \eqref{cubicNLS} along with the crucial estimates involved in proving them. We believe the latter to be of interest in their own right. Afterwards, we state the properties of the \eqref{star} equation that are proved in Sections \ref{sectionhamiltonian}-\ref{sectionvariational}. Finally we present a heuristic interpretation of \eqref{star} and elaborate on how it fits in the weak turbulence picture. 

\subsection{Formal derivation of (CR)}\label{formal argument}
Recall that we denoted by $v(t,x)$ the solution of \eqref{cubicNLS}, by $u(t,x)\overset{def}{=}\epsilon^{-1} v(t,x)$ the solution of \eqref{NLSe}, and by $a_K(t)\overset{def}{=}\widehat u(t,K)$ for $K \in \Z^2_L$. We then went to the interaction representation picture and defined $\widetilde a_K(t)= a_K(t)e^{-4\pi^2i|K|^2t}$ which satisfies the system \eqref{FNLSe} 
\begin{equation*}
-i\partial_t \widetilde a_K(t)=\frac{\epsilon^2}{L^4}\sum_{\mathcal S(K)} \widetilde a_{K_1}(t) \overline{\widetilde a_{K_2}(t)}\widetilde a_{K_3}(t)e^{4\pi^2i\Omega t}
\end{equation*}
with $\Omega\overset{def}{=}|K_1|^2-|K_2|^2+|K_3|^2-|K|^2$ and $\mathcal S(K)=\{(K_1, K_2, K_3): K_1-K_2+K_3=K\}$. Note that $a_K(t)$ can be directly reconstructed from $\widetilde a_K(t)$ and vice-versa. Splitting the interactions on the right-hand-side of the above equation into resonant and non-resonant gives
\begin{equation*}
-i\partial_t \widetilde a_K(t)=\underbrace{\frac{\epsilon^2}{L^4}\sum_{\mathcal R(K)} \widetilde a_{K_1}(t) \overline{\widetilde a_{K_2}(t)}\widetilde a_{K_3}(t)}_{\text{Resonant interactions}}+\underbrace{\frac{\epsilon^2}{L^4}\sum_{\mathcal S(K)\setminus \mathcal R(K)} \widetilde a_{K_1}(t) \overline{\widetilde a_{K_2}(t)}\widetilde a_{K_3}(t)e^{4\pi^2i\Omega t}}_{\text{Non-resonant interactions}}
\end{equation*}
where $\mathcal R(K)=\{(K_1, K_2, K_3)\in \mathcal S(K): \Omega=0\}$. It is well-known that if $\epsilon$ is small enough resonant interactions have considerably more importance on the dynamics than non-resonant ones \cite{CKSTT, CF, Hani}. This is often made precise by resorting to a \emph{normal form transformation}. A normal form transformation is an invertible change of variable $\widetilde a_K \to c_K=\mathcal U(\widetilde a_K)$ that is sufficiently close to the identity operator and which transforms the equation for $\widetilde a_K$ into the following equation for $c_K(t)$:
\begin{equation*}
-i\partial_t c_K(t)=\frac{\epsilon^2}{L^4}\sum_{\mathcal R(K)} c_{K_1}(t) \overline{c_{K_2}(t)}c_{K_3}(t)+\epsilon^4 (\text{quintic term}).
\end{equation*}
Several normal form transformations give the above result, we mention Poincar\'e-Dulac and Birkhoff normal forms. We will use the somewhat simpler Poincar\'e-Dulac normal form, but leave the details to Section \ref{proof of approx theorem section} when we prove Theorem \ref{approx thm1}. If $\epsilon$ is small enough (depending on $L$ and the control we have on relevant norms of $\{c_K\}$), the resonant sum above dominates the remainder quintic term and drives the dynamics. It is here that the weak non-linearity assumption comes into the picture and dictates the range of $\frac{(-\log \epsilon)}{\log L}$ under consideration.

The upshot of the above discussion is that we can pretend, at least formally, that $\widetilde a_K(t)$ satisfies the resonant system \eqref{RS}, namely 
\begin{equation}\label{pretend eqn}
-i\partial_t \widetilde a_K(t)``="\frac{\epsilon^2}{L^4}\sum_{\mathcal R(K)} \widetilde a_{K_1}(t) \overline{\widetilde a_{K_2}(t)}\widetilde a_{K_3}(t)
\end{equation}
At this step, one can interpret $``="$ as an equality sign with some lower order terms added to the right-hand-side, but things will get even more formal below. The next step is to take the large-box limit of the above equation. This is equivalent to changing the above discrete sums in \eqref{pretend eqn} to integrals via a suitable equidistribution analysis. Recall that if $u(z)$ is a sufficiently regular and decaying function on $\R^2$ then 
$$
L^{-2}\sum_{K \in \Z_L^2} u(K) \to \int_{\R^2}u(z) \,dz, \quad \mbox{when} \quad L \to +\infty. 
$$
Of course the sum in \eqref{pretend eqn} is much more complicated than a standard Riemann sum, so our first task is to re-parametrize it in a way that allows us to pass to the large-box limit. For this, we change variables in $\mathcal R(K)$ and write $K_i=K+N_i$ for $i=1,2,3$. The defining relations of $\mathcal R(K)$ now translate to 
$N_1+N_3=N_2$ and $ |N_1|^2+|N_3|^2=|N_2|^2$.
As a result, we have by Pythagoras's theorem that
\begin{equation}
\label{eq:rect1}
\mathcal{R}(K) = \{(K + N_1, K + N_2 + N_3, K + N_3) , \quad N_1 \cdot N_3 = 0, \quad (N_1,N_3) \in \Z^2_L\}. 
\end{equation}
The degenerate sums when either $N_1$ or $N_3$ is zero disappear in the large-box limit and can be treated as error terms because they account for much fewer lattice points than the non-degenerate case when $N_1, N_3\neq 0$. This will be very clear when we do the estimates in Section \ref{T_L to T section}, but for now it can be seen from Figure \ref{figure1} which shows the plots of rectangles in $\mathcal R(K)$ with $K = (0,0)$ on lattices $\Z_L^2$ with different size of $L$. Notice how taking $L$ larger and larger makes the density of rectangles tend to a continuous one. 

\begin{figure}[h!]
\centering
\includegraphics[width=0.3\linewidth]{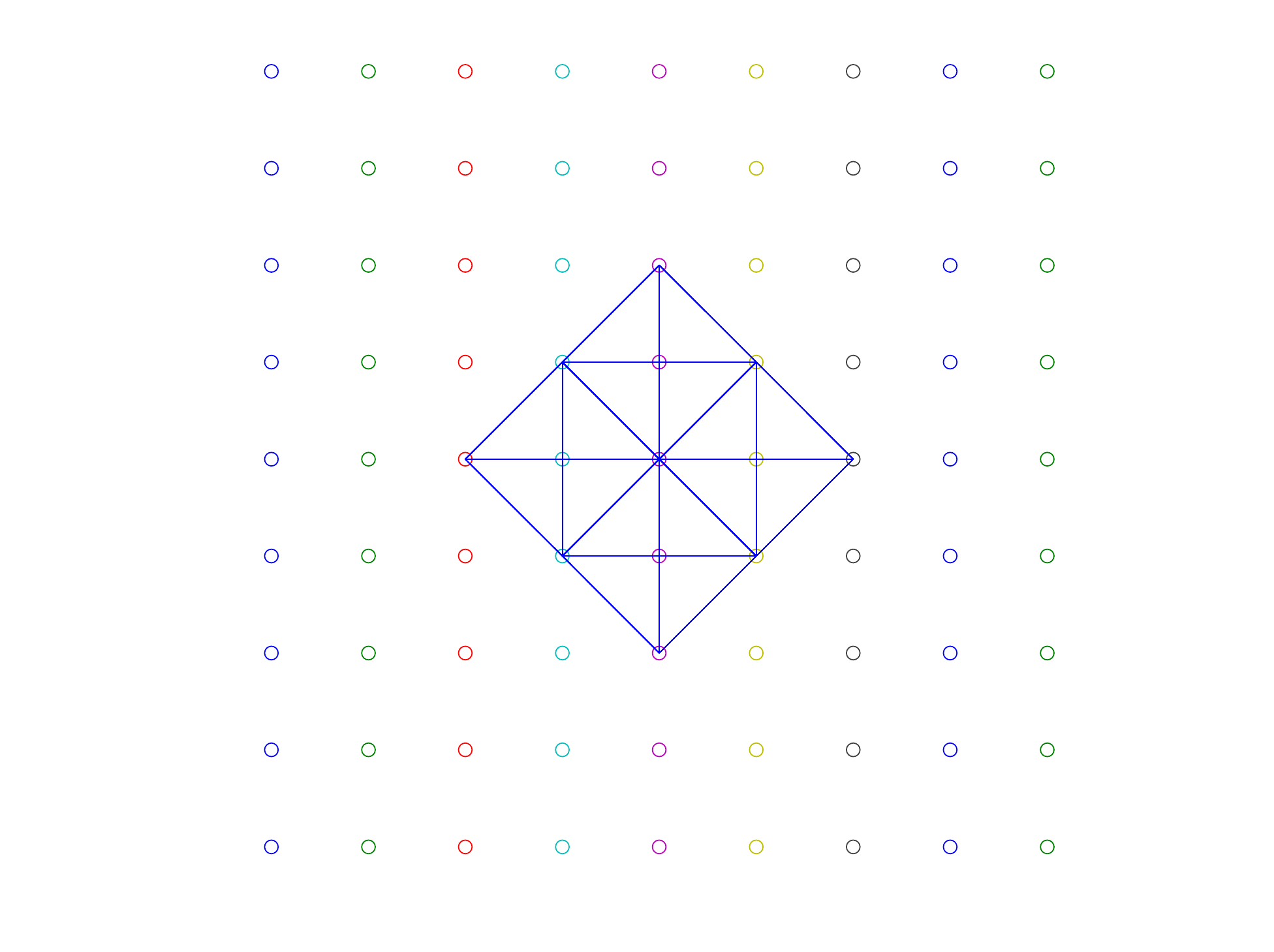}
\includegraphics[width=0.3\linewidth]{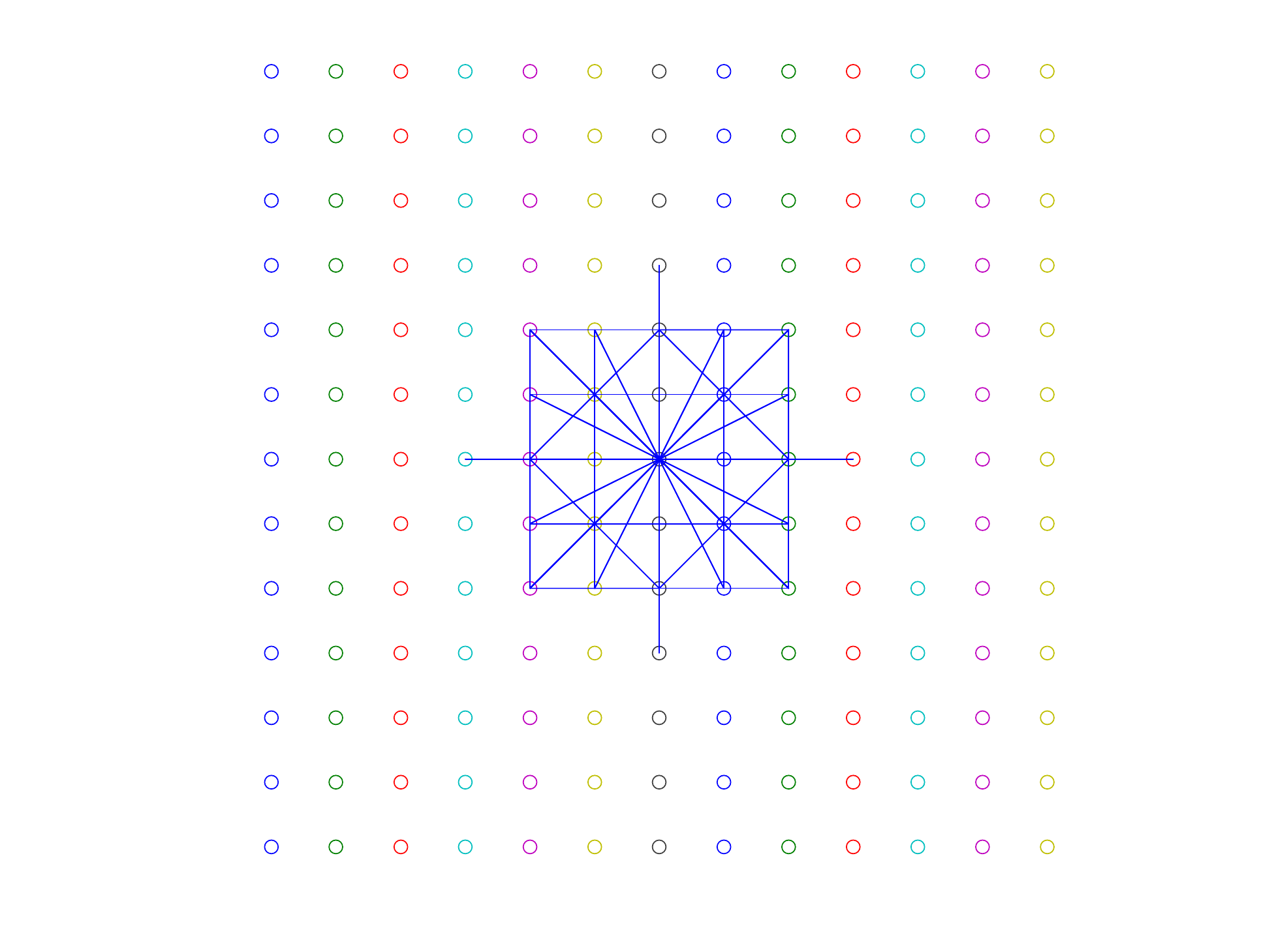}\\
\includegraphics[width=0.3\linewidth]{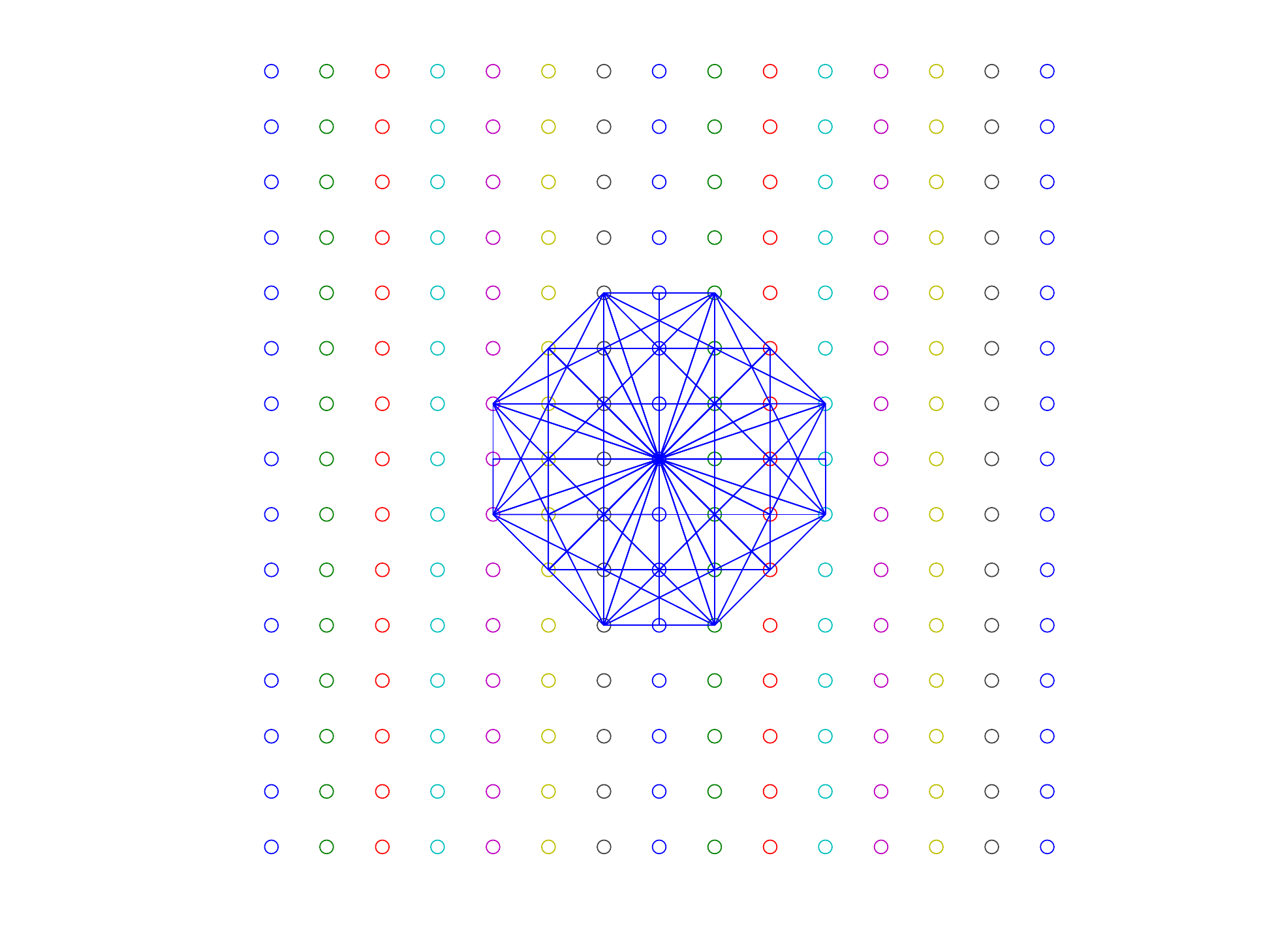}
\includegraphics[width=0.3\linewidth]{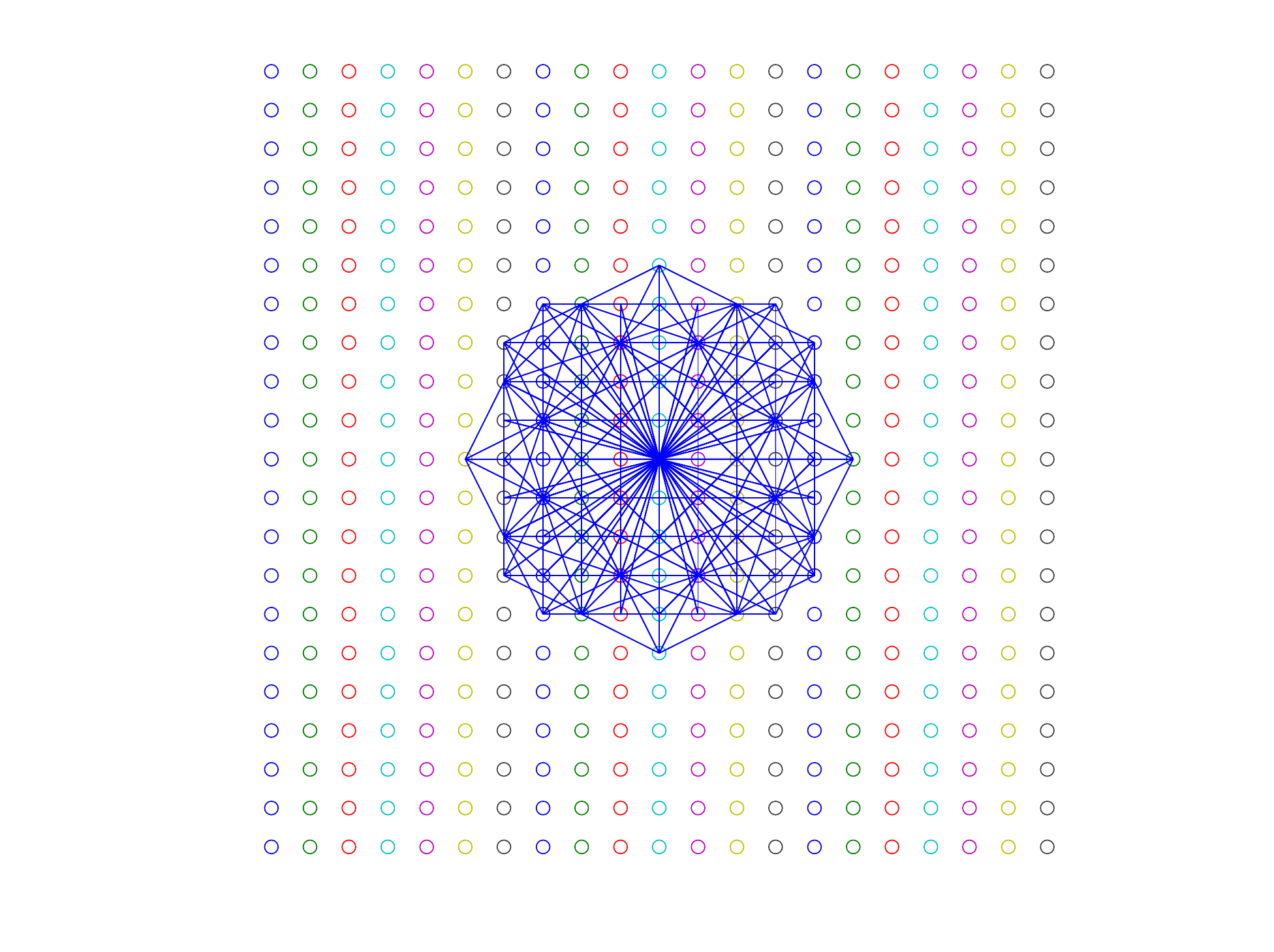}
\caption{Representations of the resonant sets $\mathcal{R}(0,0)$ when $L$ increases}\label{figure1}
\end{figure}

As a result, we can again pretend that $\widetilde a_K(t)$ satisfies the following system:
\begin{equation}\label{pretend eqn2}
-i\partial_t \widetilde a_K(t)``="\frac{\epsilon^2}{L^4}\sum_{\substack{N_1, N_3 \in {\Z^2_L}^*\\N_1 \perp N_3}} \widetilde a_{K +N_1}(t) \overline{\widetilde a_{K+N_1+N_3}(t)}\widetilde a_{K+N_3}(t).
\end{equation}
The above sum is still not very amenable to taking a large-box limit, so we further parametrize it as follows. Write $N_1=\frac{\alpha}{L}(p, q)$ where $\alpha \in \N$ and $p, q \in \Z$ are relatively prime in the sense that $g.c.d.(|p|, |q|)=1$\footnote{Here we use the convention that if $p$ (resp. $q$) is zero, then $p$ and $q$ are relatively prime if and only if $q=\pm1$ (resp. $p=\pm 1$).}. Since $N_3\perp N_1$, then it has to be of the form $ \frac{\beta}{L}(-q, p)=\frac{\beta}{L}(p, q)^\perp$ for some $\beta \in \Z^*$. Lattice points in $\Z^2_L$ of the form $\frac{(p, q)}{L}$ are special and are called \emph{visible lattice points}.
\begin{definition}
\label{defvisible}
For a given number $L> 0$, we say that a lattice point $z\in (\mathbb{Z}_L^2)^*$ is \emph{visible} (from $0$) if the segment $[0,z]$ does not intersect the lattice $\mathbb{Z}_L^2$ apart from the endpoints $0$ and $z$. It is easy to see that a lattice point $z \in {\mathbb{Z}^2_L}^*$ is visible if and only if $z=\frac{1}{L}(p,q)$ where $(p,q)$ are co-prime. \end{definition}

As a result, we can again morally pretend that $\widetilde a_K(t)$ satisfies
\begin{equation}\label{pretend eqn3}
-i\partial_t \widetilde a_K(t)``="\frac{\epsilon^2}{L^4}\sum_{\alpha\in \N, \beta \in \Z^*}\sum_{\substack{J \in {\Z^2_L}^*\\J \, visible}} \widetilde a(K +\alpha J,t) \overline{\widetilde a(K+\alpha J +\beta J^\perp,t)}\widetilde a(K+\beta J^\perp,t)
\end{equation}
where $J^\perp$ is the rotation of $J$ by $+\frac{\pi}{2}$. We are now in a better position to take the large-box limit. The only remaining obstruction is that the sum over $J$ is not a standard equidistribution sum over all lattice points in $\Z^2_L$ but only over visible ones. It turns out that visible lattice points have density $\zeta(2)^{-1}=\frac{6}{\pi^2}$ in $\Z^2_L$ where $\zeta(2)$ is the Riemann zeta function evaluated at 2. In other words, if one picks a lattice point at random then with probability $\zeta(2)^{-1}$ this lattice point is visible. The translation of this statement to equidistribution theory is the fact that if $u(z)$ is a sufficiently regular and decaying function then
$$
L^{-2}\sum_{\substack{J \in \Z^2_L\\J \; visible}} u(z)\, dz \to \frac{1}{\zeta(2)}\int_{\R^2}u(z) dz.
$$
The proof of the claim on the density of visible lattice points relies on the M\"obius inversion formula (see for example \cite{HW}). We briefly repeat it in Section \ref{co-prime equidistribution subsection} as it serves as a prototype for the argument needed to replace the more complicated trilinear sum (over $J$ visible) in \eqref{pretend eqn3} by an integral over $\R^2$ and obtain good estimates on the discrepancy error (see Lemma \ref{equidistribution lemma}). 

Now that we know how to replace sums over visible lattice points by integrals, we are ready to pass to the large-box limit, at least formally. Pretending that $\{\widetilde a(K, t)\}_{K \in \Z^2_L}$ is the trace of a sufficiently smooth and decaying function $\{\widetilde a(\xi, t)\}_{\xi \in \R^2}$ and replacing the sums in \eqref{pretend eqn3} by integrals, one obtains after a somewhat tedious calculation (performed in Section \ref{T_L to T section} with the needed error estimates) that $\widetilde a(K, t)$ satisfies the equation:
\begin{equation}\label{pretend eqn4}
-i\partial_t \widetilde a(K, t)`` ``=""\frac{4\epsilon^2 \log L}{\zeta(2) L^2} \int_{-1}^1 \int_{\mathbb{R}^2} \widetilde a(K+z) \overline{\widetilde a(K+z+\lambda z^\perp)}\widetilde a(K+\lambda z^\perp ) \,dz\,d\lambda, 
\end{equation}
where we used the notation $````=""$ to emphasize the heuristic nature of this last analysis. 

Reparameterizing time gives the equation
$$
-i\partial_t \widetilde a(K, t)`` ``=""\int_{-1}^1 \int_{\mathbb{R}^2} \widetilde a(K+z) \overline{\widetilde a(K+z+\lambda z^\perp)}\widetilde a(K+\lambda z^\perp ) \,dz\,d\lambda, 
$$
which is exactly \eqref{star}! Making sense of the above formal argument is not a straightforward task and several technical difficulties have to be overcome before any sort of rigorous statement can be made in this direction. We elaborate on some of them below.

\subsection{Uniform bounds and approximation theorems} The ultimate aim here is to understand and prove in what sense and to what extent do the long-time dynamics of \eqref{star} inform on that of \eqref{FNLSe}. Some answers are provided by Theorem \ref{approx thm1}, which pertains to \eqref{NLSe} on the box of size $L$, and Theorem \ref{result on T^2} which pertains to the \eqref{cubicNLS} on the torus $\T^2$ (or $L=1$).  
\subsubsection{Uniform bounds on resonant sums}
A crucial part in the proof of the above-mentioned theorems is obtaining \emph{sharp} bounds on resonant sums like that appearing on the right-hand-side of \eqref{pretend eqn}. More concretely, let us define the normalized trilinear operator
\begin{equation}\label{TL}
\mathcal{T}_L (e,f,g)(K) \overset{def}{=} \frac{\zeta(2)}{2L^2 \log L} \sum_{\mathcal{R}(K)} e_{K_1} f_{K_2} g_{K_3},
\end{equation}
acting on sequences $e = \{e_K\}$, $f = \{f_K\}$ and $g = \{g_K\}$ on $\Z^2_L$. The normalization of $\mathcal{T}_L$ was chosen merely for notational convenience so that the following sharp bounds appear uniform in $L$. Estimates on such operators are intimately connected to periodic Strichartz estimates and are in fact equivalent to the $L_x^2\to L_{t,x}^4$ estimate if one is bent on working in $L^2-$based spaces (see Section \ref{tri est section}). Notice that since we are essentially passing to the $L \rightarrow \infty$ limit, sharp estimates are indespensible  and even a logarithmic loss in $L$ cannot be tolerated. We will see that estimates with very small loss ($L^\epsilon$ for any $\epsilon$) can be somewhat easily proved from the existing literature on periodic Strichartz estimates (see \cite{Bourgain, Bourgain3} and Section \ref{non sharp estimates}), but they turn out to be almost completely useless by themselves to prove the approximation results. Unfortunately, the sharp version of the Strichartz estimate we need is open and is widely considered to be a very difficult problem at the intersection of harmonic analysis and number theory \cite{Bourgain2, Kishi}. Our way around this dilemma is to work with norms that are not $L^2-$based, namely spaces $X^\sigma_L(\Z^2_L)$ defined in \eqref{eq:defXsigmaL}, and prove the \emph{sharp} Stichartz-type estimate in these norms, which turns out to be more tractable than the $L^2-$based Strichartz estimate. The argument here relies on quite delicate estimates on lattices and the end result is the following theorem proved in Section \ref{tri est section}

\begin{theorem} \label{tri est theo}
Let $\sigma>2$ and $\{a_K\}, \{b_K\} $, and $\{c_K\}$ be three sequences in $X^\sigma_L$. The following estimate holds uniformly in $L$:
\begin{equation}\label{tri est}
 \left\| \mathcal{T}_L (a_K,b_K,c_K) \right\|_{X^\sigma_L} \lesssim \|a_K\|_{X^\sigma_L}  \|b_K\|_{X^\sigma_L}  \|c_K\|_{X^\sigma_L}.
\end{equation}
Moreover, the dependance in $L$ in this estimate is sharp.
\end{theorem}
\medskip
The true meaning of the last statement will become clear below where we show that the operator $\mathcal T_L$ is not only uniformly bounded in $X^\sigma_L$, but also converges as $L \to \infty$ to a \emph{non-zero} entity when it acts on sequences that are traces of sufficiently smooth and decaying functions. It is worth mentioning that as a corollary to Theorem \ref{tri est theo}, we obtain what seems to be the first $L^4_{t,x}$ Strichartz-type estimate on $\T^2$ with the sharp $L^2$-critical scaling (See Corollary \ref{new strichartz corollary}). 

\subsubsection{Discrete to continuous convergence} Let us first define the trilinear operator $\mathcal T$ acting on complex valued functions $f, g, h: \R^2 \to \C$ as follows
\begin{equation}
\label{eq:defT}
\mbox{if $\xi \in \mathbb{R}^2$}, \qquad \mathcal{T}(f,g,h) (\xi) \overset{def}{=}
\int_{-1}^1 \int_{\mathbb{R}^2} f(\xi+x) \overline{g(\xi+x+\lambda x^\perp)} h(\xi+\lambda x^\perp ) 
\,dx\,d\lambda.
\end{equation}
The crucial observation is that, as $L \rightarrow \infty$, $\mathcal{T}_L$ converges to $\mathcal{T}$ in a sense made precise by the following theorem:

\begin{theorem}[Discrete to continuous approximation in $X^\sigma$] 
\label{disctocont}
Let $\sigma>2$ and suppose that $g :\R^2\to \C$ satisfies the following bounds:
\begin{equation}\label{assumptions on g}
\|g\|_{X^{\sigma+1}(\R^2)} +\|\nabla g\|_{X^{\sigma+1}(\R^2)}\leq B.
\end{equation}
The following estimate holds for $L>1$
\begin{equation}\label{DtoC}
\left\| \mathcal T(g,g, g)(K)- \mathcal{T}_L(g,g,g)(K) \right\|_{X_L^\sigma(\Z^2_L)}\lesssim \frac{B^3}{\log L}.
\end{equation}
\end{theorem}
This theorem is proved in Section \ref{T_L to T section}. As it stands, estimate \eqref{DtoC} is sharp in terms of the decay of $(\log L)^{-1}$, however it might be possible to isolate the terms that decay as such, move them to the left-hand-side and obtain better convergence rates. For the purposes of this paper, \eqref{DtoC} is more than enough.

\subsubsection{Approximation theorems}
We are finally ready to state our approximation theorems. We would like to start with a solution $g(t, \xi)$ of \eqref{star} defined on \emph{an arbitrarily long} time interval $[0,M]$ and construct a solution of \eqref{NLSe} whose dynamics in Fourier space is approximated by that of $g(t, \xi)$. The fact that $M$ is chosen arbitrarily allows to carry ``any" nonlinear dynamic of the limit equation \eqref{star} and project it into that of \eqref{FNLSe}. However, our heuristic calculation in Subsection \ref{formal argument} (precisely equation \eqref{pretend eqn4}) indicates that the two equations do not evolve with the same time scales. In fact, since the time-scale of \eqref{star} is normalized to 1, the ``clock" of \eqref{FNLSe} is much slower and its time-scale can be read from \eqref{pretend eqn4} to be $T_*=\frac{\zeta(2) L^2}{4\epsilon^2 \log L}$. As a result, we should actually expect to approximate the evolution of \eqref{FNLSe} with that of the rescaled function $g(\frac{t}{T^*}, \xi)$ for $0\leq t \leq MT^*$ where $g(t,\xi)$ solves \eqref{star} on $[0,M]$. This is the content of the following theorem.

\begin{theorem}\label{approx thm1}
Fix $\sigma>2$ and $0<\gamma<1$. Suppose that $g(t,\xi)$ is a solution of (CR) over a time interval $[0, M]$ with initial data $g_0=g(t=0)$ such that $g_0, \nabla g_0 \in X^{\sigma+1}(\R^2)$. By Theorem \ref{eider2} below, (CR) is globally well-posed for such initial data. Denote
$$
B\overset{def}{=}\sup_{t\in [0, M]}\|g(t)\|_{X^{\sigma+1}(\R^2)}+\|\nabla g(t)\|_{X^{\sigma+1}(\R^2)}.
$$
There exists $c=c(\gamma)$ such that if $L >1$ and $\epsilon< c L^{-1-\gamma} B^{-1}$, the solution $\widetilde a_K(t)$ of \eqref{FNLSe} with initial data $\widetilde a_K(0)\overset{def}{=}g_0(K)$ satisfies the following estimate:
\begin{equation}\label{tildea_K and g} 
\|\widetilde a_K(t)-g(\frac{t}{T^*},K)\|_{X_L^\sigma(\Z_L^2)}\lesssim \frac{B}{(\log L)^{1-\gamma}} \qquad \text{with } T^*\overset{def}{=}\frac{\zeta(2) L^2}{2\epsilon^2 \log L},
\end{equation}
for all $0\leq t \leq T^*\min\left(M, \frac{\gamma \log\log L}{c B^2}\right)$.

In particular, if $L$ is chosen so that $L\geq \exp \exp(c\gamma^{-1} MB^2)$, then \eqref{tildea_K and g} holds for all $t\in [0, MT^*]$ (the full domain of definition of $g(\frac{t}{T^*})$.
\end{theorem}

Note that by the approximation interval in \eqref{tildea_K and g} corresponds for the limiting equation \eqref{star} to the interval
$\left[ 0 , \min \left(M, \frac{\gamma \log\log L}{c B^2}\right) \right]$, which means that if $L$ is large enough, one can project the full dynamics of $g(t)$ over the interval $[0,M]$ into that of \eqref{NLSe}. 

\begin{remark} How optimal is the condition on $L$ and $\epsilon$ in the statement of the theorem? This question leads to distinguishing several natural regimes in the weakly nonlinear, large-box limit for {\rm (NLS)}:
\begin{itemize}
\item If $\epsilon^2 L \rightarrow \infty$, it is natural to expect that the nonlinear dynamics will be effectively  described (upon appropriate rescaling) by {\rm (NLS)} on $\mathbb{R}^2$. Indeed, for data of size one, localized around $B(0,1)$ in space and frequency, the nonlinear time-scale\footnote{Prior to this time-scale, the solution behaves linearly, and thus remains almost constant in the interaction-representation picture (i.e. constant linear profile).} is of order $\epsilon^{-2}$ while the wave travels at a speed $O(1)$. Thus, it takes a time $L$ before the wave can tell the difference between $\mathbb{R}^2$ and $\mathbb{T}_L^2$. Therefore, if $\frac{1}{\epsilon^2}\ll L$, the effective limiting equation is {\rm (NLS)} on $\mathbb{R}^2$.
\item The regime described in the above theorem is essentially $\epsilon L \rightarrow 0$. It seems optimal for the type of result we prove: indeed, the nonlinear time scale is $\frac{1}{\epsilon^2}$, and the necessary time to average over non-resonant interactions is $L^2$ (which also corresponds to the period of the linear Schr\"odinger equation). If $L^2 \ll \frac{1}{\epsilon^2}$, non-resonant interactions are averaged out before the nonlinear time-scale is reached, and only the resonant ones contribute to the dynamics. This numerology will become even more transparent once we perform the normal form analysis in the proof of Theorem \ref{approx thm1}. 
\item There remains the regime $\frac{1}{L} < \epsilon < \frac{1}{\sqrt{L}}$. What should be expected there is not very clear, but it seems likely that not only exactly resonant interactions should play a significant role.
\end{itemize}
\end{remark}

\medskip
It is easy by a scaling argument to translate the above theorem into a result for \eqref{cubicNLS} on $\mathbb{T}^2$ in the weakly non-linear, high frequency regime. Here, one can interpret \eqref{star} as an equation for the frequency envelopes of \eqref{cubicNLS}. Starting with solutions of \eqref{star}, we will be constructing solutions of \eqref{cubicNLS} having size $\sim 1$ in some Sobolev space $H^s(\T^2)$ with $s>1$, and describing the dynamics in that space as well. The parameter $\epsilon$ which is comparable to size of the initial data in $L^2(\T^2)$ will be used to normalize the data under consideration in $H^s(\T^2)$. 

\begin{theorem}\label{result on T^2}
Fix $s>1$ and $0<\gamma<1$. Suppose that $g(t,\xi)$ is a solution of 
(CR) over a time interval $[0, M]$ with initial data $g_0(\xi)=g(0, \xi)$ such that $g_0, \nabla g_0 \in X^{\sigma}(\R^2)$ with $\sigma>s+2$. Recall that \eqref{star} is globally well-posed for such initial data by Theorem \ref{eider2}. Denote by 
$$
B\overset{def}{=}\sup_{t\in [0, M]}\|g(t)\|_{X^{\sigma}(\R^2)}+\|\nabla g\|_{X^{\sigma}(\R^2)}.
$$
Let $N>1+CB^{\frac{2}{s-1}}$ and $v(t)$ be the solution to the cubic NLS equation 
\begin{equation}\label{nlsnormalized}
-i\partial_t v+\Delta v=\mu |v|^2v, \quad v(0,x) = v_0(x), \quad \mu=\pm 1, \quad x \in \T^2,
\end{equation}
where the Fourier transform of the initial data is given by 
\begin{equation}
\label{initnormalized}
\widehat v_0(k) = \frac{1}{N^{s+1}}g_0(\frac{k}{N})\quad \forall\, k \in \Z^2, 
\end{equation}
(consequently $\|v(0)\|_{H^s(\T^2)}\sim \|\langle \xi \rangle^s g_0\|_{L^2(\R^2)}\sim 1$ uniformly in $N$). Then the following estimate holds:
\begin{equation}\label{v_K and g} 
\|\widehat v(k,t)- e^{4\pi^2i |k|^2t}N^{-1-s}\mu g(\frac{t}{T_N}, \frac{k}{N})\|_{\ell^{2,s}(2\pi\Z^2)}\lesssim_{s} \frac{B}{(\log N)^{1-\gamma}}, \qquad \text{with }T_N=\frac{\zeta(2) N^{2s}}{2\log N}
\end{equation}
for all $0\leq t \leq T_N \min\left(M, \frac{\gamma \log\log N}{c B^2}\right)$. 

In particular, if $N\geq \exp \exp(c\gamma^{-1} MB^2)$, then \eqref{v_K and g} holds for all $t\in [0, MT_N]$ (the full domain of definition of $g(\frac{t}{T_N})$.
\end{theorem}

The condition on $s>1$ is the translation of the condition $\epsilon<L^{-1-\gamma}$ in Theorem \ref{approx thm1}. Notice that as we increase $N$, $v_0$ carries more and more refined information about $g_0$ and the approximation becomes better and better. The Sobolev norm of $v(t)$ is uniformly comparable to $\|\langle \xi \rangle^s g(\frac{t}{T_N},\xi )\|_{L^2(\R^2)}$ for all $0\leq t \leq T_NM$. As in Theorem \ref{approx thm1}, the approximation interval is long enough to carry all the nonlinear behavior of $g(t)$ on the interval $[0,M]$ of interest. This shows that we can find solutions to the (NLS) equation \eqref{nlsnormalized} in $H^s(\T^2)$ over very large times (both in the focusing and defocusing case) with rather interesting nonlinear behaviors inherited from \eqref{star}. Note, however, that these solutions are small in $H^r(\T^2)$ for $r < s$ (Their $L^2$ norm is $\sim N^{-s}\ll 1$ which guarantees global existence), and large in $H^r(\T^2)$ for $r > s$. The only ``non-trivial" dynamics described in the limit of large $N$ (i.e. uniformly in $N$) happens in $H^s(\T^2)$.

A striking application of this corollary happens when the data is taken to be any of the explicit solutions of \eqref{star} that will be discussed below. The simplest of which is when we take $g_0 = G(x)= \frac{1}{\sqrt \pi} e^{-\frac12|x|^2}$. We will see that the solution $g(t,x)$ is given explicitely by $g(t,x) = e^{i\omega_0t} G(x)$ for some \emph{non-zero} constant $\omega_0$. In this case, $B$ is uniformly bounded by a universal constant independent of $M$. The previous corollary can be interpreted as a stability result saying that the solution of the $\eqref{cubicNLS}$ equation \eqref{nlsnormalized} starting with the initial value \eqref{initnormalized} (which is close to a gaussian with small variance in the $x$-space) remains close in $H^s$ over long times to the function
\begin{equation}
\label{eq:solgaussian}
v_N(t,x) =  \sum_{k \in \Z^2} e^{2\pi ik \cdot x} e^{i(4\pi^2|k|^2t + C_N t)}\frac{1}{N^{s+1}}G( \frac{k}{N}), 
\end{equation}
where the phase shift $C_N=\omega_0T_N^{-1}$ is very small (of order $N^{-2s}$), but its effect is felt for times $t\geq T_N$. In other words, the function $v(t,x)$ behaves like a solution that is almost periodic in time, and is of size one in $H^s$. Note that in particular at times $t\in (2\pi)^{-1}\Z$ the solution \eqref{eq:solgaussian} focuses again to a Gaussian (up to unimodular factor), while for the other times, the solution is not concentrated anymore in the $x$-space. This is observed numerically as well. The nonlinear effect is only exhibited via the modulation factor $e^{i C_Nt}$ with $C_N=\omega_0 T_N^{-1}$. Since $\omega_0\neq 0$, this gives an explicit example of a dynamic of the cubic NLS equation on $\T^2$ that violates \eqref{Euclidean stability} at the tractable time scales.

Let us conclude by mentioning that this result can be compared with the long-time Sobolev stability of plane wave solutions\footnote{These are the solutions of the form $Ae^{i(2\pi n\cdot x+4\pi^2|n|^2t\pm |A|^2t)}$.} of (NLS) in \cite{FLG}, but note that the construction is radically different: Plane waves in the defocusing case are of order $1$ in $H^s$ for {\em all} $H^s$, $s \geq 0$,  and their stability comes from the creation of generically non-resonant frequencies around the orbit of the plane wave. The same game could be played with the other explicit solutions of $g$ in Section~\ref{sectionhamiltonian}.

\subsection{Analysis of  equation (CR)}
It turns out that equation \eqref{star} has a surprisingly rich structure in terms of its Hamiltonian nature, symmetries,  explicit solutions, and global well-posedness properties. We describe them in this section.

\subsubsection{Hamiltonian structure}
Equation {\rm (CR)} is Hamiltonian: it derives from the Hamiltonian functional
\begin{equation}\label{Hamil1}
\mathcal{H}(f) = \frac{1}{4} \int_{-1}^1 \int_{\mathbb{R}^2} \int_{\mathbb{R}^2} 
 f(x+z) f(\lambda x^\perp + z)  \overline{f(x+\lambda x^\perp +z)} \, \overline{f(z)} \,dx\,dz\,d\lambda
\end{equation}
with the symplectic form on $L^2(\R^2)$ given by
$$
\omega(f,g) = \frak{Im} \int_{\mathbb{R}^2} f(x) \overline{g (x)} \,dx.
$$
A small computation performed in Section \ref{sectionhamiltonian} reveals that the Hamiltonian functional can also be written
\begin{equation}\label{scarlatti}
\mathcal{H}(f) = \frac{\pi}{2} \int_{\mathbb{R}} \int_{\mathbb{R}^2} |e^{it\Delta_{\R^2}} \widehat{f}|^4 \,dx\,dt.
\end{equation}
Thus $\mathcal{H}(f)$ is simply a multiple of the $L^4_{t,x}$ Strichartz norm on $\R^2$! While $\mathcal H(f)$ can be easily seen, from \eqref{Hamil1} and Cauchy-Schwartz, to be bounded on $L^2(\R^2)$ in the sense that
\begin{equation}\label{HL2bdd}
\mathcal H(f)\lesssim \|f\|_{L^2(\R^2)}^4
\end{equation}
this estimate is nothing but the $L^4_{t,x}$-Strichartz estimate on $\R^2$ (\cite{Strichartz,Tomas}). The Hamiltonian structure of (CR) is examined in detail in Section~\ref{sectionhamiltonian}, let us mention here a few striking facts. In all what follows suppose that $g(t)$ is a solution of \eqref{star}.
\begin{itemize}
\item \underline{Conserved quantities.} The coercive conserved quantities are $\mathcal{H}(g)$, $\int|g|^2$, $\int |\nabla g|^2$ and $\int |xg|^2$.
\item \underline{Group Symmetries.} The equation (CR) enjoys a large group of symmetries, including in particular $g(x) \mapsto g(x+x_0)$, $g(x) \mapsto g(x)e^{ix\xi_0}$, 
$g \mapsto e^{i\lambda x^2} g$, $g \mapsto e^{i\lambda \Delta} g$ for any $x_0, \xi_0\in \R^2$ and $\lambda \in \R$.
\item \underline{Invariance under the Fourier transform.} The Fourier transform $\widehat g(t)=\mathcal F_{\R^2}g(t)$ is also a solution of $\eqref{star}$!
\item \underline{Invariance of the Eigenspaces of the harmonic oscillator.} Recall that the harmonic oscillator $-\Delta+|x|^2$ admits a discrete orthonormal basis of $L^2(\R^2)$ formed of eigenvectors. Each eigenspace $\{E_k\}_{k \in \N}$ is finite dimensional of dimension $k$. The flow of (CR) leaves each of the eigenspaces of the harmonic oscillator operator $-\Delta+|x|^2$ invariant, i.e. if $g_0 \in E_k$ for some $k$, then $g(t) \in E_k$ for all $t\in \R$.
\item \underline{Gaussians are stationary solutions.} By the previous point, the ground state of $-\Delta+|x|^2$ is invariant. It is given by $\mathbb{R} e^{-\frac{x^2}{2}}$. Thus, a stationary solution of
(CR) is given by $e^{-\frac{x^2}{2}+i\omega_0t}$ for a constant $\omega_0$.
\item \underline{Other stationary solutions.} We mention here $\frac{e^{iCt}}{|x|}$, for a suitable constant $C$. This solution corresponds exactly to the Raleigh-Jeans solution of the Kolmogorov-Zakharov equation (cf.~Section \ref{WT subsection}). Other explicit solutions are given in Section~\ref{sectionhamiltonian}. 
\end{itemize}

The rich structure of symmetries of \eqref{star} is somewhat reminiscent of the Szeg\"o equation studied by
G\'erard and Grellier~\cite{GG1,GG2,GG3} on $\T$ and Pocovnicu~\cite{P1,P2} on $\R$. It turns out that the Szeg\"o equation is completely integrable; however the question of whether one can show that \eqref{star} is completely integrable or not is very interesting and open with remarkable consequences on the long-time behavior of \eqref{cubicNLS}.

\subsubsection{Analytic properties}

The first step is to describe boundedness properties of $\mathcal{T}$. Recall that
\begin{equation}
\dot X^1(\R^2) \overset{def}{=} \{ f(x) \, \, | \,Ê\, |x| f(x) \in L^\infty(\R^2)\} 
\end{equation}
\begin{proposition}
\label{eider1}
The trilinear operator $\mathcal{T}$ is bounded from $(L^2)^3$ to $L^2$, and from $(\dot X^1)^3$ to $\dot X^1$.
\end{proposition}

This statement easily gives obvious local well-posedness results in $L^2$ and $\dot X^1$; combining it with the conservation laws of \eqref{star} as well as its invariance with respect to the Fourier transform, one obtains the following global well-posedness theorem.

\begin{theorem}
\label{eider2}
\begin{itemize}
\item The equation {\rm (CR)} is locally well-posed in $\dot X^1(\R^2)$, and globally well-posed in $L^2(\R^2)$.
\item \eqref{star} is globally well-posed in $H^s(\R^2)$ and $L^{2,s}(\R^2)$ for any $s\geq 0$.
\item \eqref{star} is globally well-posed in $X^\sigma(\R^2)$ for any $\sigma>2$. Moreover, if $\nabla g(0) \in X^\sigma$, then $g, \nabla g \in C_{\operatorname{loc}}(\R; X^\sigma)$.
\end{itemize}
\end{theorem}

Further well-posedness results are mentioned in Theorem \ref{th:7.4}. Next, we observe that the trilinear operator $\mathcal{T}$ has a remarkable weak boundedness property, reminiscent of the div-curl lemma~\cite{Murat,Tartar}: 
\begin{theorem}
\label{eider3}
Assume that $\{f^n\}_{n \in \N}$, $\{g^n\}_{n\in \N}$ and $\{h^n\}_{n\in \N}$ are sequences converging weakly in $L^2$ to $f$, $g$ and $h$ respectively. Then the sequence $\{\mathcal{T}(f^n,g^n,h^n)\}_{n \in \N}$ converges weakly in $L^2$ to $\mathcal{T}(f,g,h)$.
\end{theorem}

This implies a weak compactness result for solutions of (CR).

\begin{theorem}
\label{eider4}
Given a sequence $\{f^n\}_{n \in \N}$ of solutions of {\rm (CR)}, uniformly bounded in the space $L^\infty([0,T],L^2)$ for some $T>0$, there exists a solution $f$ of ${\rm (CR)}$ towards which
a subsequence of $\{f^n\}_{n \in \N}$ converges in the weak-* topology of $L^\infty([0,T],L^2)$ . 
\end{theorem}

Theorems~\ref{eider1}--\ref{eider4} are proved, along with other results, in section~\ref{sectionanalytic}.

\subsubsection{Variational properties} The last of the properties of \eqref{star} that we mention here relates to the variational characterization of the Gaussian family of explicit solutions in relation to the Hamiltonian function $\mathcal H(f)$.  The theorem below is not new given the realization that $\mathcal H(f)$ can be equivalently written as \eqref{scarlatti} and the study of extremizers of Strichartz estimates in \cite{Fos, HZ}. Still, we chose to give a full proof in order to keep the paper as self-contained as possible. It relies on the heat flow monotonicity following ideas in~\cite{CLL,BCCT, BBCH}. 

\begin{theorem}
\label{duck1}
For a fixed mass $\|f\|_{L^2}=M$, \eqref{HL2bdd} is only saturated by Gaussians with mass $M$, up to the finite group of symmetries of $\mathcal H$. Furthermore, the optimal implicit constant in the \eqref{HL2bdd} is $\frac{\pi}{8}$.
\end{theorem}

Using in addition the profile decomposition for the $L^4$ Strichartz norm~\cite{MV} gives compactness of maximizing sequences and the $L^2-$orbital stability of the Gaussian family.

This last result and Theorem~\ref{duck1} are proved in section~\ref{sectionvariational}.

\subsection{Physical-space interpretation}\label{Interpretation} Here we try to give a heuristic explanation why the limiting equation (CR) is natural; this could play a role in the dynamics beyond the regime $\epsilon< L^{-1}$ mentioned above.

\bigskip

Let us first approach the equation in Fourier space. The previous arguments can then be recast as follows. For $L$ finite, the (discrete) resonant system derives from the Hamiltonian $\mathcal{H}_L (f) =  \frac14 \langle \mathcal{T}_L(f,f,f)\,,\,f \rangle$. We saw that $\mathcal{T}_L \rightarrow \mathcal{T}$ as $L \rightarrow \infty$, therefore $\mathcal{H}_L$ converges to $\mathcal{H} = \frac14 \langle \mathcal{T}(f,f,f)\,,\,f \rangle$, and the equation deriving from $\mathcal{H}$ is (CR).

\bigskip

It is perhaps physically more illuminating to examine the convergence to (CR) in physical space. There are then two key elements to be taken into account. 
\begin{enumerate}
\item The flow of the linear Schr\"odinger equation is periodic on $\mathbb{T}^2_L$ with period $T = L^2$ (the situation is actually more complicated: the period is actually shorter depending on the considered frequency, and this is ultimately responsible for the $\zeta(2)$ factor appearing in our derivation. We will however gloss over this additional difficulty). 
\item Let $S$ be the scattering operator for \eqref{cubicNLS} on $\R^2$ (This is the operator that maps the profile at $t = - \infty$ to the profile at $t = + \infty$), and denote by $\widetilde S=\mathcal F S\mathcal F^{-1}$ its Fourier space incarnation. Under appropriate assumptions, and for a real constant $C$, $\widetilde S$ can be expanded around $0$ as
\begin{equation}
\label{bluejay}
\widetilde S f = f + C i \mathcal{T}(f,f,f) + \text{higher order terms in }f
\end{equation}
as can be seen by a small computation from Duhamel's formula.
\end{enumerate}
Let $v(t)$ be a solution of \eqref{cubicNLS} (with $\epsilon=1$)and let $f(t, K):=e^{-it|K|^2}\widehat v(t, K)$ be the Fourier transform of its linear profile. We can now explain why (CR) comes out as a natural limiting system governing the dynamics of $f$.  Suppose that at the initial time $t=0$, the data $f(0,K)$ is small and localized (so that all (NLS) solutions considered are global). Since $L$ is very large, the corresponding solution $v(t)$ of \eqref{cubicNLS} on $\T^2_L$ will initially evolve like a solution of (NLS) on $\mathbb{R}^2$: it will scatter, and be well-approximated by the linear flow; thus for $t$ sufficiently large the linear profile will keep a nearly constant value given in Fourier space by $f\left(\frac{T}{2}\right)$... until the period $T$ when the linear flow will focus $v$ again. The nonlinear interaction will then come into play, and since $L$ is very large it will essentially amount to applying the scattering operator $\widetilde S$. Thus, essentially, $f\left( \frac{3T}{2} \right) \simeq \widetilde S f\left( \frac{T}{2} \right)$. The same reasoning gives for any integer $n$ that $f\left( \frac{T}{2} + n T \right) \simeq \widetilde S f\left( \frac{T}{2} + (n-1) T \right)$. Using in addition the smallness of $f$ and~(\ref{bluejay}), this yields
$$
f\left( \frac{T}{2} + n T \right) - f\left( \frac{T}{2} + (n-1) T \right) \simeq C i \mathcal{T} \left( f ,f ,f \right)\left( \frac{T}{2} + (n-1) T \right),
$$
which, rescaling time, can be approximated by (CR).

\subsection{Relation to weak turbulence}\label{WT subsection}

Weak turbulence, also known as wave turbulence, seeks a statistical description\footnote{More precisely, this refers to kinetic wave turbulence, as opposed to more recent theories like \emph{discrete} or \emph{mesoscopic} wave turbulence.} of the out-of-equilibrium dynamics of equations ranging from \eqref{NLSe} to water waves or Vlasov-Maxwell. The fundamental equations of weak turbulence theory, sometimes called the Kolmogorov-Zakharov equations, appeared early on in the past century \cite{Peierls, Hassel1, Hassel2}; however the theory was most invigorated after the work of Zakharov in the 60s \cite{Zakh1, Zakh2}, who sought to parallel Kolmogorov's theory of hydrodynamic turbulence. 

The paradigm of kinetic weak turbulence can be briefly summarized as follows. We refer the reader to the classical works of Zakharov, L'vov and Falkovich~\cite{ZLF} and Nazarenko~\cite{Nazarenko}, as well as to Spohn~\cite{Spohn} for more precise information. Roughly speaking, one starts with initial data $a_K(0)=\sqrt{J_K(\omega)} \phi_K(\omega)$ where $J_K(\omega)\geq 0$ and $\phi_K(\omega)\in \mathbb{S}^1$ are random variables such that $\mathbb{E} \phi_k \overline{\phi_L} = \delta_{J=K}$: this is the {\it random phase approximation}. Then, one tries to write down an equation for the \emph{wave spectrum} $n_K(t)=\mathbb{E} |a_K|^2=\mathbb{E}J_K$. The equation of $n_K(t)$ is obtained in three steps: 1) Performing a sequence of statistical and time averages\footnote{This step is particularly far from being rigorous due to the assumption of ``preservation of chaos" i.e. preservation of the random phase approximation which is usually made in the physics literature.}, 2) taking the large-box limit $L\to \infty$ thus making $n_K(t)$ a field on $\R^2$, and 3) taking the weak non-linearity limit $\epsilon\to 0$ which has the effect of restricting to the resonant manifold in the obtained equation. The latter is called the Kolmogorov-Zakharov equation and, for \eqref{NLSe}, it often appears in the physics literature in the following form \cite{ZLF, Nazarenko}:

\begin{equation}\label{KZ}\tag{KZ}
\begin{split}
\partial_t n_K(t)= \epsilon^4 \iiint &n(K_1)n(K_2)n(K_3)n(K)\left(\frac{1}{n(K_1)}-\frac{1}{n(K_2)}+\frac{1}{n(K_3)}-\frac{1}{n(K)}\right)\\
&\delta(K_1-K_2+K_3-K)\,\delta(|K_1|^2-|K_2|^2+|K_3|^2-|K|^2)dK_1 \, dK_2\, dK_3.
\end{split}
\end{equation}

The derivation of this equation is very formal in the physics literature and a rigorous justification seems to be a difficult open question of great importance both in physics and mathematics (see however Lukkarinen and Spohn~\cite{LS} for first steps in this direction). The importance of the \eqref{KZ} equation is mostly due to the existence of four explicit power-type stationary solutions which are a manifestation of the equilibrium and out-of-equilibrium frequency dynamics of \eqref{NLSe}. Two of the those solutions can be easily seen by inspecting \eqref{KZ}: namely the \emph{thermodynamic equilibrium solution} $n_K(t)=1$ and the \emph{Raleigh-Jeans} solution $n_K(t)=|K|^{-2}$. The other two solutions are the interesting ones from the point of view of out-of-equilibrium dynamics and correspond to the constant flux of energy through scales (\emph{backward and forward cascade} spectra). Nonetheless, such constant-flux solutions can only be sustained by adding external sources and sinks of energy to the system (say at very high and low frequencies). 

There are strong similarities between weak turbulence theory and the ideas followed in the present paper: indeed, we follow the steps 2) and 3) in the above paradigm. Formal similarities are more obvious if we write \eqref{star} in the form
$$
-i \partial_t f_K (t) =  \iiint f(K_1)\overline{f(K_2)} f(K_3) \delta(K_1-K_2+K_3-K)\,\delta(|K_1|^2-|K_2|^2+|K_3|^2-|K|^2)dK_1 \, dK_2\, dK_3
$$
(in the above, we think of $K,K_1,K_2,K_3$ as continuous variables). Also notice that the exact solution $e^{iCt}|K|^{-1}$ of \eqref{star} plays the role of the Raleigh-Jeans solution $n(K)=|K|^{-2}$ of \eqref{KZ}.

The central difference is that, while we do perform the large-box and small non-linearity limit, no randomness, or averaging over statistical ensembles, is used in the present paper. As a result, the equation (CR) is Hamiltonian, time reversible, and exhibits no trend to equilibrium. This should be contrasted with (KZ), which is not Hamiltonian, and for which a version of the H-theorem applies. Another important difference is in terms of time-scales. It is obvious that the time scale of \eqref{KZ} is $O(\epsilon^{-4})$ (called the \emph{kinetic time-scale}, as opposed to the \emph{dynamic time-scale} $\epsilon^{-2}$); whereas the time scale in our case (see \eqref{pretend eqn4}) was roughly $T^*=\frac{L^2}{\epsilon^2 \log L}$ which is $\ll \epsilon^{-4}$ in the regime we are in. Clearly, an equation like \eqref{NLSe} undergoes a lot of nonlinear dynamics prior to the kinetic time scale $\epsilon^{-4}$. \eqref{KZ} seems too course to sense those intermediate time-scales, which \eqref{star} seems to capture.

Finally, we remark that it has been recognized in the wave turbulence community that some aspects of the turbulent dynamics are not captured by \eqref{KZ} especially starting with the highly-influential works of Majda, Mclaughlin, and Tabak \cite{MMT, CMMT}. Several theories were formulated to address these and other concerns \cite{Zakh3, Kartashova, KNR, Nazarenko2} and new regimes of wave turbulence emerged like \emph{mesoscopic} and \emph{discrete} wave turbulence \cite{Kartashova, KNR, Nazarenko2, ZKP, LN} in order to account for some finite-size effects that are not captured by the kinetic description mentioned above. We refer the reader to Nazarenko's modern textbook treatment in \cite{Nazarenko} for more details and references. A particularly relevant theory for our work stands out, namely that of \emph{discrete wave turbulence} (see for example \cite{Kartashova0, CNP, TY, Kartashova3, DLN, KNR, Kartashova, Kartashova2}) in which one retains the Hamiltonian nature of the equations and emphasizes the importance of the \emph{``discrete"} exactly resonant interactions on the dynamics, the same perspective and regime we adopt here. It is conceivable that the (CR) equation provides a good model for this framework.


\section{Proof of Theorem~\ref{tri est theo}: uniform boundedness of the $\mathcal{T}_L$}\label{tri est section}

The purpose of this section is to prove Theorem~\ref{tri est theo} which is of interest in its own right. It gives a uniform bound on the operators $\mathcal{T}_L$ on the spaces $X^\sigma(\mathbb{Z}_L^2)$. 

\subsection{Sharp Strichartz and multilinear estimates}


Estimates on resonant sums such as \eqref{tri est} are directly related to the endpoint\footnote{We call this estimate endpoint because $p=4$ is the smallest $L^p_{t,x}$ exponent in \eqref{Strichartz} where one might hope for an upper bound on $C(N)$ in \eqref{Strichartz} independent of $N$. This is indeed the case on $\R^2$ where $p=4$ is also the minimum exponent where any estimate of the form \eqref{Strichartz} can hold due to Knapp's counter-example.}
 Strichartz estimates on $\T^2$ given by
\begin{equation}\label{Strichartz}
\|e^{it\Delta_{\T^2}} P_N \phi\|_{L^4_{t,x}([0,1]\times \T^2)} \leq C(N) \|\phi\|_{L^2(\T^2)},
\end{equation}
where $P_N$ denotes frequency projection onto frequencies $\lesssim N$. To see the relation, one simply notices that:
\begin{align*}
& \|e^{it\Delta_{\T^2}}P_N\phi\|_{L^4_{t,x}([0,(2\pi)^{-1}]\times \T^2)}^4=\int_0^{(2\pi)^{-1}} \int_{\T^2}\left|\sum_{n \in \Z^2} \widehat \phi(n)e^{i(2\pi n \cdot x+4\pi^2|n|^2t)}\right|^4 dx dt\\
&\quad = \sum_{(k_1, k_2, k_3,k_4)}\int_0^{(2\pi)^{-1}}\int_{\T^2}e^{i 2\pi (k_1-k_2+k_3-k_4) \cdot x+4\pi^2i(|k_1|^2-|k_2|^2+|k_3|^2-|k|^2)t}\widehat \phi(k_1)\overline{\widehat \phi(k_2)} \widehat \phi(k_3)\overline{\widehat \phi(k)}\,dx\,dt \\
& \quad = \sum_{k\in \Z^2}\sum_{\mathcal R(k)} \widehat \phi(k_1)\overline{\widehat \phi(k_2)} \widehat \phi(k_3)\overline{\widehat \phi(k)}=\left\langle \sum_{\mathcal R(k)} \widehat \phi(k_1)\overline{\widehat \phi(k_2)} \widehat \phi(k_3)\, ,\widehat \phi(k)\right\rangle_{\ell^2\times \ell^2}.
\end{align*}

Bourgain initiated the systematic study of such estimates and proved in \cite{Bourgain} that the optimal $C(N)$ is $\lesssim_\epsilon N^\epsilon$ for any positive $\epsilon$. More precisely, using the divisor bound on integral domains and a counting argument, he shows that
\begin{equation}\label{Bourgain's bound}
C(N)\lesssim \exp(\frac{\log N}{\log \log N})\lesssim_\epsilon N^\epsilon.
\end{equation}
A lower bound on $C(N)$ is implicit in Bourgain's work \cite{Bourgain, Bourgain2} (but can be found explicitly in \cite{Kishi}), and is given by 
$$
C(N)\gtrsim (\log N)^{1/4}.
$$
This is in sharp contrast with the analogous situation when $\T^2$ is replaced by $\R^2$ in which case $C(N)$ is bounded uniformly in $N$.
The analogue of Theorem \ref{tri est theo} with $X_L^\sigma$ replaced by $\ell^{2, \sigma}_L$ is effectively equivalent (see proof of Corollary \ref{new strichartz corollary} below) to proving that $(\log N)^{1/4}$ is also an upper bound (up to a constant factor) for the optimal $C(N)$. Indeed, based on Theorem \ref{tri est theo} and Corollary \ref{new strichartz corollary} below, we conjecture the following 
\begin{conjecture}
For all $N$ sufficiently large and $\phi \in L^2(\T^2)$, it holds that:
\begin{equation}\label{conjectured strichartz}
\|e^{it\Delta_{\T^2}}P_N\phi\|_{L_{t,x}^4([0,1]\times \T^2)}\lesssim (\log N)^{1/4}\|\widehat \phi\,\|_{\ell^2(\Z^2)}.
\end{equation}
\end{conjecture}

While this particular estimate remains open and is widely expected to be quite difficult, it is equivalent to \eqref{tri est} with $X^\sigma$ replaced by the $L^2-$based weighted spaces $\ell^{2,s}_L$ (the factor of $L^2$ in the definition \eqref{TL} of the operator $\mathcal{T}_L$ is acquired once one rescales the estimate to the torus $\T^2_L$ and the factor of $\log L$ is the effect of the constant $C(N)$ raised to the forth power). In short, the normalization factor in the definition \eqref{TL} of the operator $\mathcal{T}_L$ in the estimate \eqref{tri est} corresponds exactly to the sharp constant of this Strichartz estimate. 

Nonetheless, Theorem \ref{tri est theo} implies the following alternative to \eqref{conjectured strichartz}:
\begin{corollary}\label{new strichartz corollary}
For $N$ sufficiently large and $\phi$ a function on $\T^2$, the following estimate holds:
\begin{equation}\label{new strichartz}
\|e^{it\Delta_{\T^2}}P_N\phi\|_{L_{t,x}^4([0,1]\times \T^2)}\lesssim (\log N)^{1/4}N^{1/2}\|\widehat \phi\,\|_{\ell^\infty}^{3/4}\|\widehat \phi\,\|_{\ell^1}^{1/4}
\end{equation}
Moreover, this estimate is sharp in its dependence on $N$.
\end{corollary}

The interest of \eqref{new strichartz} as opposed to other variants of \eqref{conjectured strichartz} (like the one obtained from Bourgain's bound \eqref{Bourgain's bound} for $C(N)$) is that it is scales exactly like the conjectured optimal estimate \eqref{conjectured strichartz}\footnote{In the sense that it translates into an estimate on $\T^2_L$ of the same form with $\ell^p$ replaced by $\ell^p_L$ and $N$ replaced by $LN$.}. A quick way to see this is to notice that \eqref{new strichartz} would follow from \eqref{conjectured strichartz} using the following scale-invariant estimates $\|\widehat{P_{N}\phi}\|_{\ell^2}\lesssim N^{1/2}\|\widehat \phi\|_{\ell^{4}}\lesssim N^{1/2}\|\widehat \phi\|_{\ell^1}^{1/4}\|\widehat \phi\|_{\ell^\infty}^{3/4}$. This scale-invariance turns out to be crucial for our purposes. In fact, the analysis in Section \ref{proof of approx theorem section} vitally requires us to know the sharp dependence on $L$ in Theorem \ref{tri est theo} to extent that any less refined analysis that slightly misses the right dependence (namely $L^2 \log L$ appearing in our normalization of $\mathcal T_L$ in \eqref{TL}), such as that given in Proposition \ref{R_mu sum prop} (which yields the dependence $L^{2+\epsilon}$ for any $\epsilon$), is completely useless.

\begin{proof}[Proof of Corollary \ref{new strichartz corollary}] Without loss of generality, we assume that $\widehat \phi$ is supported in the ball $B(0, 2N)$. As we noticed before, 
\begin{align*}
\|e^{it\Delta_{\T^2}}P_N\phi\|_{L^4_{t,x}([0,(2\pi)^{-1}]\times \T^2)}^4=&\int_0^{(2\pi)^{-1}} \int_{\T^2}\left|\sum_{n \in \Z^2} \widehat \phi(n)e^{i2\pi (n\cdot x+2\pi|n|^2t)}\right|^4 dx \, dt\\
=&\sum_{k\in \Z^2}\sum_{\mathcal R(k)} \widehat \phi(k_1)\overline{\widehat \phi(k_2)} \widehat \phi(k_3)\overline{\widehat \phi(k)}
=\sum_{K \in \Z^2_N}\sum_{\mathcal R(K)} a_{K_1} \overline{a_{K_2}} a_{K_3}\overline{a_{K}}
\end{align*}
where $a_K=\widehat \phi(NK)$ for any $K\in \Z^2_N$. Corollary \ref{new strichartz corollary} now follows by applying Theorem \ref{tri est theo} to the sequence $\{a_K\}$ and noting that since it is supported in $B(0,2)$ we have that $\|a_K\|_{X_N^\sigma}\sim \|\widehat \phi\|_{\ell^\infty}$ and $\sum_{K}\langle K \rangle^{-\sigma} a_K\sim \sum_K a_K=\|\widehat \phi\|_{\ell^1}$. The claim of sharpness follows either from the same example as that in \cite{Kishi} or by simply noticing it as a consequence to Theorem \ref{disctocont} proved in the next section. 
\end{proof}

We now proceed with the proof of Theorem \ref{tri est theo}. The sharpness of the estimate follows by the same argument just given above. 

\subsection{Proof of Theorem \ref{tri est theo}}
Without loss of generality, we assume that $\|a_K\|_{X^\sigma_L}=\|b_K\|_{X^\sigma_L}=\|c_K\|_{X^\sigma_L}=1$. We first parameterize $\mathcal R(K)$ as in \eqref{eq:rect1} and write
\begin{align*}
\sum_{\RR(K)}a_{K_1} b_{K_2} & c_{K_3}=\sum_{\substack{N_1, N_3 \in \Z^2_L\\ N_1\cdot N_3=0}} a_{K+N_1} b_{K+N_1+N_3} c_{K+N_3}\\
=& a_{K}\sum_{\substack{N \in \Z^2_L}}  b_{N} c_{N}+c_{K}\sum_{\substack{N \in \Z^2_L}}  a_N b_{N} -a_Kb_K c_K+\sum_{\substack{N_1, N_3 \in (\Z^2_L)^*\\ N_1\cdot N_3=0}} a_{K+N_1} b_{K+N_1+N_3} c_{K+N_3}.\\
\end{align*}
The estimate on the first three terms above is a triviality, so we only focus on proving that
\begin{align*}
\left\|\sum_{\substack{N_1, N_3 \in (\Z^2_L)^*\\ N_1\cdot N_3=0}} a_{K+N_1} b_{K+N_1+N_3} c_{K+N_3}\right\|_{X^\sigma_L}\lesssim L^2 \log L.
\end{align*}
Now using the parametrization \eqref{pretend eqn3} of the set $\{N_1, N_3 \in {\Z^2_L}^*:N_1 \cdot N_3=0\}$, we write the above sum as
$$
\sum_{\substack{N_1, N_3 \in (\Z^2_L)^*\\ N_1. N_3=0}} a_{K+N_1} b_{K+N_1+N_3} c_{K+N_3}=\sum_{\substack{\alpha\geq 1\\ \beta \in \Z^*}}\sum_{J\in \Z^2_L \, visible} a_{K+\alpha J} b_{K+\alpha J+\beta J^\perp} c_{K+\beta J^\perp}.
$$

Next we split the sum in $\alpha$ and $\beta$ into four pieces $\{1\leq \alpha \leq \beta\}, \{1\leq \alpha \leq -\beta\}, \{1\leq \beta\leq \alpha\}$ and $\{1\leq -\beta \leq \alpha\}$. Using the symmetries in the above sum (like $J\leftrightarrow -J$, $J\leftrightarrow J^\perp$), it is sufficient to prove the estimate for the first piece, namely:
\begin{equation*}
\left\|\sum_{\substack{1\leq \alpha\leq \beta}}\sum_{J\in \Z^2_L \, visible} a_{K+\alpha J} b_{K+\alpha J+\beta J^\perp} c_{K+\beta J^\perp}\right\|_{X^\sigma_L}\lesssim L^2 \log L.
\end{equation*}
This is equivalent to showing that for any $\sigma>2$ and any $K \in \Z^2_L$
\begin{equation}\label{aim tri}
\langle K \rangle^{\sigma}\sum_{\substack{1\leq \alpha\leq \beta}}\sum_{J\in \Z^2_L \, visible} \langle K+\alpha J\rangle^{-\sigma} \langle K+\alpha J+\beta J^\perp\rangle^{-\sigma} \langle K+\beta J^\perp\rangle^{-\sigma}\lesssim L^2 \log L.
\end{equation}

Recall the notation $\BB(m)=[-\frac{m}{2},\frac{m}{2})^2$. If $\gamma\in \N$, then any $N \in \Z^2_L$ can be written as $r_\gamma+\gamma Q_\gamma$ for some $r_\gamma\in \BB(\gamma/L) \cap \mathbb{Z}_L^2$ and $Q_\gamma \in \Z^2_L$.  This allows to write for any $\alpha, \beta$ of the sum \eqref{aim tri}
\begin{equation}
\label{eq:rQ}
K=r_\alpha+\alpha Q_\alpha ;\qquad K=r_\beta+\beta Q_{\beta}^\perp. 
\end{equation}
Note that we have $|r_\alpha| \leq \frac{\alpha}{\sqrt{2}L}$ and $|r_\beta| \leq \frac{\beta}{\sqrt{2}L}$. The proof of \eqref{aim tri} turns out to be particularly sensitive to two values of $J$ that we treat independently, namely $J=-Q_\alpha$ and $J=-Q_\beta$. So we denote by $S\overset{def}{=}\{ -Q_\alpha, -Q_\beta\}$, and we split the sum in $J$ into two pieces 
$$
\sum_{J\in \Z^2\, visible}=\sum_{J \notin S\; visible}+\sum_{J \in S\; visible} .
$$
and treat each case independently. We start with the first case.
\medskip 

\noindent
\textbf{Case 1: Contribution of the $J \notin S$} Here we prove that
\begin{equation}\label{aim tri1}
\langle K \rangle^{\sigma}\sum_{\substack{1\leq \alpha\leq \beta}}\sum_{J\in \Z^2_L\setminus S \, visible} \langle K+\alpha J\rangle^{-\sigma} \langle K+\alpha J+\beta J^\perp\rangle^{-\sigma} \langle K+\beta J^\perp\rangle^{-\sigma}\lesssim L^2 \log L.
\end{equation}

We further split the analysis into two cases: $|K|\leq 100$ and $|K|\geq 100$.

\medskip
\noindent
\underline {Subcase 1(a): $|K|\leq 100$.} Here we only need to show that
\begin{equation}\label{aim of subcase1}
\sum_{\substack{1\leq \alpha\leq \beta}}\sum_{\substack{J\in \Z^2_L\setminus S \\ visible}} \langle K+\alpha J\rangle^{-\sigma} \langle K+\alpha J+\beta J^\perp\rangle^{-\sigma} \langle K+\beta J^\perp\rangle^{-\sigma}\lesssim L^2 \log L. 
\end{equation}
This follows from the estimate
$$
\sum_{\substack{1\leq \alpha\leq \beta}}\sum_{\substack{J\in \Z^2_L\setminus S \\ visible}} \langle K+\alpha J+\beta J^\perp\rangle^{-\sigma} \langle K+\beta J^\perp\rangle^{-\sigma}\lesssim L^2 \log L,
$$
which is proved in Lemma \ref{2birds}. To prove Lemma \ref{2birds}, we need the following elementary inequality, which we isolate in a lemma since we will use it often:
\begin{lemma}\label{elementary}
Let $0\leq \alpha \leq \beta$ be non-negative integers with $\beta$ positive, and $\sigma>2$. Let $r\in \R^2$ be such that $|r|\leq \frac{3}{4}\frac{\sqrt{\alpha^2+\beta^2}}{L}$. Then the following estimate holds: 
\begin{equation}\label{elem1}
\sum_{J \in (\Z^2_L)^*} \langle r+\alpha J +\beta J^\perp\rangle^{-\sigma} \lesssim \min(\frac{L^2}{\beta^2}, \frac{L^\sigma}{\beta^\sigma}) = \begin{cases}\frac{L^2}{\beta^2} \quad \text{if } \beta \leq L,\\[1ex]
											\frac{L^{\sigma}}{\beta^{\sigma}} \quad \text{if } \beta \geq L.
				   \end{cases}
\end{equation}
\end{lemma}

\begin{proof}
In the sum \eqref{elem1}, $J$ runs in $(\mathbb{Z}_L^2)^*$ and can be written $j/L$ for some $j \in \Z^2$. 
Now we have $|r+\alpha J +\beta J^\perp| \geq \sqrt{\alpha^2 + \beta^2} |J| - |r| \geq \frac{1}{4L}\sqrt{\alpha^2 + \beta^2} |j| \geq \frac{\beta}{4L} |j| $. 
Hence the previous sum can always be estimated  by 
$$
 \sum_{j \in (\Z^2)^*} \left\langle \frac{\beta|j|}{4L} \right\rangle^{-\sigma}\lesssim
 \begin{cases}
   \frac{L^\sigma}{\beta^\sigma} \big(\sum_{j \in (\Z^2)^*} \frac{1}{|j|^\sigma}\big)\lesssim \frac{L^\sigma}{\beta^\sigma} \text{ if $\beta \geq L$},\\
  \frac{L^2}{\beta^2} +\frac{L^\sigma}{\beta^\sigma} \big(\sum_{\substack{j \in (\Z^2)^*\\|j|\geq L\beta^{-1}}} \frac{1}{|j|^\sigma}\big)\lesssim \frac{L^2}{\beta^2} \text{ if $\beta \leq L$},
      \end{cases} 
$$
as $\sigma > 2$. 
\end{proof}

\begin{remark}
\label{beta<L1}
Note that when $\beta \leq L$, the previous estimate by $L^2/\beta^2$ is still valid when the sum \eqref{elem1} is taken over all the point of lattice, including the origin $J = (0,0)$. 
\end{remark}
Now \eqref{aim of subcase1} follows from the following lemma.
\begin{lemma}\label{2birds}
For $L\geq 5$ and any $K \in \Z_L^2$, there holds 
\begin{equation}
\label{firstlog}
\sum_{\substack{1\leq \alpha\leq \beta}}\sum_{\substack{J\in \Z^2_L\setminus S \\ visible}} \langle K+\alpha J+\beta J^\perp\rangle^{-\sigma} \langle K+\beta J^\perp\rangle^{-\sigma}\lesssim L^2 \log L.
\end{equation}
\end{lemma}

\begin{proof}[Proof of Lemma \ref{2birds}] 
Since $J \neq -Q_\beta$, then by Lemma \ref{elementary} with $\alpha=0$, we have that
$$
\sum_{J\in \Z^2_L\setminus S \, visible} \langle K+\alpha J+\beta J^\perp\rangle^{-\sigma} \langle K+\beta J^\perp\rangle^{-\sigma}\lesssim \min(\frac{L^2}{\beta^2}, \frac{L^\sigma}{\beta^\sigma})
$$
from which we deduce that the sum in \eqref{firstlog} can be estimated up to a constant by 
$$
\sum_{\substack{ 1 \leq \alpha \leq \beta \\Ê\beta \leq L}} \frac{L^2}{\beta^2}
+
\sum_{\substack {1 \leq \alpha \leq \beta \\Ê\beta \geq L}} \frac{L^\sigma}{\beta^\sigma} 
\leq L^2 \sum_{1 \leq \beta  \leq L}\frac{1}{\beta}
+
L\sum_{ Ê\beta \geq L} \Big(\frac{L}{\beta} \Big)^{\sigma-1}\lesssim L^2 \log L+L^2,
$$
from which we easily deduce the result. 
\begin{remark}\label{beta>L1}
Notice from the previous equation that if we restrict the sum in \eqref{aim of subcase1} to $\beta\geq L$, then we can estimate it by $L^2$ rather than $L^2\log L$.
\end{remark}
\black

\noindent
\medskip
\underline{Subcase 1(b): $|K|\geq 100$.} We start by noting that we can restrict the sum in $J$ to the case $|K +\alpha J|\leq \frac{|K|}{100}$. In fact, if $|K +\alpha J|\geq \frac{|K|}{100}$, then estimate \eqref{aim tri} follows directly from Lemma \ref{2birds} above. 

The restriction $|K+\alpha J|\leq \frac{|K|}{100}$ has several implications that we enumerate below:
\begin{enumerate}
\item First we notice that $|\alpha J|\geq \frac{99}{100}|K|$ and consequently $|\beta J^\perp|\geq \frac{99}{100}|K|$ since $\beta\geq \alpha$.
\item From the orthogonality relation 
$$
|K|^2+|K+\alpha J +\beta J^\perp|^2=|K+\alpha J|^2+|K+\beta J^\perp|^2
$$
we conclude that $|K+\beta J^\perp|\geq \frac{99}{100}|K|$. Consequently, we also have that
$$
|K+\alpha J +\beta J^\perp|\geq |\beta J^\perp|-|K+\alpha J|\geq \frac{98}{100}|K|.
$$
\item Since $|K+\alpha J|\leq \frac{1}{100}|K|$, then $\alpha J \in B(-K, \frac{|K|}{100})$ and hence $\beta J^\perp \in B(-\frac{\beta}{\alpha}K^\perp, \frac{\beta}{100\alpha}|K|)$. So $K+\beta J^\perp\in B(K-\frac{\beta}{\alpha}K^\perp, \frac{\beta}{100\alpha}|K|)$. But $|K-\frac{\beta}{\alpha}K^\perp|\geq \frac{\beta}{\alpha}|K|$, so we finally deduce that in fact we have:
$$
|K+\beta J^\perp|\geq \frac{99}{100}\frac{\beta}{\alpha}|K|.
$$
\end{enumerate}
As a consequence of this analysis, \eqref{aim tri} follows once we show that 
\begin{equation}\label{aim2 tri}
\sum_{1 \leq \alpha \leq \beta} \sum_{\substack{J \in \Z^2_L \setminus S\\ J \, visible}} \langle K+\alpha J\rangle^{-\sigma}\langle K+\beta J^\perp\rangle^{-\sigma}\lesssim L^2 \log L.
\end{equation}
For this we use the third point above to estimate as follows:
\begin{multline}
\label{wwl1}
\sum_{1 \leq \alpha \leq \beta} \sum_{\substack{J \in \Z^2_L \setminus S\\ J \, visible}} \langle K+\alpha J\rangle^{-\sigma}\langle   K+\beta J^\perp\rangle^{-\sigma}\\
\lesssim
 \sum_{1 \leq \alpha \leq \beta} \left(\sum_{\substack{J \in \Z^2_L \setminus S\\ J \, visible}} \langle K+\alpha J\rangle^{-2\sigma}\right)^{1/2} \left(\sum_{\substack{J \in \Z^2_L \setminus S\\ |K+\beta J^\perp|\geq \frac{99}{100}\frac{\beta}{\alpha}|K|}} \langle K+\beta J^\perp\rangle^{-2\sigma}\right)^{1/2}\\
\lesssim  \sum_{1 \leq \alpha \leq \beta} \frac{L}{\alpha}\left(\sum_{\substack{J \in \Z^2_L \setminus S\\ |K+\beta J^\perp|\geq \frac{99}{100}\frac{\beta}{\alpha}|K|}} \langle K+\beta J^\perp\rangle^{-2\sigma}\right)^{1/2}, 
\end{multline}
by Lemma \ref{elementary} and the fact that $J \neq -Q_\alpha$. The estimate on the term in parentheses above should be handled with some extra care. Recalling that $K=r_\beta+\beta Q_\beta^{\perp}$ with $|r_\beta|\leq \frac{\beta}{\sqrt 2 L}$ and $J \neq -Q_\beta$, we have
\begin{align*}
\sum_{\substack{J \in \Z^2_L \setminus S\\ |K+\beta J^\perp|\geq \frac{99}{100}\frac{\beta}{\alpha}|K|}} \langle K+\beta J^\perp\rangle^{-2\sigma}\leq \sum_{\substack{J \in (\Z^2_L)^* \\ |r_\beta+\beta J^\perp|\geq \frac{99}{100}\frac{\beta}{\alpha}|K|}} \langle r_\beta+\beta J^\perp\rangle^{-2\sigma}\lesssim \sum_{\substack{J \in (\Z^2_L)^* \\ |\beta J^\perp|\geq \frac{99}{200}\frac{\beta}{\alpha}|K|}} |\beta J|^{-2\sigma}
\end{align*}
since $|\beta J|=|\beta J^\perp|\geq \frac{1}{2} |r+\beta J^\perp|$. As a consequence, we get that:
\begin{align*}
\sum_{\substack{J \in \Z^2_L \setminus S\\ |K+\beta J^\perp|\geq \frac{99}{100}\frac{\beta}{\alpha}|K|}} \langle K+\beta J^\perp\rangle^{-2\sigma}&\lesssim \beta^{-2\sigma}\sum_{\substack{J \in \Z^2_L \setminus \{0\}\\ |J^\perp|\geq \frac{99}{200\alpha}|K|}} |J|^{-2\sigma} \\
&\lesssim \begin{cases} L^{2\sigma}\beta^{-2\sigma} \quad &\text{if }\alpha\geq 100L|K|,\\
L^2 \beta^{-2\sigma}(\frac{|K|}{\alpha})^{-2\sigma+2} \quad &\text{if }\alpha\leq 100L|K|.						\end{cases}
\end{align*}
Going back to \eqref{wwl1}, we get that
\begin{align*}
\sum_{1 \leq \alpha \leq \beta} \sum_{\substack{J \in \Z^2_L \setminus S\\ J \, visible}} \langle K+\alpha J\rangle^{-\sigma}\langle K+\beta J^\perp\rangle^{-\sigma}\lesssim \sum_{\substack{1 \leq \alpha \leq \beta\\ \alpha \leq 100L|K|}} \frac{L^2\alpha^{\sigma-2}}{\beta^{\sigma}|K|^{\sigma-1}} +\sum_{\substack{100L|K| \leq \alpha \leq \beta\\ }} \frac{L^{\sigma+1}}{\alpha\beta^\sigma} \\
\lesssim \sum_{\substack{1 \leq \alpha \leq 100L|K|}} \frac{L^2}{\alpha |K|^{\sigma-1}}+ \sum_{\substack{100L|K| \leq \alpha\\ }} \frac{L^{\sigma+1}}{\alpha^\sigma}  \lesssim \frac{L^2\log (L|K|)}{|K|^{\sigma-1}}\lesssim L^2 \log L.
\end{align*}
which finishes the proof of this case.

\begin{remark}\label{beta>L2}
Repeating this last estimate with the additional restriction that $\beta\geq L$ we get the improved bound
\begin{align*}
\sum_{\substack{1 \leq \alpha \leq \beta\\ \beta\geq L}} \sum_{\substack{J \in \Z^2_L \setminus S\\ J visible}} \langle K+\alpha J\rangle^{-\sigma}\langle K+\beta J^\perp\rangle^{-\sigma}\lesssim \sum_{\substack{1 \leq \alpha \leq \beta\\ \alpha \leq 100L|K|\;;\; \beta\geq L}} \frac{L^2\alpha^{\sigma-2}}{\beta^{\sigma}|K|^{\sigma-1}} +\sum_{\substack{100L|K| \leq \alpha \leq \beta\\ }} \frac{L^{\sigma+1}}{\alpha\beta^\sigma} \\
\lesssim \sum_{\substack{1 \leq \alpha \leq 100L|K|}} \frac{L^2\alpha^{\sigma-2}}{\max(\alpha, L)^{\sigma-1} |K|^{\sigma-1}}+ \sum_{\substack{100L|K| \leq \alpha\\ }} \frac{L^{\sigma+1}}{\alpha^\sigma} \lesssim L^2 \frac{\log |K|}{|K|^{\sigma-1}}\lesssim L^2.
\end{align*}

This, along with Remark \ref{beta>L1}, shows that if we restrict the $\beta$ sum in \eqref{aim tri1} to $\beta\geq L$, we also have the better estimate
\begin{equation}\label{beta>Lest1}
\langle K \rangle^{\sigma}\sum_{\substack{1\leq \alpha\leq \beta\\\beta\geq L}}\sum_{\substack{J\in \Z^2_L\setminus S \\ visible}} \langle K+\alpha J\rangle^{-\sigma} \langle K+\alpha J+\beta J^\perp\rangle^{-\sigma} \langle K+\beta J^\perp\rangle^{-\sigma}\lesssim L^2.
\end{equation}
This will be useful for us later.
\end{remark}

\noindent
\textbf{Case 2: Contribution of $J \in S$}.
We now deal with the case when $J$ is equal to one of the forbidden values in $S\overset{def}{=}\{-Q_\alpha, -Q_\beta\}$. Recall that since $J$ is visible, we only consider the cases when $Q_\alpha, Q_\beta \neq 0$. We start with the simpler case $J=-Q_\alpha \neq 0$. 

\medskip
\noindent
\underline{Subcase 2(a): $J=-Q_\alpha$.} We will show that for any $K \in \Z^2_L$, it holds that
\begin{equation}\label{J=-Q_alpha est}
\langle K \rangle^\sigma \sum_{1\leq \alpha \leq \beta} \langle K -\alpha Q_\alpha\rangle^{-\sigma}\langle K -\alpha Q_\alpha -\beta Q_\alpha^\perp\rangle^{-\sigma}\langle K-\beta Q_\alpha^{\perp}\rangle^{-\sigma}\lesssim L^2.
\end{equation}

We start with some elementary observations. \begin{enumerate}
\item  Recall that $K=r_\alpha+\alpha Q_\alpha$. Since $Q_\alpha\neq 0$, the $\alpha$ has to be less than $2L|K|$, since otherwise $r_\alpha=K$ and $Q_\alpha=0$. 
\item We have that $K+\alpha J=r_\alpha, K+\beta J^{\perp}=r_\alpha+\alpha Q_\alpha-\beta Q_\alpha^\perp,$ and $K+\alpha J+\beta J^\perp=r_\alpha-\beta Q_\alpha^\perp$. 
\item Notice that since $\alpha\leq \beta$, then $|K+\beta J^\perp|\geq \beta |Q_\alpha|-\frac{\alpha}{\sqrt 2 L}\gtrsim \beta |Q_\alpha|$. Similarly, $|K+\alpha J +\beta J^\perp|\gtrsim \beta |Q_\alpha|$. 
\item Finally, notice that since $Q_\alpha\neq 0$ and $K=r_\alpha+\alpha Q_\alpha$ with $|r_\alpha|\leq \frac{\alpha}{\sqrt 2 L} \leq \frac{1}{\sqrt{2}} \alpha |Q_\alpha|$, then $|Q_\alpha|\sim \frac{|K|}{\alpha}$.
\end{enumerate}
As a result, $|K+\beta J^\perp|, |K+\alpha J+\beta J^\perp|\gtrsim \frac{\beta}{\alpha}|K|$ and estimate \eqref{aim tri} would be satisfied for $J=-Q_\alpha$ once we prove that
\begin{equation}\label{aim Q_alpha}
\sum_{\substack{1\leq \alpha \leq \beta\\\alpha \leq 2L|K|}}\langle r_\alpha\rangle^{-\sigma}\langle \frac{\beta}{\alpha}|K|\rangle^{-2\sigma} \lesssim L^2\langle K \rangle^{-\sigma}.
\end{equation} 
Now we have
\begin{align*}
\sum_{\substack{1\leq \alpha \leq \beta\\\alpha\leq 2L|K|}}\langle r_\alpha\rangle^{-\sigma}\langle \frac{\beta}{\alpha}|K|\rangle^{-2\sigma}&
\lesssim \sum_{1\leq \alpha \leq \beta}\langle \frac{\beta}{\alpha}|K|\rangle^{-2\sigma} \\
&\lesssim \left[ \sum_{\substack{1\leq \alpha \leq \beta\leq \alpha|K|^{-1}\\\alpha\leq 2L|K|}}1\right] + |K|^{-2\sigma} \left[ \sum_{\substack{1\leq \alpha \leq \beta\\\beta \geq \alpha|K|^{-1}, \alpha\leq 2L|K|}}\frac{\alpha^{2\sigma}}{\beta^{2\sigma}} \right], 
\end{align*}
where the first sum is zero if $|K|>1$. A direct computation now gives the result.

\medskip
\noindent
\underline{Subcase 2(b): $J=-Q_\beta$.}\label{J=-Q_beta} Here the analysis is more delicate and relies heavily on the fact that if $J=-Q_\beta$, then $Q_\beta$ is a visible lattice point. We will show that 
\begin{equation}\label{J=-Q_beta est}
\langle K \rangle^\sigma \sum_{1\leq \alpha \leq \beta} \langle K -\alpha Q_\beta\rangle^{-\sigma}\langle K -\alpha Q_\beta -\beta Q_\beta^\perp\rangle^{-\sigma}\langle K-\beta Q_\beta^{\perp}\rangle^{-\sigma}\lesssim L^2
\end{equation}

As before, we start with some useful observations:

\begin{enumerate}
\item\label{obs 1} Recall that $K=r_\beta +\beta Q_\beta^\perp$. Since $Q_\beta\neq 0$, then $\beta\leq 2L|K|$. Also, since $|r_\beta|\leq \frac{1}{\sqrt 2} \frac{\beta}{L}\leq \frac{1}{\sqrt 2} \beta|Q_\beta|$, then $|Q_\beta|\sim \frac{|K|}{\beta}$.
\item\label{obs 2} Since $Q_\beta^\perp$ is visible, then $\beta=g.c.d ( LK-Lr_\beta) $ (For $m=(m_1, m_2)\in (\Z^2)^*$, we use the notation $g.c.d (m)$ to denote the greatest common divisor of the $|m_1|$ and $|m_2|$). 
\item With $J=-Q_\beta$, we have that $K+\beta J^\perp=r_\beta, K+\alpha J=r_\beta-\alpha Q_\beta+\beta Q_\beta^{\perp}$, $K+\alpha J+\beta J^\perp=r_\beta -\alpha Q_\beta$.
\item\label{obs 4} Due to the orthogonality relation,
$$
|K|^2+|r_\beta-\alpha Q_\beta|^2=|r_\beta|^2+|r_\beta-\alpha Q_\beta+\beta Q_\beta^\perp|^2, 
$$
we have that either $|r_\beta|\geq |K|/\sqrt 2$ or $|r_\beta -\alpha Q_\beta+\beta Q_\beta^\perp|\geq |K|/\sqrt 2$. But $|r_\beta-\alpha Q_\beta+\beta Q_\beta^\perp|\geq (\sqrt 2- 1) |r_\beta|$ (since $|r_\beta|\leq \frac{\beta}{\sqrt 2L}\leq \frac{1}{\sqrt 2}\beta |Q_\beta^\perp|\leq \frac{1}{\sqrt 2}|-\alpha Q_\beta+\beta Q_\beta^\perp|$). Therefore, we always have that 
$$
|r_\beta -\alpha Q_\beta+\beta Q_\beta^\perp|\geq \frac{|K|}{100}.
$$
As a consequence of Observation \ref{obs 4}, we will be done once we show that
\begin{equation}\label{aim Q_beta0}
\sum_{1\leq \alpha \leq \beta \leq 2L|K|} \langle r_\beta\rangle^{-\sigma}\langle r_\beta-\alpha Q_\beta\rangle^{-\sigma} \lesssim L^2
\end{equation}

\item Next we reduce \eqref{aim Q_beta0} to the case when $\beta \geq 200L$. Indeed, 
$$
\sum_{\substack{1\leq \alpha \leq \beta\leq 200L}} \langle r_\beta\rangle^{-\sigma}\langle r_\beta-\alpha Q_\beta\rangle^{-\sigma} \leq \sum_{\substack{1\leq \alpha \leq \beta\leq 200L}} 1\lesssim L^2.
$$
Since the complementary range of $\beta$ is non-empty only when $|K|\geq 100$ (from observation (1) above), we are now reduced to proving that for any $|K|\geq 100$, we have
\begin{equation}\label{aim Q_beta1}
\sum_{\substack{1\leq \alpha \leq \beta\\ 200L \leq \beta \leq 2L|K|}} \langle r_\beta\rangle^{-\sigma}\langle r_\beta-\alpha Q_\beta\rangle^{-\sigma} \lesssim L^2
\end{equation}

\item\label{obs b6} Finally, we can also assume that $|r_\beta|\leq |K|/100$ and $|r_\beta-\alpha Q_\beta|\leq |K|/100$, since otherwise, one can trivially estimate the left-hand-side of \eqref{aim Q_beta1} by
$$
|K|^{-\sigma}\sum_{\substack{1\leq \alpha \leq \beta\\ 200L \leq \beta \leq 2L|K|}} 1\lesssim L^2 |K|^{2-\sigma}\lesssim L^2.
$$ 

\end{enumerate}

We start by estimating the sum in $\alpha$ in \eqref{aim Q_beta1} by 
$$
\sum_{1\leq \alpha \leq \beta}\langle r_\beta-\alpha Q_\beta\rangle^{-\sigma}\lesssim \min\left(\beta, \sum_{\alpha\geq 0} \langle \alpha Q_\beta\rangle^{-\sigma}\right)\lesssim \begin{cases}\min(\beta, |Q_\beta|^{-1}), \quad &\text{if }|Q_\beta|\leq 1\\
1 \quad &\text{if }|Q_\beta|\geq 1.
\end{cases}
$$

As a result, we have that
\begin{align}\label{Qbeta 001}
\sum_{\substack{1\leq \alpha \leq \beta\\ 200L \leq \beta \leq 2L|K|}} \langle r_\beta\rangle^{-\sigma}\langle r_\beta-\alpha Q_\beta\rangle^{-\sigma}\lesssim \sum_{\substack{200L \leq \beta \leq 2L|K|\\ |Q_\beta|\leq 1}}\langle r_\beta\rangle^{-\sigma}\min(\beta, |Q_\beta|^{-1}) +\sum_{\substack{200L \leq \beta \leq 2L|K|\\ |Q_\beta|\geq 1}}\langle r_\beta\rangle^{-\sigma}
\end{align}

The second sum above is easy to estimate: as $\beta$ ranges over $[200L,2L|K|]$, $r_\beta$ varies in the ball $B(0, \frac{|K|}{100})$ (see Observation \ref{obs b6} above). Moreover, for each $r \in B(0, \frac{|K|}{100})$, there is at most one $\beta$ for which $r=r_\beta$ (in fact, $\beta=g.c.d (LK -Lr)$ by Observation \ref{obs 2}. As a result, the second sum on the right-hand-side of \eqref{Qbeta 001} can be estimated by:
\begin{align*}
\sum_{r\in B(0, \frac{|K|}{100})}\langle r\rangle^{-\sigma}\lesssim L^2.
\end{align*}

Therefore, we will be done once we show that 
\begin{equation}\label{aim Q_beta 3}
\sum_{\substack{200L \leq \beta \leq 2L|K|\\ |Q_\beta|\leq 1}}\langle r_\beta\rangle^{-\sigma}\min(\beta, |Q_\beta|^{-1})\lesssim L^2.
\end{equation}

Recall from Observation \ref{obs 1} above, that $|Q_\beta|\sim \frac{|K|}{\beta}$, which means, since $|K|\geq 100$, that $\min(\beta, |Q_\beta|^{-1})=|Q_\beta|^{-1}\sim\frac{\beta}{|K|}$. Again, we change variables here from $\beta$ to $r_\beta$ due to the one-to-one correspondence between $\beta$ and $r$ ($\beta=g.c.d (LK-Lr_\beta)$  since $Q_\beta$ is visible). Consequently the left-hand side of \eqref{aim Q_beta 3} is estimated by
\begin{align*}
|K|^{-1}\sum_{\substack{100L \leq \beta \leq 2L|K|\\ |Q_\beta|\leq 1}}\langle r_\beta\rangle^{-\sigma}\beta\lesssim |K|^{-1}\sum_{\substack{|r|\leq \frac{|K|}{100}\\ L\leq g.c.d (LK-Lr)\leq 2L|K|}} \langle r\rangle^{-\sigma} g.c.d (LK-Lr).
\end{align*}
Let $k=LK\in \Z^2$, and $m=Lr\in\Z^2$, then the above sum is equal to 
\begin{align*}
&\frac{L}{|k|}\sum_{\substack{|m|\leq \frac{|k|}{100}\\ L\leq g.c.d (k-m)\leq 2|k|}}\langle \frac{m}{L}\rangle^{-\sigma} g.c.d(k-m)=\frac{L}{|k|}\sum_{\substack{n \in \Z^2, |k-n|\leq \frac{|k|}{100}\\ L\leq g.c.d (n)\leq 2|k|}}\langle \frac{k-n}{L}\rangle^{-\sigma} g.c.d(n)\\
&\lesssim\frac{L}{|k|}\sum_{j=L}^{2|k|}j \left(\sum_{\substack{n \in \Z^2, |k-n|\leq \frac{|k|}{100}\\g.c.d (n)=j}}\langle \frac{k-n}{L}\rangle^{-\sigma} \right)
\end{align*}
We claim that the term in parentheses is bounded by $\frac{L^2}{j^2}$. Assuming this claim, we get that 
\begin{align*}
\frac{L}{|k|}\sum_{\substack{|m|\leq \frac{|k|}{100}\\ L\leq g.c.d (k-m)\leq 2|k|}}\langle \frac{m}{L}\rangle^{-\sigma} g.c.d(k-m)\lesssim \frac{L}{|k|}\sum_{j=L}^{2|k|}j (\frac{L^2}{j^2})\lesssim L^2 \frac{\log(|k|/L)}{|k|/L}=L^2 \frac{\log |K|}{|K|}\lesssim L^2
\end{align*}
since $|K|\geq 1$. To prove the above-mentioned claim we notice that if $g.c.d (n)=j$, then $n$ is an element of the lattice $j \Z^2$, and hence
\begin{align*}
\sum_{\substack{n \in B(k, \frac{|k|}{100})\\g.c.d(n)=j}}\langle \frac{k-n}{L}\rangle^{-\sigma}\sim\sum_{l=0}^{\log_2(\frac{|k|}{L})}2^{-\sigma l} \sum_{\substack{n \in B(k, 2^lL) \\g.c.d(n)=j}} 1\lesssim \sum_{l=0}^{\log_2(\frac{k}{L})}2^{-\sigma l} \frac{2^{2l}L^{2}}{j^2}\lesssim \frac{L^2}{j^2},
\end{align*} 
since $\sigma>2$. This finishes the proof of Theorem \ref{tri est theo}.

\end{proof}

\begin{remark}\label{beta>L}
Combining Remark \ref{beta>L2} with estimates \eqref{J=-Q_alpha est} and \eqref{J=-Q_beta est}, we obtain that restricting the $\beta$ sum to $\beta\geq L$, one obtains the better estimate
\begin{equation}\label{beta>Lest}
\left\|\sum_{\substack{1\leq \alpha\leq \beta\\\beta\geq L}}\sum_{\substack{J\in \Z^2_L\\ visible}} a_{K+\alpha J}\,b_{K+\alpha J+\beta J^\perp}\,c_{K+\beta J^\perp}\,\right\|_{X^\sigma_L}\lesssim L^2 \|a_K\|_{X^\sigma_L}\|b_K\|_{X^\sigma_L}\|c_K\|_{X^\sigma_L}.
\end{equation}
This will be useful in the following section.
\end{remark}

\subsection{Estimates on non-resonant level sets}\label{non sharp estimates}
In order to effectively bound non-resonant interactions in Section \ref{proof of approx theorem section}, we will need estimates on trilinear sums taken over non-resonant tuples $(K_1, K_2, K_3)\in \mathcal S(K)$ that belong to some fixed non-resonant level set of the function $\Omega=|K_1|^2-|K_2|^2+|K_3|^2-|K|^2$, i.e. tuples $(K_1, K_2, K_3)$ belonging to sets of the form  
\begin{equation}\label{def of R_mu}
  \mathcal R_\mu (K)\overset{def}{=}\{(K_1, K_2, K_3) \in (\Z_L^2)^3: K_1-K_2+K_3=K, \; \mbox{and} \; |K_1|^2-|K_2|^2+|K_3|^2-|K|^2=\mu\}, 
\end{equation}
for $\mu \in \R$. In contrast to the resonant sums (corresponding to $\mu=0$) treated above, finding the sharpest dependence on $L$ in the estimates on these sums is not as crucial. The following proposition is sufficient for our purposes.
\begin{proposition}\label{R_mu sum prop} 
Let $L \geq 1$ and suppose that $|\mu| \leq L^{10}$, then for any $\epsilon>0$ and any $\sigma>2$, the following estimate holds
\begin{equation}\label{R_mu sum est}
\left\|\sum_{\RR_\mu(K)} a_{K_1}b_{K_2}c_{K_3}\right\|_{X^\sigma_L}\lesssim_{\sigma, \epsilon} L^{2+\epsilon} \|a_K\|_{X^\sigma_L}\|b_K\|_{X^\sigma_L}\|c_K\|_{X^\sigma_L}. 
\end{equation}
\end{proposition}

The proof of this lemma depends on lattice counting arguments similar to those used by Bourgain to prove periodic Strichartz estimates \cite{Bourgain, Bourgain3}. In particular, we rely on the following lemma, whose proof is implicit in the above cited works.

\begin{lemma}\label{lattice counting lemma}
Let $C(R)$ denote the circle in $\R^2$ centered at $a\in \Z^2$ of radius $R$. For any $\epsilon>0$, the number of integer lattice points on the circle and inside a box of size $L$ is bounded by $C_\epsilon L^\epsilon$ independent of $a$ and $R$.
\end{lemma}
\begin{proof}
We split into two cases:

\medskip
\noindent
\textbf{Case 1: $R\leq L^{10}$.} In this case, we use the upper bound on the number of lattice points on the circle $C(R)$. If $z$ is such a lattice point, then $(m,n)=z-a\in \Z^2$ and $(m,n)\in \Z^2$ satisfies $R^2=m^2+n^2=(m+in)(m-in)$. Therefore $m+in$ is a divisor of $R^2$ in the integral domain $\Z+i\Z$, which is bounded by $C_\epsilon R^{\frac{\epsilon}{100}}$ for any $\epsilon>0$ (see for instance \cite{HW, TaoPL}), which is sufficient in this case.

\medskip
\noindent
\textbf{Case 2: $R\geq L^{10}$.} Here we will use an idea that goes back to Jarnick \cite{Jarnick} relying on the fact that the area of any triangle in $\R^2$ with vertices in the lattice $\Z^2$ is lower bounded by $1/2$ (being $1/2$ of the cross product of two adjacent edges). We will show that the if $L \ll R^{1/3}$, then there cannot be more than two lattice points on the circle $C(R)$ inside a box of size $L$. To see this, suppose that $A, B, C$ are three such lattice points with $\theta \overset{def}{=} \widehat{AC} < \pi$  and  $B$ being on the arc between $A$ and $C$. Then
$$
\frac{1}{2}\leq \text{Area of triangle }ABC \leq \text{Area of the  outer  cap between segment and the arc described by $AC$}. 
$$
As  $\theta$ is the angle of the sector described by $AC$, a simple calculation shows that the area of this cap is $\frac{R^2}{2}(\theta-\sin \theta)\lesssim R^2 \theta^3$ if $\theta$ is small enough. But $\theta \lesssim LR^{-1}$, and hence we get that 
$$
\frac{1}{2}\lesssim L^3R^{-1}.
$$
This gives the needed contradiction if $R\ll L^{1/3}$. 
\end{proof}

\begin{proof}[Proof of Proposition \ref{R_mu sum prop}] Due to the symmetry between $K_1$ and $K_3$, we may assume without any loss of generality that $|K_1|\geq |K_3|$. We split the analysis into two cases:

\medskip
\textbf{Case 1: $|K_1|\geq \frac{|K|}{100}$ or $|K|\leq 100$.}
In this case, \eqref{R_mu sum est} reduces to showing that
\begin{equation}\label{case2 aim mu2}
\sup_{K \in \Z^2_L}\sum_{\RR_\mu(K)}\langle K_2\rangle^{-\sigma}\langle K_3\rangle^{-\sigma}\lesssim_\epsilon L^{2+\epsilon}.
\end{equation}
Now we change variable and define $N_i=K_i-K$ ($i=1,2,3$) so that the new variables satisfy the relations:
$$
N_1+N_3=N_2\qquad 2(N_3-N_2)\cdot N_3= |N_1|^2+|N_3|^2-|N_2|^2=\mu.
$$
As a result, we have that 
\begin{align}
\label{eq:klm}
\sum_{\RR_\mu(K)}\langle K_2\rangle^{-\sigma}\langle K_3\rangle^{-\sigma}=\sum_{N_2\in \Z^2_L}\langle K +N_2\rangle^{-\sigma}\sum_{l=0}^\infty 2^{-\sigma l}\sum_{\substack{N_3\in \Z^2_L\cap B(-K,2^l)\\2(N_3-N_2)\cdot N_3=\mu }}1
\end{align}
Let $n_3=LN_3, n_2=LN_2, m =\frac{L^2\mu}{2}\in \Z$ (otherwise there is nothing to prove). Then we have that 
$$
(n_3-n_2)\cdot n_3=m \Leftrightarrow |2n_3 -n_2|^2=4m+|n_2|^2,
$$
which means that for a fixed $n_2$, $2n_3$ belongs to the circle centered at $n_2$ of radius $\sqrt{4m+|n_2|^2}$. Now in the last sum in \eqref{eq:klm}, for a given $l \in \N$, we have that $n_3 = LN_3$ belongs to balls of size $L 2^l$. 
Using Lemma \ref{lattice counting lemma}, we have that for any $n_2$,  $l \geq 0$, and $\epsilon > 0$, 
$$
\sum_{\substack{N_3\in \Z^2_L\cap B(-K,2^l)\\2(N_3-N_2)\cdot N_3=\mu }}1\quad \lesssim_\epsilon (2^l L)^\epsilon.
$$
Consequently, we have that
\begin{align*}
\sum_{\RR_\mu(K)}\langle K_2\rangle^{-\sigma}\langle K_3\rangle^{-\sigma}\lesssim_\epsilon L^\epsilon\sum_{N_2\in \Z^2_L}\langle K +N_2\rangle^{-\sigma}\sum_{l=0}^\infty 2^{-\sigma l}2^{\epsilon \ell}\lesssim L^{2+\epsilon},
\end{align*}
which is \eqref{case2 aim mu2}.

\medskip
\textbf{Case 2: $|K_1|< \frac{|K|}{100}$  and $|K|Ê> 100$.} This case is non-empty only when $\mu<0$ due to the relation 
\begin{equation}\label{orth with mu}
|K_1|^2+|K_3|^2=|K|^2+|K_2|^2+\mu. 
\end{equation}
From the relation $K_2=K_1-K_3-K$ and recalling that $|K_3|\leq |K_1|$, we deduce that $|K_2|\geq \frac{|K|}{50}$, and as a result, we are reduced to showing that 
\begin{equation}\label{case2 aim mu}
\sup_{K \in \Z^2_L} \sum_{\RR_\mu(K)} \langle K_1\rangle^{-\sigma}\langle K_3\rangle^{-\sigma} \lesssim_\epsilon L^{2+\epsilon}
\end{equation}

Now we change variable and define $N_i=K_i-K$ ($i=1,2,3$) so that the new variables satisfy the relations:
$$
N_1+N_3=N_2\qquad -2N_1\cdot N_3= |N_1|^2+|N_3|^2-|N_2|^2=\mu.
$$
Note that since $|K_3|\leq |K_1|\leq \frac{|K|}{100}$, we have that $|N_1|\sim |K|\sim |N_3|$. Now we write, $N_1=\alpha J$ where $\alpha \in \N$ and $J$ is visible and denote by $m \overset{def}{=}L^2 \frac{\mu}{2} \in \Z$ (otherwise there is nothing to prove). But if $2N_1\cdot N_3=-\mu$, we must have that $\alpha$ divides $m$. 

The fact that $N_3$ satisfies the relation $J \cdot  N_3=\frac{m}{L^2\alpha}$ shows that $K+N_3$ ranges over a one-dimensional lattice of spacing $|J|$:
Indeed, if we denote $n_3 = LN_3$ and $j=L J$, then the equation $j\cdot n_3=\frac{m}{2}$ has a general solution of the form $n_3=n_3^{(0)}+\beta j^{\perp}$ with $\beta \in \Z$ and $n_3^{(0)}$ being any particular solution of the equation $j\cdot n_3^{(0)}=\frac{m}{2}$ (see \cite{HW} for the elementary argument).  Consequently, for any fixed $\alpha, J$, we have that 
$$
\sum_{N_3\cdot J=\frac{m}{L^2\alpha}} \langle K +N_3\rangle^{-\sigma} \lesssim \sum_{\beta \in \Z}\langle \beta J^\perp \rangle^{-\sigma}\lesssim \max(1, |J|^{-1}).
$$

As a result, if we write as in the previous section $K=r_\alpha+\alpha Q_\alpha$, with $r_\alpha\in \BB(\alpha/L) \cap \mathbb{Z}_L^2$ and $Q_\alpha \in \Z^2_L$ (see \eqref{eq:rQ}), 
then we can estimate
\begin{align*}
\sum_{\RR_\mu(K)} \langle K_1\rangle^{-\sigma}\langle K_3\rangle^{-\sigma}\leq &\sum_{\alpha |m} \sum_{\substack{J \in \Z^2_L\\ J \, visible}}\langle K +\alpha J\rangle^{-\sigma} \sum_{N_3\cdot J=\frac{m}{L^2\alpha}}\langle K+N_3\rangle^{-\sigma}\\
\lesssim& \sum_{\alpha |m} \sum_{\substack{J \in \Z^2_L\setminus\{-Q_\alpha\}\\ J \, visible}}\langle r_\alpha +\alpha (J+Q_\alpha)\rangle^{-\sigma} \sum_{N_3\cdot J=\frac{m}{L^2\alpha}}\langle K+N_3\rangle^{-\sigma}\\
&+\sum_{\alpha |m} \langle r_\alpha \rangle^{-\sigma} \sum_{N_3\cdot Q_\alpha=\frac{-m}{L^2\alpha}}\langle r_\alpha +\alpha Q_\alpha+N_3\rangle^{-\sigma}\\
\lesssim& \sum_{\alpha |m} \sum_{\substack{J \in \Z^2_L\setminus\{-Q_\alpha\}\\ J \, visible}}\langle r_\alpha +\alpha (J+Q_\alpha)\rangle^{-\sigma} \max(1, |J|^{-1})\\
&+\sum_{\alpha |m} \max(1, |Q_\alpha|^{-1}).
\end{align*}
For the first sum above we recall that $|J|=\frac{|N_1|}{\alpha}\sim \frac{|K|}{\alpha}$ and use Lemma \ref{elementary}, as for the second sum we simply estimate $|Q_\alpha|^{-1}\leq L$ and use that the bound on the number of divisors \cite{HW, TaoPL} of $m$ given by $C_\epsilon |m|^{\epsilon/100}\lesssim_\epsilon L^\epsilon$, as $|m| = L^2 \frac{|\mu|}{2} \lesssim L^{12}$ by assumption on $\mu$. As a result,  we get that
\begin{align*}
\sum_{\RR_\mu(K)} \langle K_1\rangle^{-\sigma}\langle K_3\rangle^{-\sigma}\lesssim \left[\sum_{\alpha |m}\frac{L^2}{\alpha^2}\max(\frac{\alpha}{|K|},1)\right]+L^{1+\epsilon}\lesssim L^{2+\epsilon},
\end{align*}
which is \eqref{case2 aim mu}. This finishes the proof of the proposition.
\end{proof}

\section{Proof of Theorem~\ref{disctocont}: convergence of $\mathcal{T}_L$ to $\mathcal{T}$}\label{T_L to T section}

The aim of this section is to prove Theorem~\ref{disctocont}. We first give the proof by assuming the co-prime equidistribution Lemma \ref{equidistribution lemma}. In a second part, we prove this lemma. 

%

\subsection{Proof of Theorem~\ref{disctocont}}
 
We will use the following expression for $\mathcal{T}_L$ (see \eqref{TL} and \eqref{eq:rect1}): 
$$
\mathcal{T}_L(g,g,g)(K) = 
\frac{1}{\GG}\sum_{K_1\cdot K_2=0} g(K+K_1)\bar g(K+K_1+K_2)g(K+K_2) \quad \mbox{with} \quad
\GG\overset{def}{=}\frac{2L^2\log L}{\zeta(2)}.
$$
Without loss of generality we assume that $B=1$ in the statement of Theorem \ref{disctocont}. 

\medskip
\noindent 
\underline{Step 1: decomposition of $\mathcal{T}_L$}. Start by writing
\begin{align*}
&\frac{1}{\GG}\sum_{K_1\cdot K_2=0} g(K+K_1)\bar g(K+K_1+K_2)g(K+K_2)\\
& =\frac{1}{\GG}\sum_{\substack{K_1\cdot K_2=0\\K_1, K_2 \neq 0}} g(K+K_1)\bar g(K+K_1+K_2)g(K+K_2)
+\frac{2}{\GG}\left(\sum_{K'} |g(K')|^2\right)g(K) -\frac{1}{\GG}|g(K)|^2g(K)\\
&=I_1+I_2+I_3. 
\end{align*}
The terms $I_2$ and $I_3$ can be easily bounded by $\frac{L^2}{\mathfrak G}\|g\|^2_{\ell_L^2}\|g\|_{\ell^\infty_L}$ and $\mathfrak G^{-1} \|g\|_{\ell^\infty_L}^3$
respectively. As for the main term $I_1$, we parameterize the sum over $\{k_1, k_2 \in \Z^2\setminus \{0\}: k_1 \cdot k_2=0\}$ as $k_1=\alpha(p,q), k_2=\beta(-q,p)$ with $\alpha, \beta, p, q \in \Z$ such that $\alpha \geq 1, \beta \neq 0$ and $g.c.d (|p|,|q|)=1$ (see Definition \ref{defvisible}). We then write it as:
\begin{align}
I_1&=\frac{1}{\GG}\sum_{\substack{\alpha \geq 1, \\ \beta \in \Z^*}}\sum_{g.c.d(p,q)=1} g(K+\frac{\alpha}{L} (p,q))\bar g(K+\frac{\alpha}{L}(p,q)+\frac{\beta}{L} (-q,p))g(K+\frac{\beta}{L}(-q, p))\nonumber\\
&=\frac{1}{\GG}\sum_{\substack{\alpha \geq 1, \\ \beta \geq 1}}\sum_{g.c.d(p,q)=1} g(K+\frac{\alpha}{L} (p,q))\bar g(K+\frac{\alpha}{L}(p,q)+\frac{\beta}{L}(-q,p))g(K+\frac{\beta}{L}(-q, p))\label{firsthalf}\\
&+\frac{1}{\GG}\sum_{\substack{\alpha \geq 1, \\ \beta \geq 1}}\sum_{g.c.d(p,q)=1} g(K-\frac{\alpha}{L} (p,q))\bar g(K-\frac{\alpha}{L}(p,q)+\frac{\beta}{L}(-q,p))g(K+\frac{\beta}{L}(-q, p))\label{sechalf}, 
\end{align}
where we used the symmetry $(p,q)\to -(p,q)$. Now we write:
\begin{align*}
\eqref{firsthalf}=& \frac{1}{\GG}\sum_{\substack{1\leq \alpha \leq  \beta }}\sum_{g.c.d(p,q)=1} +\frac{1}{\GG}\sum_{\substack{1\leq \beta \leq  \alpha }}\sum_{g.c.d(p,q)=1} -\frac{1}{\GG}\sum_{\substack{1\leq \alpha=\beta}}\sum_{(p,q)=1} \\
=&\frac{1}{\GG}\sum_{\substack{1\leq \alpha \leq  \beta }}\sum_{g.c.d(p,q)=1}g(K+\frac{\alpha}{L} (p,q))\bar g(K+\frac{\alpha}{L}(p,q)+\frac{\beta}{L}(-q,p))g(K+\frac{\beta}{L}(-q, p))\\
&+\sum_{\substack{1\leq \beta \leq  \alpha }}\sum_{g.c.d(p',q')=1}g(K+\frac{\alpha}{L} (-q',p'))\bar g(K+\frac{\alpha}{L}(-q',p')-\frac{\beta}{L}(p',q'))g(K-\frac{\beta}{L}(p', q'))\\
&-\frac{1}{\GG}\sum_{K'\in (\Z^2_L)^*}g(K+K')\bar g(K+K'+K'^\perp)g(K+K'^\perp)\\
=&\frac{1}{\GG}\sum_{\substack{1\leq \alpha \leq  \beta }}\sum_{g.c.d(p,q)=1}g(K+\frac{\alpha}{L} (p,q))\bar g(K+\frac{\alpha}{L}(p,q)+\frac{\beta}{L}(-q,p))g(K+\frac{\beta}{L}(-q, p))\\
&+\sum_{\substack{1\leq \alpha \leq  \beta}}\sum_{g.c.d(p,q)=1}g(K-\frac{\alpha}{L}(p, q))\bar g(K-\frac{\alpha}{L}(p,q)+\frac{\beta}{L}(-q',p'))g(K+\frac{\beta}{L} (-q,p))\\
&-\frac{1}{\GG}\sum_{K'\in (\Z^2_L)^*}g(K+K')\bar g(K+K'+K'^\perp)g(K+K'^\perp)
\end{align*}
(recall that the superscript $\perp$ the rotation in $\R^2$ by $\pi/2$ in the counter-clockwise direction). Arguing similarly
\begin{align*}
\eqref{sechalf}=& \frac{1}{\GG}\sum_{\substack{1\leq \alpha \leq  \beta }}\sum_{(p,q)=1} (\cdots) +\frac{1}{\GG}\sum_{\substack{1\leq \beta \leq  \alpha }}\sum_{(p,q)=1} (\cdots) -\frac{1}{\GG}\sum_{\substack{1\leq \alpha=\beta}}\sum_{(p,q)=1} (\cdots)\\
=&\frac{1}{\GG}\sum_{\substack{1\leq \alpha \leq  \beta }}\sum_{(p,q)=1}g(K-\frac{\alpha}{L} (p,q))\bar g(K-\frac{\alpha}{L}(p,q)+\frac{\beta}{L}(-q,p))g(K+\frac{\beta}{L}(-q, p))\\
&+\sum_{\substack{1\leq \alpha \leq  \beta}}\sum_{(p,q)=1}g(K+\frac{\alpha}{L}(p, q))\bar g(K+\frac{\alpha}{L}(p,q)+\frac{\beta}{L}(-q,p))g(K+\frac{\beta}{L} (-q,p))\\
&-\frac{1}{\GG}\sum_{K'\in (\Z_L^2)^*}g(K+K')\bar g(K+K'+K'^\perp)g(K+K'^\perp). 
\end{align*}

As a result,
\begin{align}
I_1&=\frac{2}{\GG}\sum_{\substack{1 \leq \alpha \leq \beta}}\sum_{(p,q)=1} g(K+\frac{\alpha}{L} (p,q))\bar g(K+\frac{\alpha}{L}(p,q)+\frac{\beta}{L} (-q,p))g(K+\frac{\beta}{L}(-q, p))\nonumber\\
&+\frac{2}{\GG}\sum_{\substack{1 \leq \alpha \leq \beta}}\sum_{(p,q)=1} g(K-\frac{\alpha}{L} (p,q))\bar g(K-\frac{\alpha}{L}(p,q)+\frac{\beta}{L} (-q,p))g(K+\frac{\beta}{L}(-q, p))\label{defofII_2}\\
&-\frac{2}{\GG}\sum_{K'\in (\Z^2_L)^*}g(K+K')\bar g(K+K'+K'^\perp)g(K+K'^\perp)\nonumber\\
&=II_1+II_2+II_3. 
\end{align}
Notice that the sum in $II_3$ is a sum over squares with vertex at $K$. Using the relation 
$$
|K|^2+|K+K'+K'^\perp|^2=|K+K'|^2+|K+K'^\perp|^2
$$
one notices that either $|K+K'|$ or $|K+K'^\perp|$ is larger than $|K|/2$. This allows to estimate $II_3(K)$ as follows:
\begin{align*}
II_3(K)\lesssim \GG^{-1}\langle K \rangle^{-\sigma}\sum_{K'\in (\Z^2_L)^*} \langle K +K' + K'^\perp\rangle^{-\sigma}\lesssim \langle K \rangle^{-\sigma} \frac{L^2}{\GG}.
\end{align*}
which means that $II_3$ has an acceptable contribution in $X^\sigma$, namely
$$
\|II_3(K)\|_{X^\sigma}\lesssim \frac{L^2}{\GG}.
$$

We will perform the analysis only on $II_1$ since that for $II_2$ is exactly the same except for the signature of $\alpha$. 

\medskip
\noindent
\underline{Step 2: restricting the range of $\beta$.}
Before performing the analysis on $II_1$, recall that we call points $z$ of $\mathbb{Z}^2$ \emph{visible} if $z=(p,q)$ with $g.c.d(|p|, |q|)=1$ (see Definition \ref{defvisible}). Consequently, we can write $II_1$ as:
\begin{equation}\label{def of II_1}
II_1\overset{def}{=}\frac{2}{\GG}\sum_{\substack{1 \leq \alpha \leq \beta}}\sum_{J\in \Z^2 \, visible} g(K+\frac{\alpha}{L} J)\bar g(K+\frac{\alpha}{L}J+\frac{\beta}{L} J^\perp)g(K+\frac{\beta}{L}J^\perp). 
\end{equation}
We split $II_1$ into a sum over $\beta < L$ and another over $\beta \geq L$ to get:
\begin{equation}
II_1=III_1 +III_2
\end{equation}
where
\begin{equation}
III_1=\frac{2}{\GG}\sum_{\substack{1\leq \alpha \leq \beta\\ \beta < L}}\sum_{J\in \Z^2 \, visible} g(K+\frac{\alpha}{L} J)\bar g(K+\frac{\alpha}{L}J+\frac{\beta}{L} J^\perp)g(K+\frac{\beta}{L}J^\perp)
\end{equation}
and 
\begin{equation}
III_2=\frac{2}{\GG}\sum_{\substack{1\leq \alpha \leq \beta\\ \beta \geq L}}\sum_{J\in \Z^2 \, visible} g(K+\frac{\alpha}{L} J)\bar g(K+\frac{\alpha}{L}J+\frac{\beta}{L} J^\perp)g(K+\frac{\beta}{L}J^\perp).
\end{equation}

Remark \ref{beta>L} says exactly that 
\begin{equation}
\|III_2\|_{X^\sigma} \lesssim \frac{L^2}{\GG},
\end{equation}
which means that the contribution of $III_2$ is an acceptable error.

\medskip
\noindent
\underline{Step 3: discrete to continuous limit in $J$.}
We move on to analyzing $III_1$. For this, we define $N\overset{def}{=}L/\beta$ and write the sum: 
\begin{multline*}
\sum_{\substack{J \in \Z^2\\   visible}} g(K+\frac{\alpha}{L} J)\bar g(K+\frac{\alpha}{L}J+\frac{\beta}{L} J^\perp)g(K+\frac{\beta}{L}J^\perp) \\Ê
=\sum_{\substack{J' \in \Z_N^2\\  visible}} g(K+\frac{\alpha}{\beta} J')\bar g(K+\frac{\alpha}{\beta}J'+J'^\perp)g(K+J'^\perp).
\end{multline*}
The estimate on this sum is contained in the following lemma. 
\begin{lemma}[Co-prime equidistribution lemma]\label{equidistribution lemma}
Let $\sigma>2$. For all $B > 0$, if $g :\R^2\to \C$ satisfies the following bounds:
$$
\|g\|_{X^{\sigma+1}(\R^2)} +\|\nabla g\|_{X^{\sigma+1}(\R^2)}\leq B, 
$$
then for all $N > 1$, we have 
\begin{equation}\label{equi}
\begin{split}
\mathcal E_N\overset{def}{=}&\left\|\frac{1}{N^2}\sum_{\substack{J \in \Z_N^2\\ J \; visible}} g(K+\frac{\alpha}{\beta} J)\bar g(K+\frac{\alpha}{\beta}J+J^\perp)g(K+J^\perp)\right. \\ 
&\qquad \qquad-\frac{1}{\zeta(2)} \left. \int_{\R^2}  g(K+\frac{\alpha}{\beta} z)\bar g(K+\frac{\alpha}{\beta} z+z^\perp)g(K+ z^\perp)dz\right\|_{X^\sigma(\R^2)} \lesssim  B^3\frac{1 + \log N}{N}. \quad
\end{split}
\end{equation}
\end{lemma}
We leave the proof of this lemma for later in order not to distract the reader. This allows us to write:

\begin{align*}
III_1&=\frac{2L^2}{\zeta(2)\GG} \sum_{1\leq \alpha \leq \beta \leq L} \frac{1}{\beta^2}\int_{\R^2} g(K+\frac{\alpha}{\beta} x)\bar g(K+\frac{\alpha}{\beta}x+x^\perp)g(K+x^\perp)\,dx+ \frac{2L^2}{\GG}\sum_{1\leq \alpha \leq \beta \leq L} \frac{1}{\beta^2}\mathcal E_{L/\beta}\\
&=IV_1+IV_2
\end{align*}
with 
\begin{align*}
\|IV_2\|_{X^\sigma}\lesssim \frac{L}{\GG}\sum_{1\leq \beta \leq L} (1+\log (L/\beta))\lesssim\frac{L^2}{\GG}.\end{align*}

\medskip
\noindent
\underline{Step 4: discrete to continuous limit in $\alpha$.}
In order to analyze $IV_1$, we view the $\alpha$ sum as a Riemann sum, which we will replace by the corresponding integral. For each $1\leq \beta \leq L$, denote:
\begin{align*}
\mathcal D_\beta \overset{def}{=}&\frac{1}{\beta}\sum_{1\leq \alpha \leq \beta} \int_{\R^2} g(K+\frac{\alpha}{\beta} x)\bar g(K+\frac{\alpha}{\beta}x+x^\perp)g(K+x^\perp)\,dx\\
&-\int_0^1\int_{\R^2} g(K+\lambda x)\bar g(K+\lambda x+x^\perp)g(K+x^\perp)\,dx\, d\lambda\\
&=\sum_{1\leq \alpha \leq \beta} \int_{\frac{\alpha-1}{\beta}}^{\frac{\alpha}{\beta}}\int_{\R^2}\left(g(K+\frac{\alpha}{\beta} x)-g(K+\lambda x)\right)\bar g(K+\frac{\alpha}{\beta}x+x^\perp)g(K+x^\perp)dx \,d\lambda\\
&+\sum_{1\leq \alpha \leq \beta} \int_{\frac{\alpha-1}{\beta}}^{\frac{\alpha}{\beta}}\int_{\R^2}g(K+\lambda x)\left(\bar g(K+\frac{\alpha}{\beta} x+x^\perp)-\bar g(K+\lambda x+x^\perp)\right)g(K+x^\perp)dx \,d\lambda.\\
\end{align*}
We only estimate the first term as the second is similar. It can be written
\begin{align*}
\sum_{1\leq \alpha \leq \beta} \int_{\frac{\alpha-1}{\beta}}^{\frac{\alpha}{\beta}}\int_{\lambda}^{\alpha/\beta}\int_{\R^2}x\cdot \nabla g (K+tx)\bar g(K+\frac{\alpha}{\beta}x+x^\perp)g(K+x^\perp)dx\,dt \,d\lambda.
\end{align*}

We will show that for each $1\leq \alpha\leq \beta$ and each $t\in [0,1]$, 
\begin{equation}\label{horse1}
\left\|\int_{\R^2}x\cdot\nabla g (K+tx)\bar g(K+\frac{\alpha}{\beta}x+x^\perp)g(K+x^\perp)\,dx\right\|_{X^\sigma}\lesssim 1.
\end{equation}

To see this we split into two cases: Either $|x|\leq 100|K|$, in which case we use the orthogonality relation
$$
|K|^2+|K+tx+x^\perp|^2=|K+tx|^2+|K+x^\perp|^2
$$
to obtain that either $|K+x^\perp|\geq|K|/2$ or $|K+tx|\geq |K|/2$. In both cases, we have since $|g(x)|+|\nabla g(x)|\lesssim \langle x\rangle^{-\sigma-1}$:
\begin{multline*}
\left| \int_{|x|\leq 100|K|}x\cdot\nabla g (K+tx)\bar g(K+\frac{\alpha}{\beta}x+x^\perp)g(K+x^\perp)\,dx\right|\\\lesssim \frac{|K|}{\langle K \rangle^{\sigma+1}}\int_{\R^2} \langle K+\frac{\alpha}{\beta}x+x^\perp\rangle^{-\sigma} \,dx\lesssim \langle K\rangle^{-\sigma},
\end{multline*}
which is \eqref{horse1}. In the case when $|x|\geq 100|K|$, then $|K+x^\perp|\geq |K|$ and $|K+\frac{\alpha}{\beta}x+x^\perp|\geq \frac12|\frac{\alpha}{\beta}x+x^\perp|\geq \frac12|x|$, and hence the estimate \eqref{horse1} follows directly.

With \eqref{horse1} in hand, Minkowski's inequality gives
\begin{equation}\label{bound on D_beta}
\|\mathcal D_\beta\|_{X^\sigma}\lesssim \frac{1}{\beta}.
\end{equation}

Consequently,
\begin{multline*}
IV_1=\frac{2L^2}{\zeta(2)\GG} \sum_{1\leq \beta \leq L} \frac{1}{\beta}\int_0^1\int_{\R^2} g(K+\lambda x)\bar g(K+\lambda x+x^\perp)g(K+x^\perp)\,dx\, d\lambda \\+\frac{2L^2}{\zeta(2)\GG} \sum_{1\leq \beta \leq L} \frac{1}{\beta}\mathcal D_{\beta}(K)
\end{multline*}
and hence 
\begin{align*}
IV_1
&=\frac{2\log^D(L)L^2}{\zeta(2)\GG} \int_0^1\int_{\R^2} g(K+\lambda x)\bar g(K+\lambda x+x^\perp)g(K+x^\perp)\,dx\, d\lambda +\frac{2L^2}{\zeta(2)\GG} \sum_{1\leq \beta \leq L} \frac{1}{\beta}\mathcal D_{\beta}(K)\\
&=V_1 +V_2,
\end{align*}
where $\log^D n = \sum_{k=1}^n \frac{1}{k}$, and where $V_2$ can be bounded as explained above by
\begin{equation}\label{bound on V_2}
\|V_2\|_{X^\sigma} \lesssim \frac{L^2}{\GG}. 
\end{equation}
Finally, using that $|\log^D L - \log L| = O(1)$, the definition of $\GG$, and Theorem~\ref{tri est theo}, we obtain
$$
V_1 = \int_0^1\int_{\R^2} g(K+\lambda x)\bar g(K+\lambda x+x^\perp)g(K+x^\perp)\,dx\, d\lambda + O_{X^\sigma}\left(\frac{1}{\log L}\right)
$$
(with the notation $A = B + O_{X^\sigma}(C)$ if $\|A - B \|_{X^\sigma} \lesssim C$).

\medskip
\noindent
\underline{Step 5: conclusion.} Combining all the above, we get that if $\GG=\frac{2\log^D(L)L^2}{\zeta(2)}$ then:
$$
II_1= \int_0^1 \int_{\R^2} g(K+\lambda z) \bar g(K+\lambda z+z^\perp)g(K+z^\perp) dz\, d\lambda +O_{X^\sigma}\left(\frac{1}{\log L}\right).
$$
The same analysis for $II_2$ gives (recalling that the only difference with the above is the sign of $\alpha$)
\begin{align*}
II_2=\int_0^1 \int_{\R^2} g(K-\lambda z) \bar g(K-\lambda z+z^\perp)g(K+z^\perp) dz\, d\lambda +O_{X^\sigma}\left(\frac{1}{\log L}\right)\\
=\int_{-1}^0 \int_{\R^2} g(K+\lambda z) \bar g(K+\lambda z+z^\perp)g(K+z^\perp) dz\, d\lambda +O_{X^\sigma}\left(\frac{1}{\log L}\right),
\end{align*}
and as a final result we get that:
\begin{equation}\label{dtcfinal}
\left\|\frac{1}{\GG}\sum_{K_1 \cdot K_3 = 0} g(K+K_1)\bar g(K+K_1+K_2)g(K+K_2)- \mathcal T(g,g,g)(K)\right\|_{X^\sigma} \lesssim \frac{1}{\log L},\\
\end{equation}
which finishes the proof.

\subsection{Proof of the co-prime equidistribution Lemma \ref{equidistribution lemma}}\label{co-prime equidistribution subsection} We start by recalling some elementary number theory.

\medskip
\noindent
\underline{Step 0: number theoretic preliminaries.} For any $N > 0$,  define for any subset $P$ of $\R^2$ the function:
$$
g_{P}(N)=\#\{z \in (\Z^2_N)^*\,\cap \,P\}.
$$
Note that for reasonably nice regions $P$, the area of $P$ is equal to the quantity:
$$
V_P=\lim_{N \to +\infty} \frac{g_P(N)}{N^2}.
$$
Recall that a lattice point $z\in (\Z^2_N)^*$ is \emph{visible} (from $0$) if the segment $[0z]$ does not intersect the lattice $\Z^2_N$ apart from the endpoints $0$ and $z$; equivalently, $z=N^{-1}(p,q)$ where $(p,q)$ are co-prime (see Definition \ref{defvisible}). This definition allows one to define the following counter-part of $g_P$:
\begin{equation*}
f_P(N)\overset{def}{=}\#\{z \in \Z^2_N \cap P: z \hbox{ is visible}\}.
\end{equation*}
It is a classical result in analytic number theory \cite{HW} that
\begin{equation}
\label{rossignol}
\lim_{N \to +\infty} \frac{f_P(N)}{N^2} = \frac{V_P}{\zeta (2)},
\end{equation}
where $\zeta(2)=\frac{\pi^2}{6}$ is the Riemann zeta function at $2$. The generalization of this result from characteristic functions to more general functions, along with the relevant discrepancy or error estimates, will be a byproduct of the analysis in this section, and we include it at the end in Corollary \ref{discrep corollary}. 

To prove Lemma  \ref{equidistribution lemma}, we need to estimate the discrepancy between the equidistribution sum in \eqref{equi} and the integral in the strong norm $X^\sigma(\Z^2_L)$. The proof of (\ref{rossignol}), which is actually the prototype for our proof, is based on an inversion formula for the multiplicative M\"obius function $\mu(m): \N \to \Z$, whose definition we recall first: 
\begin{equation}\label{def of mob}
\mu(k)=\begin{cases}
1 \quad \hbox{if } k=1,\\
(-1)^\kappa \quad\hbox{if $k=p_1 p_2 \ldots p_\kappa$ and all the primes $p_1, \ldots, p_\kappa$ are different},\\
0 \quad \hbox{if $k$ is not square-free}.\\
\end{cases}
\end{equation}

Recall that the M\"obius function is multiplicative (i.e. $\mu(ab)=\mu(a)\mu(b)$ whenever $a$ and $b$ are co-prime) and satisfies the following fundamental inversion formula proved by M\"obius \cite{HW}:
\begin{equation*}
g(n)=\sum_{d|n}f(d) \Rightarrow f(n)=\sum_{d|n}\mu(\frac{n}{d}) g(d),
\end{equation*}
from which follows:
\begin{lemma}[\cite{HW}]\label{Mobinv2} Suppose that $f: \N \to \C$ satisfies, for all $x \in \N$, $\sum_{m,n} |f(mn x)|=\sum_{k}d(k) |f(kx)|<\infty$, where $d(k)$ is the number of divisors of $k$. Then for any $x\in \N$, there holds
\begin{equation}
g(x)=\sum_{m=1}^\infty f(mx) \Leftrightarrow f(x)=\sum_{n=1}^\infty\mu(n) g(nx).
\end{equation}
\end{lemma}
The final property of $\mu(n)$ that will be needed in the proof is its connection to the Riemann zeta function defined for $\operatorname{Re}s>1$ by
$$
\zeta(s)=\sum_{n \in \N} \frac{1}{n^s}.
$$
Applying Lemma \ref{Mobinv2} with $f(x)=x^{-s}$ and $g(x)=\zeta(s)x^{-s}$, one gets that
\begin{equation*}
\frac{1}{\zeta (s)}=\sum_{n = 1}^{\infty} \frac{\mu(n)}{n^s}. 
\end{equation*}

\bigskip
\noindent
\underline{Step 1: a new formulation for \eqref{equi}.} 
Again we assume without any loss of generality that $B=1$. Given $u:\R^2\to \C$ and such that $\sup_{N > 1} N^{-2} \sum_{z\in \Z^2_N}|u(z)|<\infty$, we define the following two sums:
\begin{equation}\label{S and D}
S(u,N)\overset{def}{=}\sum_{\substack{z\in \Z^2_N \\visible}} u(z) \quad \mbox{and} \quad D(u,N)=\sum_{z\in (\Z^2_N)^*}u(z).
\end{equation}
Note that if $u \in X^\kappa(\R^2)$ for some $\kappa>2$, then for any $N > 0$, we have that 
\begin{equation}\label{bound on D}
N^{-2} D(|u|, N)\lesssim 1
\end{equation}
We will first compare the discrepancy between $S(u,N)$ and $D(u,N)$. 

Notice that every $z \in \Z^2_N$ is a visible point for exactly one lattice in the family $\{\Z^2_{\frac{N}{m}}\}_{m \in \N}$. As a result, we have
$$
\sum_{z\in \Z^2_N} u(z) =\sum_{m=1}^\infty \sum_{\substack{z\in \Z^2_{\frac{N}m} \\visible}} u(z)
$$
and hence for all $N > 0$, 
\begin{equation*}
D(u,N)=\sum_{m=1}^\infty S(u,\frac{N}{m}).
\end{equation*}
Consequently, applying Lemma \ref{Mobinv2} to the functions $f(y) = D(u,\frac{N}y)$ and $g(y) = S(u,\frac{N}{y})$ at $x=1$ we obtain for all $N > 0$, 
\begin{equation}\label{S to D}
S(u,N)=\sum_{n=1}^\infty \mu(n)D(u, \frac{N}{n}), 
\end{equation}
provided that $\sum_{m,n \in \N}|S(u,\frac{N}{mn})| <\infty$. But this follows from \eqref{bound on D} as soon as $u \in X^\kappa(\R^2)$ with $\kappa > 2$, and the fact that $S(u, \sigma) \leq D(|u|, \sigma)$ for any $\sigma>0$.

With these notations, the error term in the estimate \eqref{equi} can be written as follows
\begin{equation*}
\mathcal E_N(K)\overset{def}{=}\frac{1}{N^2}\sum_{z\in \Z^2_N\, visible} u_{K}(z)  
-\frac{1}{\zeta(2)} \int_{\R^2}  u_K(z) dz=\frac{1}{N^2} S(u_K, N) -\frac{1}{\zeta(2)} \int_{\R^2}  u_K(z) dz,\end{equation*}
where 
\begin{equation}
\label{eq:uK}
u_K(z)=g(K+\frac{\alpha}{\beta} z)\bar g(K+\frac{\alpha}{\beta} z+z^\perp)g(K+ z^\perp).
\end{equation}
 Now by \eqref{S to D}
\begin{align*}
\frac{1}{N^2}S(u_K, N) &=\sum_{n=1}^\infty \frac{\mu(n)}{n^2}\Big(\frac{n^2}{N^2}D(u_K,\frac{N}{n})\Big)\\
&=\sum_{n=1}^\infty \frac{\mu(n)}{n^2}\int_{\R^2} u_K(z) \,dz +\sum_{n=1}^\infty\frac{\mu(n)}{n^2}\left( \frac{n^2}{N^2}D(u_K,\frac{N}{n})-\int_{\R^2} u_K(z) \,dz\right)\\
&=\frac{1}{\zeta(2)}\int_{\R^2} u_K(z) \,dz +\sum_{n=1}^\infty\frac{\mu(n)}{n^2}\left( \frac{n^2}{N^2}D(u_K,\frac{N}{n}) -\int_{\R^2} u_K(z) \,dz\right).\\
\end{align*}

As a result, we have that 
\begin{align*}
\rho^2S(u_K, \rho) -\frac{1}{\zeta(2)}\int_{\R^2} u_K(z) \,dz &=\sum_{n=1}^\infty\frac{\mu(n)}{n^2}\left( \frac{n^2}{N^2}D(u_K,\frac{N}{n})-\int_{\R^2} u_K(z) \,dz\right)\\
&\overset{def}{=}\mathfrak D_1+\mathfrak D_2.
\end{align*}
with
\begin{align*}
&\mathfrak D_1 =\sum_{n=1}^{[N]}\frac{\mu(n)}{n^2}\left( \frac{n^2}{N^2}D(u_K,\frac{N}{n}) -\int_{\R^2} u_K(z) \,dz\right)\quad \mbox{and}\\
&\mathfrak D_2 = \sum_{n=[N]+1}^\infty \frac{\mu(n)}{n^2}\left( \frac{n^2}{N^2}D(u_K,\frac{N}{n}) -\int_{\R^2} u_K(z) \,dz\right).\\
\end{align*}

\medskip
\noindent
\underline{Step 2: estimates of $\mathfrak D_1$ and $\mathfrak D_2$.}
The estimate of $\mathfrak D_1$ and $\mathfrak D_2$ relies on Lemma \ref{discrep lemma} below, whose proof will be given in the next paragraph. 

For any given $R > 0$ and $K \in \R^2$, there exists a unique $r \in \mathcal{B}(\frac{1}{R})$ and 
$Q \in \Z_R^2$ such that $K = r + Q$.
Define 
\begin{equation}\label{def of D*}
D^*(u_K, R)\overset{def}{=}\sum_{M \in (\Z_R^2)^* \setminus \{Q^\perp\}} u_K(M).
\end{equation}
Note that $D^*(u_K, R)=D(u_K, R)$ if $Q=0$. In particular, if
\begin{equation}
\label{tgv}
 |K| \leq  \frac{1}{2R} \quad \Longrightarrow \quad D^*(u_K, R)=D(u_K, R). 
\end{equation}
\begin{lemma}\label{discrep lemma}
With the previous notations, we have 
\begin{enumerate}
\item If $R \leq 1$, then 
\begin{equation}\label{aimequi1}
\|R^{-2}D^*(u_K, R) -\int_{\R^2}u_K(z) \, dz\|_{X^\sigma(\R^2)}\lesssim 1.
\end{equation}
\item If $R\geq 1$, then
\begin{equation}\label{aimequi2}
\|R^{-2} D(u_K, R) -\int_{\R^2}u_K(z) \, dz\|_{X^\sigma(\R^2)}\lesssim \frac{1}{R}.
\end{equation}
\end{enumerate}
\end{lemma}

To estimate $\mathfrak D_1$, we use the second part this Lemma (with $R = N/n \geq 1$) to get that
\begin{align*}
\|\mathfrak D_1\|_{X^\sigma(\R^2)} \lesssim \sum_{n=1}^{[N]}\frac{1}{n^2} \frac{n}{N}\lesssim \frac{1 + \log N}{N}\end{align*}
which is acceptable. 

To estimate $\mathfrak D_2$, we must isolate from the sum defining $D(u_K, \frac{N}{n})$ the problematic point $Q^\perp$ (depending on $K$ and $R=\frac{N}{n}$) and define $D^*(u_K, \frac{N}{n})$ as in \eqref{def of D*}. Using \eqref{tgv} with $R = N/n \leq 1$, we can write 
\begin{align*}
\mathfrak D_2=&\sum_{n=[N]+1}^\infty \frac{\mu(n)}{n^2}\left( \frac{n^2}{N^2} D^*(u_K,\frac{N}{n}) -\int_{\R^2} u_K(z) \,dz\right)+\frac{1}{N^2}\sum_{n=[N]+1}^{2[N|K|]+1}\mu(n) u_K(Q^\perp)\\
\overset{def}{=} &\mathfrak D_2^{(1)}+\mathfrak D_2^{(2)}, 
\end{align*}

The estimate on $\mathfrak D_2^{(1)}$ follows from the first part of Lemma \ref{discrep lemma} below, which gives that
$$
\|\mathfrak D_2^{(1)}\|_{X^\sigma(\R^2)}\lesssim \sum_{n=[N]+1}^\infty \frac{1}{n^2}\lesssim \frac{1}{N}.
$$

To estimate the contribution of $\mathfrak D_2^{(2)}$, we first note that
$$
u_K(Q^\perp)=g(r+\frac{\alpha}{\beta}Q^\perp+Q)\overline g(r+\frac{\alpha}{\beta}Q^\perp)g(r)
$$ 
and hence due to the orthogonality relation
$$
|K|^2+|r+\frac{\alpha}{\beta}Q^\perp|^2=|r|^2+|r+\frac{\alpha}{\beta}Q^\perp+Q|^2
$$
we get that $|r+\frac{\alpha}{\beta}Q^\perp+Q|\geq \frac{|K|}{2}$ or $|r| \geq \frac{|K|}{2}$. 
Using the fact that $g\in X^{\sigma+1}$, we can estimate
$$
|\mathfrak D_2^{(2)}|\lesssim \frac{1}{N^2}\sum_{n=[N]+1}^{2[N|K|]+1}\langle K \rangle^{-\sigma-1}\lesssim \frac{1}{N^2}\frac{N|K|}{\langle K \rangle^{\sigma+1}}\lesssim \frac{1}{N \langle K \rangle^{\sigma}},
$$
and consequently we have that 
$$
\|\mathfrak D_2^{(2)}\|_{X^\sigma(\R^2)}\lesssim \frac{1}{N}
$$
as needed. This finishes the proof of Lemma \ref{equidistribution lemma}, modulo Lemma \ref{discrep lemma} that we now prove.

\medskip
\noindent
\underline{Step 3: Proof of Lemma \ref{discrep lemma}.}
We start with part (1). In this case $R \leq 1$ and we will show that \begin{equation}\label{disc0.1}
\|R^{-2} D^*(u_K, R)\|_{X^\sigma(\R^2)}+ \|\int_{\R^2}u_K(z)\|_{X^\sigma(\R^2)} \lesssim 1. 
\end{equation}
Recall that we have written $K=r+Q$ with $Q \in \Z^2_R$ and $r\in \mathcal B(\frac{1}{R})$, and defined $D^*(u_K, R)$ so that:
\begin{align}
\frac{1}{R^2}D^*(u_K, R)=&\frac{1}{R^2}\sum_{M \in (\Z^2_R)^*\setminus \{Q^\perp\}}g(K +\frac{\alpha}{\beta} M)\overline{g}(K+\frac{\alpha}{\beta} M +M^\perp) g(K+M^\perp).\label{sumJnotM}
\end{align}

The bound on \eqref{sumJnotM} follows by splitting into two cases: either $|K+\frac{\alpha}{\beta} M|\geq \frac{|K|}{100}$ or $|K+\frac{\alpha}{\beta}M|\leq \frac{|K|}{100}$. The contribution of the first case ($|K+\frac{\alpha}{\beta} M|\geq \frac{|K|}{100}$) to the $X^\sigma$ norm of \eqref{sumJnotM} can be bounded by 
\begin{multline*}
\sup_{K \in \R^2} \frac{1}{R^2} \sum_{M \in \Z^2_R\setminus \{Q^\perp\}}\langle K+\frac{\alpha}{\beta} M +M^\perp\rangle^{-\sigma}\langle K+M^\perp\rangle^{-\sigma} \\
\lesssim \sup_{K \in \R^2} \frac{1}{R^2} \sum_{M \in \Z^2_R\setminus \{Q^\perp\}}\langle K+M^\perp\rangle^{-\sigma}\lesssim 1,
\end{multline*}
where we have used Lemma \ref{elementary} (in the case $\alpha = 0$, $\beta = 1$ and $L \leq 1$) and the fact that $M \neq Q^\perp$ in the last step. 

In the case when $|K+\frac{\alpha}{\beta} M|\leq \frac{|K|}{100}$, we can assume that $|K|\geq 100$ since otherwise the same argument as in the case $|K+\frac{\alpha}{\beta} M| > \frac{|K|}{100}$ applies. But $|K+\frac{\alpha}{\beta}M|\leq \frac{K}{100}$ implies that $|M|\geq |\frac{\alpha}{\beta}M|\geq \frac{99}{100}|K|$ (recall that $\alpha \leq \beta$ in \eqref{eq:uK}). This also implies that $|K+\frac{\alpha}{\beta}M +M^\perp|\geq |M^\perp|-|K+\frac{\alpha}{\beta} M|\geq \frac{98}{100}|K|$ and consequently we have that the contribution of the second case to the $X^\sigma$ norm of \eqref{sumJnotM} can be bounded by
$$
\sup_{K \in \R^2} \frac{1}{R^2}\sum_{M \in \Z^2_R\setminus \{Q^\perp\}}\langle K+\frac{\alpha}{\beta} M \rangle^{-\sigma}\langle K+M^\perp\rangle^{-\sigma}\lesssim\sup_{K \in \R^2} \frac{1}{R^2} \sum_{M \in \Z^2_R\setminus \{Q^\perp\}}\langle K+M^\perp\rangle^{-\sigma}\lesssim 1,
$$
where we used again Lemma \ref{elementary} and the fact that $M\neq Q^\perp$. In conclusion, we have that
$$
\|\eqref{sumJnotM}\|_{X^\sigma}\lesssim 1.
$$
As a result,
\begin{equation}\label{boundsigmaL1}
\|R^{-2} D^*(u_K, R)\|_{X^\sigma}\lesssim 1
\end{equation}
The estimate on $\int_{\R^2}u_K(z)dz$ is similar to (even simpler than) that performed above for the discrete sum, and in view of \eqref{eq:uK}, one gets the bound 
$$
\left\|\int_{\R^2}u_K(z)\,dz \right\|_{X^\sigma}\lesssim 1,
$$
which finishes the proof of \eqref{aimequi1}.

\medskip

Now we move on to proving Part (2). Here $R \geq 1$. Let $\mathcal B_M(\frac{1}{R})$ denote the box $M+[0, \frac1R)^2$. Then
\begin{align}
&\frac{1}{R^2} D(u_K, R) -\int_{\R^2}u_K(z) \, dz=\sum_{M \in \Z^2_R} \int_{\mathcal B_M(\frac{1}{R})}[g(K+\frac{\alpha}{\beta}M) \bar g(K+\frac{\alpha}{\beta}M+M^\perp)g(K+M^\perp)\nonumber\\
&-g(K+\frac{\alpha}{\beta}z) \bar g(K+\frac{\alpha}{\beta}z+z^\perp)g(K+z^\perp)]\,dz\nonumber\\
=& \sum_{M\in \Z_R^2}\int_{\mathcal B_M(\frac{1}{R})}\left[g(K+\frac{\alpha}{\beta}M) -g(K+\frac{\alpha}{\beta}z)\right]\bar g(K+\frac{\alpha}{\beta}M+M^\perp)g(K+M^\perp)dz \label{eqeq0.01}\\
&+\sum_{M\in \Z_R^2} \int_{\mathcal B_M(\frac{1}{R})}g(K+\frac{\alpha}{\beta}z)\left[ \bar g(K+\frac{\alpha}{\beta}M+M^\perp)-\bar g(K+\frac{\alpha}{\beta}z+z^\perp)\right]g(K+M^\perp)\,dz\label{eqeq0.02}\\
&+\sum_{M \in \Z_R^2} \int_{\mathcal B_M(\frac{1}{R})}g(K+\frac{\alpha}{\beta}z) \bar g(K+\frac{\alpha}{\beta}z+z^\perp)\left[ g(K+M^\perp)- g(K+z^\perp)\right] \,dz. \label{eqeq0.03}
\end{align}
We start by bounding \eqref{eqeq0.01}. For this we write
\begin{align*}
&|\eqref{eqeq0.01}|=\\
&\left|\sum_{M \in\Z_R^2} \int_{\mathcal B_M(\frac{1}{R})} \int_0^1 \frac{\alpha}{\beta}(M-z).\nabla g(K+\frac{\alpha}{\beta}z+\frac{\alpha}{\beta}t(M-z))\bar g(K+\frac{\alpha}{\beta}M+M^\perp)g(K+M^\perp) dt\,dz \right|\\
\leq& \frac1R \sum_{M \in \Z_R^2} \int_{\mathcal B_0(\frac{1}{R})} \int_0^1|\nabla g|(K+\frac{\alpha}{\beta}M+\frac{\alpha}{\beta}z-\frac{\alpha}{\beta}tz)|g|(K+\frac{\alpha}{\beta}M+M^\perp)|g|(K+M^\perp) dt\,dz \\
=& \frac1{R^3} \fint_{\mathcal B_0(\frac{1}{R})} \int_0^1\sum_{M \in \Z_R^2} |\nabla g|(K+\frac{\alpha}{\beta}M+\frac{\alpha}{\beta}z-\frac{\alpha}{\beta}tz)|g|(K+\frac{\alpha}{\beta}M+M^\perp)|g|(K+M^\perp) dt\,dz ,\\
\end{align*}
where we denoted by $\fint_{\mathcal B_0(\frac1R)}=R^2\int_{\mathcal B_0(\frac1R)}$. To prove that $\|\eqref{eqeq0.01}\|\lesssim \frac1R$, it suffices -given our bounds on $g$ and $\nabla g$ in $X^\sigma$- to prove that for any $0\leq t \leq 1$ and $z \in \mathcal B_0(\frac1{R})$
\begin{equation}\label{eqeq0.011}
\frac1{R^2} \sum_{M \in \Z_R^2}\langle K+\frac{\alpha}{\beta}M+\frac{\alpha}{\beta}(1-t)z\rangle^{-\sigma}\langle K+\frac{\alpha}{\beta}M+M^\perp\rangle^{-\sigma}\langle K +M^\perp \rangle^{-\sigma}\lesssim \langle K \rangle^{-\sigma}.
\end{equation}
But this follows by repeating the by-now-familiar argument: If $|K|\leq 200$ or if $|K+\frac{\alpha}{\beta}M|\geq \frac{|K|}{100}$, then since $(1 -t)z \in \mathcal B_0(\frac{1}{R})$, $\langle K+\frac{\alpha}{\beta}M+\frac{\alpha}{\beta}z-\frac{\alpha}{\beta}tz\rangle \gtrsim \langle K \rangle$; and \eqref{eqeq0.011} follows from Lemma \ref{elementary} (with $\beta = 1$ and $L \geq 1$) since
$$
\frac{1}{R^2} \sum_{M \in \Z^2_R}\langle K +M^\perp \rangle^{-\sigma}\lesssim 1.
$$
Otherwise, we have that $|K|\geq 200$ and $|K+\frac{\alpha}{\beta}M|\leq \frac{|K|}{100}$. In this case, \eqref{eqeq0.011} follows again from Lemma \ref{elementary} since $|K+\frac{\alpha}{\beta}M + M^\perp |\gtrsim |K|$ (using the orthogonality relation).

The argument for \eqref{eqeq0.02} and \eqref{eqeq0.03} are similar  upon using again Lemma \eqref{elementary} and the bounds on $g$ and $\nabla g$. 
This finishes the proof of Lemma \ref{discrep lemma}. 

\begin{remark}
A direct consequence of the analysis performed in the previous section is the following co-prime equidistribution result, which we state for functions $u$ satisfying $(|u|, |\nabla u|)\in X^\sigma(\R^2)$, but can be extended easily to much larger classes of functions.
\begin{corollary}\label{discrep corollary}
Suppose that $u:\R^2\to \C$ satisfies $\|u\|_{X^\sigma}+\|\nabla u\|_{X^\sigma}\leq B$ for $\sigma > 2$. Then for any $L\geq 1$, the following discrepancy bound holds:
\begin{equation}\label{discrep corollary est}
\left| L^{-2}\sum_{\substack{J \in \Z_L^2\\visible}} u(J) -\frac{1}{\zeta(2)}\int_{\R^2} u(z)\, dz\right| \lesssim B \frac{1 + \log L}{L}. 
\end{equation}
\end{corollary}
\end{remark}

\section{Proof of Approximation Theorems: convergence of (NLS) to (CR)}\label{proof of approx theorem section}

The purpose of this section is to prove Theorems \ref{approx thm1} and \ref{result on T^2}. We will start by adopting a notationally convenient rescaling. Recall that $\widetilde a_K(t)=e^{-4\pi^2i|K|^2t}a_K(t)$ satisfies the equation
\begin{equation*}
-i\partial_t \widetilde a_K(t)=\frac{\epsilon^2}{L^4}\sum_{\mathcal S(K)} \widetilde a_{K_1}(t)\overline{\widetilde a_{K_2}}(t)\widetilde a_{K_3}(t)e^{4\pi^2i\Omega t}.
\end{equation*}
where $\Omega(K,K_1,K_2,K_3) = -|K|^2 + |K_1|^2 -|K_2|^2 + |K_3|^2$. Set $\displaystyle C_0 = \frac{\zeta(2)}{2}$, and define the renormalized profile $b_K$ as
\begin{equation}
\label{eq:bK}
b_K(t) = \, \widetilde a_K(\frac{C_0 L^2}{\log L}t)=a_K(\frac{C_0 L^2}{\log L}t) e^{-i\frac{4\pi^2 C_0 L^2}{\log L}|K|^2t}. 
\end{equation}
It will be much more notationally convenient to work with $b_K(t)$ instead of $a_K(t)$. It satisfies
$$
- i \partial_t b_K = \frac{C_0 \epsilon^2}{L^2 \log L} \sum_{\mathcal{S}(K)} b_{K_1}(t) \overline{b_{K_2}(t)} b_{K_3}(t) e^{i \frac{4\pi^2C_0 L^2}{\log L}\Omega(K,K_1,K_2,K_3)t},
$$
Recall the definition of the operator $\mathcal T_L$ as
\begin{equation*}
\mathcal{T}_L (e,f,g)(K) \overset{def}{=} \frac{C_0}{L^2 \log L} \sum_{\mathcal{R}(K)} e_{K_1} f_{K_2} g_{K_3},
\end{equation*}
for sequences $e = \{e_K\}$, $e = \{f_K\}$ and $g = \{g_K\}$ on $\Z^2_L$. 
Hence  the equation can be written
\begin{equation}
\label{bassan}
- i \partial_t b_K = \underbrace{\epsilon^2 \mathcal{T}_L (b,b,b)(K)}_{\mbox{resonant interactions}} 
+ \underbrace{\frac{\epsilon^2 C_0}{L^2 \log L} \sum_{\mathcal{S}(K) \setminus \mathcal{R}(K)} b_{K_1}(t) \overline{b_{K_2}(t)} b_{K_3}(t) e^{i\frac{4\pi^2C_0 L^2}{\log L}\Omega t}}_{\mbox{non-resonant interactions}}.
\end{equation}

In sequel, the initial value $b_K(0)$ will be the projection of a function $g_0: \R^2 \to \C$, that is $b_K(0) = g_0(K)$ for $K \in \Z_L^2$. With this rescaling, Theorem \ref{approx thm1} follows once we show that for any $0<\gamma<1$, there exists $c_\gamma$ such that if $\epsilon< c_\gamma L^{-1-\gamma} B^{-1}$, it holds that
\begin{equation}\label{aim5}
\|b_K(t)-g(\epsilon^2 t, K)\|_{X^\sigma_L}\lesssim \frac{B}{(\log L)^{1-\gamma}}
\end{equation}
for all $0\leq t \leq \epsilon^{-2}\min(M, \gamma \frac{\log \log L}{c_\gamma B^2})$. If we require from the start that $L\geq \exp \exp \gamma^{-1} c_\gamma M B^2$, Theorem \ref{approx thm1} follows as stated. 

Before proving Theorem~\ref{approx thm1}, we start with the following elementary lemma:
\begin{lemma}\label{crude multiestimates}
Let $\sigma > 2$. For any sequences $\{a_K\}, \{b_K\} $, and $\{c_K\}$ in $X^\sigma_L$, we have 
\begin{equation}
\left\|\sum_{K_1-K_2 +K_3=K} a_{K_1}b_{K_2} c_{K_3}\right\|_{X_L^\sigma} \lesssim  L^4 \|a_K\|_{X^\sigma_L}\|b_K\|_{X^\sigma_L}\|c_K\|_{X^\sigma_L}
\end{equation}
\end{lemma}

\begin{proof}
The proof is straightforward. Without loss of generality, we assume that the sequences $\{a_K\}$, $\{b_K\}$,  and $\{c_K\}$ are non-negative. Then we have for $K \in \Z_L^2$, 
\begin{align*}
\langle K \rangle^{\sigma}\sum_{K_1-K_2+K_3=K} a_{K_1}b_{K_2}c_{K_3}\leq& \sum_{K_1-K_2+K_3=K} \langle K_1 \rangle^{\sigma}a_{K_1}b_{K_2}c_{K_3}+\sum_{K_1-K_2+K_3=K} a_{K_1}\langle K_2 \rangle^{\sigma}b_{K_2}c_{K_3}\\
&+\sum_{K_1-K_2+K_3=K} a_{K_1}b_{K_2}\langle K_3 \rangle^{\sigma}c_{K_3}.
\end{align*}
We bound the contribution of the first terms; the other being similar.
\begin{align*}
 \sum_{K_1-K_2+K_3=K} \langle K_1 \rangle^{\sigma}a_{K_1}b_{K_2}c_{K_2}\leq \|a_K\|_{X^\sigma_L} \sum_{K_2, K_3}b_{K_2}c_{K_2}\leq& L^4 \|a_K\|_{X^\sigma}\|b_{K}\|_{\ell^1_L}\|c_K\|_{\ell^1_L}\\
 \lesssim& L^4 \|a_K\|_{X^\sigma_L}\|b_{K}\|_{X^\sigma_L}\|c_K\|_{X^\sigma_L}.
\end{align*}
\end{proof}

\begin{proof}[Proof of Theorem \ref{approx thm1}]
Let $g^\epsilon(t,\xi)\overset{def}{=}g(\epsilon^2 t, \xi)$. It satisfies the equation
\begin{equation}\label{g^epsilon}
-i\partial_t g^\epsilon(t,\xi)=\epsilon^2 \TT(g^\epsilon,g^\epsilon,g^\epsilon)(t,\xi).
\end{equation}
Note that $b_K(0) =  g_0(K) \in X_L^{\sigma + 1}$. 
From \eqref{bassan}, the sequence  $\{b_K\}$ satisfies the equation
\begin{multline*}
-i\partial_t b_K= \epsilon^2 \mathcal T_L(g^\epsilon,g^\epsilon,g^\epsilon)(K)+\epsilon^2 \left[ \mathcal T_L(b,b,b)(K) - \mathcal T_L(g^\epsilon,g^\epsilon,g^\epsilon)(K)\right] \\
+\frac{\epsilon^2}{L^2\log L}\sum_{\mathcal S(K)\setminus \RR(K)} b_{K_1}\overline{b_{K_2}}b_{K_3}e^{i \frac{4\pi^2C_0 L^2\Omega}{\log L}t}.
\end{multline*}
Let $w_K(t) \overset{def}{=} b_K(t)-g(\epsilon^2 t, K)$ for $K \in \Z_L^2$ and $t \leq \varepsilon^{-2}M$. Then the equation satisfied by $w_K$ is
\begin{align*}
-i\partial_t w_K(t)=&\, \epsilon^2 \left[\mathcal \mathcal T_L(b,b,b)(K) - T_L(g^\epsilon,g^\epsilon,g^\epsilon)(K)\right]+\epsilon^2 \left[\mathcal T_L(g^\epsilon,g^\epsilon,g^\epsilon)(K)-\TT(g^\epsilon,g^\epsilon,g^\epsilon)(K)\right]\\
&+\frac{C_0 \epsilon^2}{L^2\log L}\sum_{\mathcal S(K)\setminus \RR(K)} b_{K_1}\overline{b_{K_2}}b_{K_3}e^{i \frac{4\pi^2C_0 L^2\Omega}{\log L}t}.
\end{align*}

Our first step is an application of a normal form transformation: For this we write
\begin{align*}
\frac{C_0 \epsilon^2}{L^2\log L}\sum_{\mathcal S (K)\setminus \RR (K)} b_{K_1}\overline{b_{K_2}}b_{K_3}e^{i \frac{4\pi^2C_0 L^2\Omega}{\log L}t}=&\partial_t \left(\frac{\epsilon^2}{L^2}\sum_{\mathcal S(K)\setminus \RR(K)}\frac{1}{4\pi^2i L^2 \Omega} b_{K_1}\overline{b_{K_2}}b_{K_3}e^{i \frac{ 4\pi^2C_0 L^2\Omega}{\log L}t}\right)\\
&-\frac{\epsilon^2}{L^2}\sum_{\mathcal S(K)\setminus \RR(K)}\frac{1}{4\pi^2iL^2 \Omega} \partial_t \left(b_{K_1}\overline{b_{K_2}}b_{K_3}\right)e^{i \frac{ 4\pi^2C_0 L^2\Omega}{\log L}t}. 
\end{align*}
As a result, if we define the change of coordinates
$$
e_K(t)\overset{def}{=}w_K(t) - \frac{\epsilon^2}{L^2}\sum_{\mathcal S(K)\setminus \RR(K)}\frac{1}{4\pi^2L^2 \Omega} b_{K_1}(t)\overline{b_{K_2}}(t)b_{K_3}(t)e^{i \frac{4\pi^2C_0 L^2\Omega}{\log L}t},
$$
then the equation satisfied by $e_K(t)$ is 
\begin{align}
-i\partial_t e_K(t)=&\epsilon^2 \left[\mathcal T_L(b,b,b)(K) - \mathcal T_L(g^\epsilon,g^\epsilon,g^\epsilon)(K)\right]+\epsilon^2 \left[ \mathcal T_L(g^\epsilon,g^\epsilon,g^\epsilon)(K)-\TT(g^\epsilon,g^\epsilon,g^\epsilon)(K)\right]\label{line1 eKeqn}\\
&-\frac{\epsilon^2}{L^2}\sum_{\mathcal S(K)\setminus \RR(K)} \frac{1}{4\pi^2i L^2\Omega}\partial_t\left(b_{K_1}\overline{b_{K_2}}b_{K_3}\right)e^{i \frac{4\pi^2C_0 L^2\Omega}{\log L}t},\label{line2 eKeqn}\\
e_K(0)=&-\frac{\epsilon^2}{L^2}\sum_{\mathcal S(K)\setminus \RR(K)}\frac{1}{ 4\pi^2L^2 \Omega} b_{K_1}(0)\overline{b_{K_2}(0)}b_{K_3}(0)
\end{align}
Of course, the advantage of working with $e_K$ as opposed to $b_K$ is the fact that \eqref{line2 eKeqn} is quintic in $b_K$.

The proof follows a bootstrap argument. Recall that $\|g^\epsilon\|_{X^\sigma} \leq B$ for all $0\leq t \leq \epsilon^{-2}M$. We will show that we can find a constant $c=c_\gamma$, such that if 
$$
\epsilon< c_\gamma L^{-1-\gamma}B^{-1} \quad \hbox{ and }\sup_{t\in [0,T]}\|b_K(t)\|_{X^\sigma_L}\leq 2B
$$
for some $T\leq \epsilon^{-2}\min(M, c B^{-2}\log \log L)$, then \eqref{aim5} holds on the interval $[0,T]$. The result then follows on the full interval $[0,\epsilon^{-2}\min(M, c B^{-2}\log \log L)]$ via a continuity argument. 

We start by proving that our normal form transformation from $w_K\to e_K$ is close to the identity. Indeed, recalling the definition of \eqref{def of R_mu} of the set $\mathcal R_\mu$ for $\mu \in \R$, we can write
\begin{align*}
\|w_K(t)-e_K(t)\|_{X^\sigma_L}\lesssim & \frac{\epsilon^2}{L^2}\left\|\sum_{m\in \Z\setminus\{0\}}m^{-1}\sum_{\RR_{mL^{-2}}(K)}b_{K_1}\overline{b_{K_2}}b_{K_3}e^{i \frac{4\pi^2C_0 L^2\Omega}{\log L}t}\right\|_{X^\sigma_L}\\
\leq&\frac{\epsilon^2}{L^2}\left\|\sum_{0<|m|< L^{10}}m^{-1}\sum_{\RR_{mL^{-2}}(K)}b_{K_1}\overline{b_{K_2}}b_{K_3}e^{i \frac{4\pi^2C_0 L^2\Omega}{\log L}t}\right\|_{X^\sigma_L}\\
&+\frac{\epsilon^2}{L^2}\left\|\sum_{|m|\geq L^{10}}m^{-1}\sum_{\RR_{mL^{-2}}(K)}b_{K_1}\overline{b_{K_2}}b_{K_3}e^{i \frac{4\pi^2C_0 L^2\Omega}{\log L}t}\right\|_{X^\sigma_L}\\
\leq&\frac{\epsilon^2}{L^2}\sum_{0<|m|< L^{10}}|m|^{-1}\left\|\sum_{\RR_{mL^{-2}}(K)}b_{K_1}\overline{b_{K_2}}b_{K_3}e^{i \frac{4\pi^2C_0 L^2\Omega}{\log L}t}\right\|_{X^\sigma_L}\\
&+\frac{\epsilon^2}{L^2}L^{-10}\left\|\sum_{\mathcal S(K)}|b_{K_1}||b_{K_2}||b_{K_3}|\right\|_{X^\sigma_L}\\
\lesssim_\gamma& \epsilon^2 L^\gamma B^3
\end{align*}
for any $\gamma>0$, where we have used the bootstrap assumption that $\|b_K\|_{X^\sigma_L}\leq 2B$, Proposition \ref{R_mu sum prop} to bound the first sum, and Lemma \ref{crude multiestimates} to bound the second sum. 
As a result,
\begin{equation}\label{w_K and e_K}
\|w_K(t)-e_K(t)\|_{X^\sigma}\lesssim_\gamma \epsilon^2L^\gamma B^3.
\end{equation}

We start by estimating the terms on the right-hand-side of \eqref{line1 eKeqn}. We start with the first term:
\begin{align*}
\epsilon^2\left\| \mathcal T_L(b,b,b)(K) - \mathcal T_L(g^\epsilon,g^\epsilon,g^\epsilon)(K)\right\|_{X^\sigma}\lesssim \epsilon^2 B^2 \|w_K\|_{X^\sigma_L}\lesssim \epsilon^2 B^2 \|e_K\|_{X^\sigma_L}+c_\gamma \epsilon^4L^\gamma B^5.
\end{align*}
where we used Theorem \ref{tri est} and the trilinearity of $\mathcal T_L$ in the first step. The second term on the right-hand-side of \eqref{line1 eKeqn} can be estimated using Theorem~\ref{disctocont}, which gives
\begin{align*}
\epsilon^2 \left\|T_L(g^\epsilon,g^\epsilon,g^\epsilon)(K)-\TT(g^\epsilon,g^\epsilon,g^\epsilon)(K)\right\|_{X^\sigma}\lesssim \frac{\epsilon^{2}B^3}{\log L}.
\end{align*}
Now notice that \eqref{line2 eKeqn} can be written as:
\begin{align*}
\eqref{line2 eKeqn}=&- \frac{\epsilon^2}{4\pi^2 L^2}\sum_{m \in \Z\setminus\{0\}}(im)^{-1}\sum_{\RR_{m L^{-2}}(K)}(\partial_t b_{K_1})\overline{b_{K_2}}b_{K_3}e^{4\pi^2i \frac{L^2\Omega}{\log L}t}\\
&+\text{two similar terms}. 
\end{align*}
As a result, we can estimate as before using Proposition \ref{R_mu sum prop} and Lemma \ref{crude multiestimates}
\begin{align*}
\|\eqref{line2 eKeqn}\|_{X^\sigma_L}\lesssim& \frac{\epsilon^2}{L^2}\sum_{0<|m|<L^{10}}|m|^{-1}\left\|\sum_{\RR_{m L^{-2}}(K)}\partial_t b_{K_1}\overline{b_{K_2}}b_{K_3}e^{i \frac{4\pi^2C_0 L^2\Omega}{\log L}t}\right\|_{X^\sigma_L}+\frac{\epsilon^2}{L^2}L^{-10}\left\|\sum_{\mathcal S(K)}|\partial_t b_{K_1}||b_{K_2}||b_{K_3}|\right\|_{X^\sigma_L}\\
\lesssim_\gamma& \frac{\epsilon^2}{L^2}L^{2+\gamma} \|\partial_t b_K\|_{X^\sigma_L}\|b_K\|_{X^\sigma_L}^2. 
\end{align*}
Using once again Lemma \ref{crude multiestimates} gives
$$
\|\eqref{line2 eKeqn}\|_{X^\sigma_L} \lesssim _\gamma \epsilon^4L^{2+\gamma} \|b_K\|_{X^\sigma_L}^5\leq c_\gamma \epsilon^4L^{2+\gamma} B^5.
$$
Collecting all the above estimates, we get that $\|e_K\|_{X^\sigma_L}$ satisfies the following differential inequality: 
$$
\|\partial_t e_K(t) \|_{X^\sigma(\Z^2_L)}\leq  c\epsilon^2 B^2 \|e_K\|_{X^\sigma_L}+c\frac{\epsilon^{2}B^3}{\log L}+c_\gamma \epsilon^4L^{2+\gamma} B^5\quad\mbox{and}\quad
\|e_K(0)\|_{X^\sigma_L}\leq \, c_\gamma\epsilon^2 L^\gamma B^3, 
$$
for some universal constant $c$ and a constant $c_\gamma$ depending on $\gamma$.
An application of Gronwall's lemma then implies that, if $0\leq t \leq T\leq \epsilon^{-2} M$, we have 
\begin{align*}
\|e_K(t)\|_{X^\sigma_L}\leq \|e_K(0)\|_{X^\sigma_L}e^{c\epsilon^2B^2 t} +(e^{c\epsilon^2B^2 t}-1)\left(\frac{B}{\log L}+\frac{c_\gamma}{c}\epsilon^2 L^{2+\gamma} B^3\right).
\end{align*}
As a consequence, if $\epsilon< c_\gamma' L^{-1-\gamma} B^{-1}$, then it holds that 
$$
\|e_K(t)\|_{X^\sigma_L}\leq e^{c\epsilon^2 B^2 t} \frac{2B}{\log L}.
$$
Using in addition $t\leq T\leq \gamma \frac{\log \log L}{c\epsilon^2 B^2}$, we get that
$$
\|e_K(t)\|_{X^\sigma_L} \leq \frac{2B}{(\log L)^{1-\gamma}}.
$$
In conclusion, combining the above estimate with \eqref{w_K and e_K} we get
$$
\|b_K - g^\epsilon\|_{X^\sigma_L} = \|w_K\|_{X^\sigma_L} \leq \|e_K\|_{X^\sigma_L} + \|w_K - e_K\|_{X^\sigma_L} 
\leq \frac{B}{L^{2+\gamma}} + \frac{2B}{(\log L)^{1-\gamma}}.
$$
This completes the bootstrap argument, and the proof.
\end{proof}

Finally, we give the simple proof of how to deduce Corollary \ref{result on T^2} from the above theorem.

\begin{proof}[Proof of Corollary \ref{result on T^2}]
Shrink $\gamma$ if needed so that $\gamma \leq \frac{s-1}{100}$. Define $(\widetilde a_K(t))_{K \in \Z^2_L}$ such that $e^{-4\pi^2it|k|^2}\widehat v(t,k)=\frac{\epsilon}{N}\widetilde a_{\frac{k}{N}} (N^2t)$ for $k \in \Z$. Then $\widetilde a_K(t)$ satisfies \eqref{FNLSe}. 

Choose $\epsilon=N^{-s}<c_\gamma N^{-1-\gamma}B^{-1}$ if $N \geq CB^{\frac{s-1}{2}}$ thanks to our choice $\gamma$, and apply Theorem \ref{approx thm1} for $\widetilde a_K$ to get for $T^*=\frac{\zeta(2) N^2}{2\epsilon^2 \log N}$
\begin{equation*}
\|\widetilde a_K(t)-g(\frac{t}{T^*},K)\|_{X_L^\sigma(\Z_N^2)}\lesssim \frac{B}{(\log N)^{1-\gamma}}
\end{equation*}
for all $0\leq t \leq T^*\min\left(M, \frac{\gamma \log\log N}{c B^2}\right)$. This gives that
\begin{equation*}
\|e^{-4\pi^2itN^2|K|^2}\widehat v(t,NK)-\frac{\epsilon}{N}g(\frac{N^2t}{T^*},K)\|_{X_N^\sigma(\Z_N^2)}\lesssim \frac{\epsilon B}{N(\log N)^{1-\gamma}}
\end{equation*}
for all $0\leq t \leq N^{-2}T^*\min\left(M, \frac{\gamma \log\log N}{c B^2}\right)$. Letting $T_N\overset{def}{=}\frac{\zeta(2) N^2}{2\epsilon^2 \log N}$, we get that since $\sigma> s+1$ that
\begin{equation*}
\|e^{-4\pi^2itN^2|K|^2}\widehat v(t,NK)-\frac{\epsilon}{N}g(\frac{t}{T_N},K)\|_{\ell^{2,s}_N(\Z_N^2)}\lesssim \frac{\epsilon B}{N(\log N)^{1-\gamma}},
\end{equation*}
for all $0\leq t \leq T_N \min\left(M, \frac{\gamma \log\log N}{c B^2}\right)$. This translates on $\Z^2$ to the estimate
\begin{equation*}
\|e^{-4\pi^2it|k|^2}\widehat v(t,k)-\frac{\epsilon}{N}g(\frac{t}{T_N},\frac{k}{N})\|_{\ell^{2,s}(\Z^2)}\lesssim \frac{(\epsilon N^s) B}{(\log N)^{1-\gamma}},
\end{equation*}
over the same time interval. Recalling that $\epsilon=N^{-s}$ gives \eqref{v_K and g}.
\end{proof}

\section{Hamiltonian structure of (CR)}\label{sectionhamiltonian}
Recall that the \eqref{star} equation is given by
\begin{equation}\tag{CR}
\begin{split}
-i\partial_t g(\xi,t)=& \mathcal T(g,g,g)(\xi,t);\qquad \xi \in \R^2\\
\mathcal T(g,g,g)(\xi,t)=& \int_{-1}^1\int_{\R^2} g(\xi+\lambda z, t) \overline{g}(\xi+\lambda z+z^\perp)g(\xi+z^\perp) \,dz\,d\lambda.
\end{split}
\end{equation}
By changing variables (formally for now) one can notice that the trilinear operator $T(g,g,g)$ can be written in several different ways like 
\begin{equation*}
\begin{split}
\mathcal T(g,g,g)(\xi,t)=& \int_{-1}^1\int_{\R^2} g(\xi+ z)\overline{g}(\xi+\lambda z^\perp+z)g(\xi+\lambda z^\perp, t)  \,dz\,d\lambda\\
=&\frac{1}{2}\int_\R\int_{\R^2} g(\xi+\lambda z, t) \overline{g}(\xi+\lambda z+z^\perp)g(\xi+ z^\perp)\,dz \,d\lambda.
\end{split}
\end{equation*}
In the latter integral, $\lambda$ runns over $\mathbb{R}$ instead of $[-1,1]$, as can be seen 
by a simple change of integration variables $\lambda \to \lambda^{-1}$.

\subsection{The Hamiltonian functional}

The equation (CR) derives from the Hamiltonian functional
\begin{equation}\label{Hamil2}
\mathcal{H}(f) = \frac{1}{4} \int_{-1}^1 \int_{\mathbb{R}^2} \int_{\mathbb{R}^2} 
 f(\xi+z) f(\xi +\lambda z^\perp)  \overline{f(\xi +z+\lambda z^\perp)} \, \overline{f(\xi)} \,dz\,d\xi\,d\lambda
\end{equation}
given the symplectic form
\begin{equation}
\label{homme}
\omega(f,g) = \frak{Im} \int_{\mathbb{R}^2} f(x) \overline{g (x)} \,dx. 
\end{equation}

Let us notice right away that the Hamiltonian can be equivalently written
\begin{equation}
\label{penguin}
\mathcal{H}(f) = \frac{1}{8} \int_{\mathbb{R}} \int_{\mathbb{R}^2} \int_{\mathbb{R}^2} 
 f(\xi+z) f(\xi+\lambda z^\perp)  \overline{f(\xi +z+\lambda z^\perp)} \, \overline{f(\xi)} \,dz\,d\xi\,d\lambda.
\end{equation}

A remarkable identity is the fact that 
\begin{equation}\label{HamStich}
\mathcal H(f)=\frac{\pi}{2} \int_{\R_t \times \R^2_x} \left| e^{it\Delta_{\R^2}} \widehat f\,\right| \,dx\,dt=\frac{\pi}{2} \|e^{it\Delta_{\R^2}} \widehat f\|_{L^4_{t,x}(\R_t\times \R_x^2)}^4
\end{equation}
which can be seen by expanding the right-hand-side as follows. Suppose first that $f\in \mathcal S(\R^2)$ (the case of general $f\in L^2$ follows by a limiting argument using \eqref{HL2bdd} which we prove later). Then
\begin{align*} 
& \|  e^{it\Delta} \widehat f \|_{L^4}^4 = \| | e^{it\Delta} \widehat f |^2  \|_{L^2}^2
= \frac{1}{4\pi^2} \| e^{-it|\xi|^2} f(\xi) * e^{it |\xi|^2} \overline{ f (\xi)} \|_{L^2}^2 \\
& = \frac{1}{4\pi^2} \int_{\mathbb{R}^2} \int_{\mathbb{R}^2}  \int_{\mathbb{R}^2}  e^{it 2 \xi \cdot \alpha} f(\eta+\alpha) \overline{ f(\eta+\alpha+\xi)} \, \overline{ f(\eta)} f(\eta+\xi) \,d\xi \,d\eta\,d\alpha \\
& = \frac{1}{4\pi^2} \int_{\mathbb{R}^2}\int_{\mathbb{R}^2} \int_{\mathbb{R}} \int_{\mathbb{R}} e^{i2t\mu |\alpha|^2}  f(\eta+\alpha) \overline{ f(\eta+\alpha+\mu \alpha + \lambda \alpha^\perp)} \,  \overline{ f(\eta)} f(\eta+\mu \alpha + \lambda \alpha^\perp) |\alpha|^2 d\lambda \, d \mu \,d\eta\,d\alpha,
\end{align*}
where we changed in the last line coordinates by $\xi = \mu \alpha + \lambda \alpha^\perp$.
Integrating in time and using the identity $\int_{\mathbb{R}} e^{ix y}dy = 2\pi \delta_{x=0}$ gives
$$
\|  e^{it\Delta} \widehat f \|_{L^4}^4 = \frac{1}{4\pi} \int \int \int  f(\eta+\alpha) \overline{ f(\eta+\alpha+\lambda \alpha^\perp) } \,\overline{ f(\eta) } f (\eta+\lambda \alpha^\perp)  d\lambda \,d\eta\,d\alpha = \frac{2}{\pi} \mathcal{H}(f).
$$
In particular, $\mathcal H(f)$ is positive definite, i.e. $\mathcal H(f)\geq 0$ and $\mathcal H(f)=0$ if and only if $f=0$.

\subsection{Symmetries}

\begin{proposition}
\label{robin}
The following symmetries leave the Hamiltonian $\mathcal{H}$ invariant:
\begin{itemize}
\item[(i)] Phase rotation: $f \mapsto e^{i \theta} f$.
\item[(ii)] Translation: $f \mapsto f(\,\cdot\,+x_0)$ for any $x_0 \in \mathbb{R}^2$.
\item[(iii)] Modulation: $f \mapsto e^{ix\cdot \xi_0} f$ for any $\xi_0 \in \mathbb{R}^2$.
\item[(iv)] Quadratic modulation: $f \mapsto e^{i\tau|x|^2} f$ for any $\tau \in \mathbb{R}$.
\item[(v)] Schr\"odinger group: $f \mapsto e^{i\tau\Delta} f$ for any $\tau \in \mathbb{R}$.
\item[(vi)] Rotation: $f \mapsto f(R_\theta\, \cdot\,)$, for any $\theta$, where $R_\theta$ is the rotation 
of angle $\theta$.
\item[(vii)] Scaling $f \mapsto \mu f (\mu \,\cdot\,)$ for any $\mu > 0$.
\item[(viii)] Fourier transform $f \mapsto \widehat{f}$.
\end{itemize}
\end{proposition}

\begin{proof}
The points (i), (ii), (iii), (iv), (vi), (vii) follow simply from the formula giving $\mathcal{H}$, in particular
from the fact that the arguments of $f$ in this formula satisfy the relations
$$
\left\{ \begin{array}{l} 
(\xi+z) + (\xi+\lambda z^\perp) = (\xi +z+\lambda z^\perp) + (\xi) \quad \mbox{and}\quad \\
|\xi+z|^2 + |\xi+\lambda z^\perp|^2 = |\xi +z+\lambda z^\perp|^2 + |\xi|^2.
\end{array}
\right.
$$
The point (v) is a consequence of (iv) and (viii). Thus we are left with proving (viii); this will be done in 
Lemma~\ref{albatros}.
\end{proof}

Since the transformations (i) to (vi) are generated by Hamiltonian flows, and as a result they give, by Noether's theorem, conserved quantities for the equation (CR).

\begin{corollary}
\label{pelican}
The following quantities are conserved by the flow of {\rm (CR)}:
\begin{itemize}
\item[(i)] Mass: $\int |f(x)|^2\,dx$.
\item[(ii)] Momentum: $\int \xi \left| \widehat{f}(\xi) \right|^2\,d\xi$.
\item[(iii)] Position: $\int x |f(x)|^2 \,dx$.
\item[(iv)] First moment $\int |x|^2 |f(x)|^2 \,dx$.
\item[(v)] Kinetic energy: $\int \left| \nabla f (x) \right|^2\,dx$.
\item[(vi)] Angular momentum $i \int (x \times \nabla) f(x) \overline{f(x)} \,dx$.
\item[(vii)] Hamiltonian $\mathcal{H}(f)$.
\end{itemize}
\end{corollary}

The following is essentially a corollary of Theorem~\ref{robin}.

\begin{corollary}
The symmetries mentioned in Proposition~\ref{robin} (acting only in the $x$ variable) leave the 
set of solutions of {\rm (CR)} invariant. This is also the case for the transformation
\begin{equation}
\label{mesange}
f \mapsto \lambda^\alpha \left( \lambda^{2\alpha - 2\gamma} t \,,\,\lambda^\gamma x \right) \qquad \mbox{where $\lambda,\alpha,\gamma>0$}.
\end{equation}
\end{corollary}

\begin{remark}[scaling]
It is natural to ask which is the critical space for {\rm (CR)}. Since it enjoys a family of possible scalings, as noted above, the answer is not as
clear as for, say, {\rm (NLS)}. However, it seems natural to think of $L^2$ as a critical space for data. Indeed,
\begin{itemize}
\item The space $L^2$ is critical for {\rm (NLS)}, from which {\rm (CR)} is derived.
\item It is essentially the largest Sobolev space for which local well-posedness can be proved, see Section~\ref{sectionanalytic}.
\item Finally, the only scaling above which does not act on the time variable, namely $f \mapsto \lambda f (\lambda \,\cdot\,)$, leaves the norm of $L^2$ invariant.
\end{itemize}
\end{remark}

\begin{lemma}
\label{albatros}
If $f,g,h \in L^2$, 
$$
\mathcal{F} \left( \mathcal{T} (f,g,h) \right) = \mathcal{T} \big( \widehat{f},\widehat{g},\widehat{h} \big)
\quad \mbox{and}\quad
\mathcal{H}(f) = \mathcal{H} \big( \widehat{f} \big).
$$
\end{lemma}

\begin{remark}
Why this Fourier transform symmetry actually holds is somewhat mysterious. It is reminiscent of the invariance of the cubic nonlinear Schr\"odinger equation set on $\mathbb{R}^2$ by the pseudo-conformal transformation
$$
u(t,x) \mapsto u_c(t)\overset{def}{=}\frac{1}{t} e^{i \frac{|x|^2}{4t}} u \left( -\frac{1}{t} \,,\,\frac{x}{t} \right).
$$
which can also be written at the level of profiles $f=e^{-it\Delta_{\R^2}}u$ and $f_c=e^{-it\Delta_{\R^2}} u_c$ as
$$
f(t)=\overline{\mathcal F f_c}(t^{-1}).
$$
Also notice that the well known fact that
$$
\mbox{as $t \rightarrow \infty$}, \qquad \left\| e^{it\Delta} \psi(x) - \frac{1}{4 \pi i t} e^{i \frac{|x|^2}{4t}}
\widehat{\psi}\left(\frac{x}{2t} \right) \right\|_{L^2} \longrightarrow 0,
$$
if $\psi \in \mathcal S(\R^2)$, combined with $L^2$ boundedness of $\mathcal{H}$, implies that points (iii) and (vii) of Proposition~\ref{robin} are equivalent.
\end{remark}

\begin{proof}
Again we assume that all functions are in Schwartz class $\mathcal S(\R^2)$. The case of functions in $L^2$ follows from a standard limiting argument using Lemma \ref{baldeagle}. The second equality is a consequence of the first and Plancherel's theorem:
$$
\mathcal{H}(f) = \frac{1}{4} \left< \mathcal{T}(f,f,f)\,,\,f \right> = \frac{1}{4} \left< \mathcal{F} \mathcal{T}(f,f,f) \,,\, \widehat{f}
\right> = \frac{1}{4} \left< \mathcal{T} \left( \widehat{f},\widehat{f},\widehat{f} \right) \,,\, \widehat{f} \right>
=  \mathcal{H}\left( \widehat{f} \right).
$$
We now prove the first identity: By Fourier inversion and using the identity $\frac{1}{4\pi^2} \int e^{ix \cdot \xi} \,dx = \delta_{\xi = 0}$ we compute
\begin{equation*}
\begin{split}
\mathcal{F} \mathcal{T} &(f,g,h)(\xi)  = \frac{1}{2\pi} \int_{\mathbb{R}^2} \int_{-1}^1 \int_{\mathbb{R}^2} e^{-iz\cdot \xi} f(x+z) g(\lambda x^\perp + z) 
\overline{h(x+\lambda x^\perp +z)}\,dx\,d\lambda \,dz \\
& = \frac{1}{16 \pi^4} \int \int_{-1}^1 \int \int  \int \int e^{iz\cdot (-\xi+\alpha+\beta-\gamma)} e^{ix\cdot(\alpha-\lambda \beta^\perp + \lambda \gamma^\perp - \gamma)} \widehat{f}(\alpha) \widehat{g} (\beta) \overline{\widehat{h}(\gamma)} \,d\alpha \,d\beta \,
d\gamma \,dx\,d\lambda \,dz \\
& =  \int_{-1}^1 \int \widehat{f} \left( \xi + \frac{1}{\lambda} \xi^\perp - \frac{1}{\lambda} \beta^\perp \right) \widehat{g}(\beta) 
\overline{\widehat{h} \left( \frac{1}{\lambda} \xi^\perp + \beta - \frac{1}{\lambda} \beta^\perp \right)}\,d\beta \,
\frac{d \lambda}{\lambda^2}.
\end{split}
\end{equation*}
Changing variables to $\beta' = \frac{1}{\lambda} ( \xi - \beta)^\perp$ and then omitting the primes gives
$$
\mathcal{F} \mathcal{T} (f,g,h)(\xi) = \int_{-1}^1 \int \widehat{f} (\xi + \beta) \widehat{g} (\xi+\lambda \beta^\perp)
\overline{\widehat{h} ( \beta + \lambda \beta^\perp + \xi )} \,d\beta\,d\lambda,
$$
which is the desired result.
\end{proof}

\subsection{Invariance of eigenspaces of the Harmonic oscillator}

Recall that the eigenvalues of the harmonic oscillator $-\Delta+|x|^2$ on $L^2(\mathbb{R}^2)$ are given by $2k$, $k\in \N$. The smallest eigenvalue, $2$, has an eigenspace $E_2$ of dimension 1, generated by the Gaussian $e^{-|x|^2/2}$. The eigenspace $E_{2k}$ has dimension $k$, and is generated by Hermite functions.

\begin{proposition}
The eigenspaces $E_k$, $k \in 2 \mathbb{N}+2$, of the harmonic oscillator are invariant by the flow of {\rm (CR)}.
\end{proposition}
\begin{proof} 

Let us introduce some notations: 
Given a functional $\mathcal F$, denote its symplectic gradient with respect to the symplectic form $\omega$ (see \eqref{homme}) by $\nabla_\omega \mathcal F$,
thus $d\mathcal F(X) = \omega (X,\nabla_\omega \mathcal F)$; denote further $\{ \mathcal F,\mathcal G\}$ for the Poisson bracket $\omega (\nabla_\omega \mathcal F,\nabla_\omega \mathcal G)$; 
finally $[X,Y]$ is the Lie bracket of the vector fields $X$ and $Y$ . Then the Hamiltonian $\mathcal G$ is invariant by the Hamiltonian flow induced by 
$\mathcal F$ if and only if $\{ \mathcal F , \mathcal G \} = 0$. Furthermore, the Hamiltonian flows of $\mathcal F$ and $\mathcal G$ commute if and only if $[ \nabla_\omega \mathcal F , \nabla_\omega \mathcal G]$.
The formula $[ \nabla_\omega \mathcal F, \nabla_\omega \mathcal G] = \nabla_\omega \{\mathcal F,\mathcal G\} = 0$ (see~\cite{Arnold}) implies that: if $\mathcal{G}$ is invariant by the flow of $\mathcal{F}$, then the flows of $\mathcal{F}$ and $\mathcal{G}$ commute.

We saw that the transformations $e^{it\Delta}$ (the flow associated with the Hamiltonian $\mathcal K_1 \overset{def}{=} \int |\nabla f|^2$)  and $e^{it|x|^2}$ (the flow of the Hamiltonian $\mathcal K_2 \overset{def}{=} \int |x f|^2 $) both leave $\mathcal{H}$ invariant. 
If we denote by $\{\cdot,\cdot\}$ the Poisson bracket associated with the symplectic form \eqref{homme} this implies that  $\{Ê\mathcal H, \mathcal K_1\} = 0$ and $\{\mathcal H, \mathcal K_2\} = 0$, and hence 
$\{Ê\mathcal H, \mathcal K_1 + \mathcal K_2\} = 0$. Now, the flow associated with the Hamiltonian $\mathcal K_1 + \mathcal K_2$ is the transformation $e^{it(-\Delta+|x|^2)}$. By the preceding paragraph, this implies that 
the unitary flow of the harmonic oscillator and the flow of $\mathcal{H}$ commute. 

Using in addition that phase rotations commute with the flow of $\mathcal{H}$ (which we denote here by $U(t))$, we get if $f \in E_k$: for any $s,t$
$$
e^{is(-\Delta+|x|^2)} U(t)f = U(t) e^{is(-\Delta+|x|^2)} f = U(t)e^{isk} f = e^{isk} U(t)f.
$$
The equality $e^{is(-\Delta+|x|^2)} U(t)f = e^{isk} U(t)f$ implies that $U(t)f \in E_k$.
\end{proof}

\subsection{Stationary waves}

We are concerned here with solutions of the type $f = e^{i\omega t} \phi$; thus $\phi$ verifies
$$
\omega \phi = \mathcal{T} (\phi,\phi,\phi).
$$
First notice that the boundedness of $\mathcal{T}$ on $L^2$ (see Section~\ref{sectionanalytic} and \eqref{optimal cst H}) implies, after taking
the scalar product of the above with $\phi$, that
$$
0 \leq \omega \leq \frac{\pi}{8} \|\phi\|_2^2.
$$
(this is in conformity the scaling of the equation of course). Many stationary solutions can be uncovered by examining the eigenspaces of the harmonic oscillator which, as we know from the previous subsection, are 
invariant by the flow of (CR). We give a few examples:
\begin{itemize}
\item The first eigenspace $E_2$ is generated by the Gaussian $e^{-\frac{|x|^2}{2}}$; its stability and the conservation of the $L^2$ norm imply that for some constant $\omega_0$, $e^{i\omega_0 t} e^{-\frac{|x|^2}{2}}$ is an exact solution. In fact, this is equivalent to saying that 
\begin{equation}\label{Gaussians and T}
\mathcal T(e^{-\frac{|x|^2}{2}},e^{-\frac{|x|^2}{2}},e^{-\frac{|x|^2}{2}})=\omega_0 e^{-\frac{|x|^2}{2}},
\end{equation}
which can also easily be verified explicitly.

Letting the symmetries of the equation act on this Gaussian, we obtain the family of stationary waves: for $(\alpha,\beta,x_0,v_0,\gamma) \in \mathbb{R}_+ \times \mathbb{R} \times \mathbb{R}^2 \times \mathbb{R}^2 \times \mathbb{C}$,
$$
\gamma e^{-(\frac{\alpha}{2}+i\beta) |x-x_0|^2 + i \left( v_0 \cdot x + \omega_0 \frac{|\gamma|^2}{\alpha^2} t \right)}.
$$
\item The second eigenspace $E_4$ is generated by $x_1 e^{-\frac{|x|^2}{2}}$ and $x_2 e^{-\frac{|x|^2}{2}}$. Using the invariance of $E_4$, conservation of the $L^2$ norm, and of the oddness with respect to $x_1$, we obtain that $x_1 e^{-\frac{|x|^2}{2}+ i \omega_1 t}$ is an exact solution. By invariance by rotation, we obtain the family of stationary solutions
$$
(\cos \theta x_1 + \sin \theta x_2) e^{-\frac{|x|^2}{2} + i \omega_1 t}, 
$$
where $\omega_1$ is a constant and $\theta$ arbitrary in $\mathbb{R}$ (more stationary solutions can then be obtained by applying the symmetries).

\item The third eigenspace $E_6$ is generated by $x_1 x_2 e^{-\frac{|x|^2}{2}}$, $(2x_1^2 - 1)e^{-\frac{|x|^2}{2}}$ and $(2x_2^2 - 1)e^{-\frac{|x|^2}{2}}$. By conservation of the $L^2$ norm, invariance of $E_6$, and invariance of the set of radial functions, we find new exact solutions of the form
$$
(2|x|^2 - 1) e^{-\frac{|x|^2}{2}+\omega_2 t}
$$
for a constant $\omega_2$.
\end{itemize}

Orbital stability of Gaussians with respect to perturbations in $L^{2,1} \cap {H}^1$ can be easily established by relying on the conserved quantities $\int |f|^2$, $\int |xf|^2$, $\int |\nabla f|^2$. This is done in the following proposition.

\begin{proposition}
Consider data $f_0$ such that $\|f_0 - e^{-\frac{|x|^2}{2}} \|_{H^1 \cap L^{2,1}} = \delta >0$, and let $f(t)$ be the corresponding solution of \eqref{star}. Then
$$
\sup_t \operatorname{dist}_{H^1\cap L^{2,1}} (f(t),\mathbb{S}^1 e^{-\frac{|x|^2}{2}}) \lesssim \sqrt{\delta},
$$
where we denoted $\mathbb{S}^1 e^{-\frac{|x|^2}{2}}$ for $\{ \gamma e^{-\frac{|x|^2}{2}}, \gamma \in \mathbb{C} \;\mbox{and}\; |\gamma| = 1 \}$.
\end{proposition}

\begin{proof} Let us denote by $P_{k}$, with $k \in 2 \mathbb{N}$, the projectors on the eigenspaces of $-\Delta + |x|^2$. Notice first that the assumption that $\|f_0 - e^{-\frac{|x|^2}{2}} \|_{H^1 \cap L^{2,1}} = \delta$, implies that 
$$
\left| \|f_0\|_{L^2}^2 - \| e^{-\frac{|x|^2}{2}} \|_{L^2}^2 \right| \lesssim \delta
$$
and that
$$
\left| \langle (-\Delta + |x|^2) f_0\,,\,f_0 \rangle - \langle (-\Delta + |x|^2) e^{-\frac{|x|^2}{2}}\,,\,e^{-\frac{|x|^2}{2}} \rangle \right| \lesssim \delta.
$$
Furthermore, these two inequalities remain true if one replaces $f_0$ by $f(t)$, which follows by using the conserved quantities of the equation. In terms of the spectral projectors $P_{k}$, this means
\begin{align*}
& \left| \sum_{k\in 2 \mathbb{N}} \|P_k f(t)\|_{L^2}^2 - 4\pi^2 \right| \lesssim \delta \\
& \left| \sum_{k\in 2 \mathbb{N}} k^2 \|P_k f(t)\|_{L^2}^2 - 16 \pi^2 \right| \lesssim \delta.
\end{align*}
This implies that
$$
\sum_{k\in 2 \mathbb{N}+2} k^2\|P_k f(t)\|_{L^2}^2 \lesssim \delta,
$$
which easily implies the desired result.
\end{proof}

Finally, another interesting explicit solution is given by
$$
e^{i\omega_3 t} \frac{1}{|x|} \qquad \mbox{for a constant $\omega_3$}.
$$
This follows easily from the identity $\mathcal{T}(\frac{1}{|x|},\frac{1}{|x|},\frac{1}{|x|}) = \frac{\omega_3}{|x|}$, which will be established in Section~\ref{sectionanalytic}.
Once again, applying the symmetries of the equation (detailed in Section~\ref{sectionhamiltonian}) gives new solutions. 
The action of these symmetries is easily described, except for the expression $e^{i s \Delta} \frac{1}{|x|}$. We refer to the book of
Cazenave~\cite{Cazenave}, where it is shown that it is $L^\infty$, smooth, and decaying like $\frac{1}{|x|}$ for $s>0$.

As mentioned before the solution $g(t, \xi)=e^{i\omega_3 t}|\xi|^{-1}$ carries particular significance in terms of its relation to wave turbulence theory as it corresponds to the Raleigh-Jeans solution $n(\xi)=|\xi|^{-2}$ of the \eqref{KZ} equation mentioned in Section \ref{obtained results} (recall that $n(\xi,t)$ is related to the square of the Fourier modes).

\section{Analytic properties}

\label{sectionanalytic}

The analytic properties of solutions of (CR) reflect essentially those of $\mathcal{T}$. This operator does not seem to belong to a well-studied class of operators.
If one fixes $\lambda$ and integrates only over $y$ in the definition~(\ref{eq:defT}), the operator under consideration belongs to the class studied by Brascamp and Lieb
~\cite{BL}; however, the integration over $\lambda$ seems responsible for some important properties of $\mathcal{T}$, such as Theorem~\ref{hibou} below. 

The feature of $\mathcal{T}$ which differs from the setting of Brascamp and Lieb is that the arguments of $f$, $g$, and $h$ are nonlinear functions of $\lambda$ and $x$.
A nonlinear Brascamp-Lieb theory is not available, but a nonlinear analog of the related Loomis-Whitney inequality has been established by Bennett, Carbery and Wright~\cite{BCW}; we will
make use of it.

\subsection{Boundedness of $\mathcal{T}$}

\begin{proposition}
\label{baldeagle}
For the following spaces $X$, the trilinear operator $\mathcal{T}$ is bounded from $X^3$ to $X$:
\begin{itemize}
\item[(i)] $L^2$.
\item[(ii)] $\dot{L}^{\infty,1}$.
\item[(iii)] $\dot{L}^{p,1-\frac{2}{p}}$ for $p \geq 2$.
\item[(iv)] $L^{2,\sigma}$ for any $\sigma \geq 0$.
\item[(v)] $H^\sigma$ for any $\sigma \geq 0$.
\item[(vi)] $Y^{\sigma, p} \overset{def}{=} \{ f(x) \, | \, e^{\sigma |x|^2}f(x) \in  L^p\}$ for any $\sigma \geq 0$, $p\geq 2$. \black 
\end{itemize}
\end{proposition}

\begin{remark} Due to~(\ref{HamStich}), the statement (i) is equivalent to the classical $L^4$ Strichartz estimate~\cite{Tomas, Strichartz}; the proof we give is very simple, and seems not be found in the literature.

Notice that the three first statements above are at the scaling of the equation; the three last spaces
are subcritical, thus they will give for the equation the possibility to propagate smoothness, or localization.
\end{remark}

\begin{proof}
To prove (i), we argue by duality and use the Cauchy-Schwarz inequality to obtain
\begin{equation*}
\begin{split}
& \lan \mathcal{T}(f,g,h) , F \ran = \int_{-1}^1 \int \int f(x+z) g(\lambda x^\perp + z) \overline{h(x + \lambda x^\perp + z)}
\, \overline{F(z)} \,dx\,dz\,d\lambda \\
& \qquad \lesssim \int_{-1}^1 \left( \int \int |f(x+z)|^2 |F(z)|^2\,dx\,dz \right)^{1/2} \\
& \qquad \qquad \qquad \left( \int \int |h(x + z)|^2 |g( z)|^2\,dx\,dz \right)^{1/2} \,d\lambda \\
& \qquad \lesssim \|f\|_{L^2} \|g\|_{L^2} \|h\|_{L^2} \|F\|_{L^2}.
\end{split}
\end{equation*}
To prove (ii), it suffices to establish that
$$
\mathcal{T}\left(\frac{1}{|x|},\frac{1}{|x|},\frac{1}{|x|} \right) \leq  \frac{C}{|x|},
$$
this will be done in Lemma~\ref{cormorant} below. Indeed, equality holds!

The point (iii) can be obtained from (i) and (ii) and complex interpolation.

To prove (iv), proceed as for (i), using in addition the inequality
$$
\mbox{for any $\lambda \in \mathbb{R}$, $\xi,z$ in $\mathbb{R}^2$}, \qquad
\frac{\lan \xi \ran^\sigma}{\lan \xi + z \ran^\sigma \lan \xi +z+\lambda z^\perp \ran^\sigma 
\lan \xi + \lambda z^\perp + z \ran^\sigma} \lesssim 1.
$$
which follows from the relation $\xi=(\xi+z)-(\xi+z+\lambda z^\perp)+(\xi+\lambda z^\perp)$.

The assertion (v) is a consequence of (iv) and the invariance of $\mathcal{T}$ under Fourier transform proved in
Proposition~\ref{robin}.

Finally, (vi) follows from the interpolation between the two cases $p=2$ and $p=\infty$. The endpoint $p=\infty$ follows from \eqref{Gaussians and T}; as for the
endpoint $p=2$, it can be proved by combining (i) with the fact that for any $\lambda,x,z$,
$$
\mbox{for any $\lambda \in \mathbb{R}$, $x,z$ in $\mathbb{R}^2$}, \qquad
\frac{e^{\sigma|z|^2}}{e^{\sigma|x+z|^2}e^{\sigma|\lambda x^\perp +z|^2}e^{\sigma|x+\lambda x^\perp +z|^2}} \lesssim 1. 
$$
\end{proof}

We now prove the identity that was alluded to earlier.

\begin{lemma} \label{cormorant}
For some constant $C$, we have
$$
\mathcal{T}\left( \frac{1}{|x|} , \frac{1}{|x|}  , \frac{1}{|x|}  \right) = \frac{C}{|x|}.
$$
\end{lemma}

\begin{proof}
Denote $e_1$ for the point $(1,0)$ in $\mathbb{R}^2$. By scale invariance, it suffices to show that
$$
\mathcal{T} \left( \frac{1}{|x|} , \frac{1}{|x|}  , \frac{1}{|x|}  \right) (e_1) 
= \int_{-1}^1 \int_{\mathbb{R}^2} \frac{1}{|x+e_1|} \frac{1}{|\lambda x^\perp + e_1|} \frac{1}{|x + 
\lambda x^\perp + e_1|}\,dx\,d\lambda < \infty.
$$
We prove this bound by splitting the integration domain into several regions.

\bigskip

\noindent
\underline{Where $|x| \ll 1$.} The contribution of this set is easily seen to be bounded.

\bigskip

\noindent
\underline{Where $|x|>> 1$ and $|\lambda| >> \frac{1}{|x|}$.} The contribution of this set is
bounded by
$$
C \int_{|x|>>1} \int_{\frac{1}{|x|} \ll |\lambda| \ll 1} \frac{1}{|x|} \frac{1}{|\lambda x|} \frac{1}{|x|} 
\,d\lambda\,dx \lesssim \int_{|x|>>1} \frac{1}{|x|^3} \log |x| \,dx < \infty.  
$$

\bigskip

\noindent
\underline{Where $|x|>> 1$ and $|\lambda| \ll \frac{1}{|x|}$.} This contributes at most
$$
C \int_{|x| >> 1} \int_{|\lambda| \ll \frac{1}{|x|}} \frac{1}{|x|^2} d\lambda\,dx \lesssim 
\int_{|x| >> 1} \frac{1}{|x|^3} \,dx <\infty.
$$

\bigskip

\noindent
\underline{Where $|x|>> 1$ and $|\lambda| \sim \frac{1}{|x|}$.} This gives at most
$$
C \int_{|x| >> 1} \int_{\lambda \sim‎ \frac{1}{|x|}} \frac{1}{|\lambda x^\perp + e_1|} \frac{1}{|x|^2} d\lambda\,dx
= \int_{|\lambda| \ll 1} \lambda^2 \int_{|x| \sim \frac{1}{\lambda}} \frac{1}{|\lambda x^\perp + e_1|}
\,dx\,d\lambda.
$$
Changing variable $y=\lambda x^\perp + e_1$, this is less than
$$
\int_{\lambda \ll 1} \int_{|y| \lesssim 1} \frac{1}{|y|} \,dy\,d\lambda < \infty.
$$

\bigskip

\noindent
\underline{Where $|x| \sim 1$, $|x+e_1| \gtrsim 1$, $|\lambda x^\perp +e_1| \gtrsim 1$, and $|x + \lambda x^\perp + e_1| \gtrsim 1$.} This contribution is immediately seen to be finite.

\bigskip

\noindent
\underline{Where $|x|\sim 1$ and $|\lambda x^\perp + e_1| \ll 1$.} This region contributes less than
$$
C \int_{|\lambda| \sim 1} \int_{|x|\sim 1,|\lambda x^\perp + e_1| \ll 1} \frac{1}{|\lambda x^\perp + e_1|}
\,dx\,d\lambda \lesssim \int_{|\lambda| \sim 1} \int_{|y| \ll 1}  \frac{1}{|y|} \,dy\,\frac{d\lambda}{\lambda^2} < \infty,
$$
were we changed variables in the integral, setting $y = \lambda x^\perp + e_1$.

\bigskip

\noindent
\underline{Where $|x| \sim 1$, and $|x + \lambda x^\perp + e_1| \ll 1$.} By the relation 
$$
|e_1+x|^2+|e_1+\lambda x^\perp|^2=1+|e_1+x+\lambda x^\perp|^2$$
We have that \emph{both} $|x+e_1|$ and $|\lambda x+e_1|\gtrsim 1$ (since $|x|\sim 1$), and therefore this contribution can be bounded by
$$
C \int_{|\lambda| \sim 1} \int_{|x + \lambda x^\perp + e_1| \ll 1} \frac{1}{|x+\lambda x^\perp + e_1|}\,dx\,d\lambda < \infty
$$
by arguments similar to the ones already employed.

\bigskip
We are thus left with the case when $|x|\sim 1$ and $|x+e_1|\ll1$, which we split into three cases:

\medskip
\noindent
\underline{Where $|x+e_1| \ll 1$ and $|\lambda|>4|x+e_1|$.} The integral over this region can be bounded by
$$
C \int_{|x+e_1|\ll1} \int_{\substack{ -1<\lambda < 1 \\ |\lambda|>4|x+e_1|}} \frac{1}{|x+e_1|} \frac{1}{|x+\lambda x^\perp + e_1|}\,d\lambda\,dx
\lesssim \int_{|y|\ll1} \int_{\substack{ -1<\lambda < 1 \\ |\lambda|>4|y|}} \frac{1}{|y|} \frac{1}{|y + \lambda (y - e_1)^\perp |}\,d\lambda\,dy
$$
upon changing variable $y = x + e_1$. Now notice that for fixed $\lambda$
$$
|y+\lambda (y-e_1)^\perp | \geq\left|\,|y+\lambda y^\perp|-\lambda\right|=\left| \sqrt{1+\lambda^2} |y| - |\lambda| \right|.
$$
Therefore, the above can be bounded by
\begin{equation*}
\begin{split}
\dots & \lesssim \int_{|y|\ll1} \int_{\substack{ -1<\lambda < 1 \\ |\lambda|>4|y|}} \frac{1}{|y|} \frac{1}{\left| \sqrt{1+\lambda^2} |y| - |\lambda| \right|} \,d\lambda \,dy
\lesssim \int_{|y|\ll1} \int_{\substack{ -1<\lambda < 1 \\ |\lambda| >4|y|}} \frac{1}{|y|} \frac{1}{\left| \lambda \right|} \,d\lambda \,dy \\
& \lesssim \int_{|y|\ll1} \frac{|\log|y||}{|y|}\,dy <\infty.
\end{split}
\end{equation*}

\bigskip

\noindent
\underline{Where $|x+e_1| \ll 1$ and $|\lambda|<\frac{1}{4}|x+e_1|$.} Proceeding as in the previous case, we see that this contributes at most
\begin{equation*}
\begin{split}
C \int_{|y|\ll1} \int_{|\lambda|<\frac{1}{4}|y|} \frac{1}{|y|} \frac{1}{\left| \sqrt{1+\lambda^2} |y| - |\lambda| \right|} \,dy\,d\lambda
& \lesssim \int_{|y|\ll1} \int_{|\lambda|<\frac{1}{4}|y|} \frac{1}{|y|^2} \, dy \\
& \lesssim \int_{|y|\ll 1}\frac{dy}{|y|}<\infty.
\end{split}
\end{equation*}

\bigskip

\noindent
\underline{Where $|x+e_1| \ll 1$ and $\frac{1}{4}|x+e_1|<|\lambda|<4|x+e_1|$.} For simplicity in the notations, we only treat the case $\lambda>0$. This region contributes at most
\begin{equation*}
\begin{split}
& \int_{|y|\ll1}  \int_{\frac{1}{4}|y|<\lambda<4|y|} \frac{1}{|y|} \frac{1}{|y + \lambda (y - e_1)^\perp |}\,d\lambda\,dy \\
& \qquad \qquad \qquad \qquad = 2\pi \int_{|r|\ll1} \int_{\theta \in \mathbb{S}^1} \int_{\frac{1}{4}<\alpha<4} 
\frac{1}{|\omega(\theta) + \alpha (r\omega(\theta)-e_1)^\perp|}\,d\alpha\,dr\,d\theta ,
\end{split}
\end{equation*}
where we set $\lambda = \alpha r$, and parameterized $y$ by $y = r\omega(\theta)$, with $\omega(\theta)=\left(\begin{array}{l} \cos \theta \\ \sin \theta \end{array} \right)$,
and $(r,\theta) \in \mathbb{R}^+ \times \mathbb{S}^1$. Now define
$$
P(\alpha,r,\theta) = |\omega(\theta) + \alpha (r\omega(\theta)-e_1)^\perp |^2 = 1 + \alpha^2 r^2 + \alpha^2 -2 \alpha (\sin \theta + \alpha r \cos \theta).
$$
The following properties are easily established:
\begin{itemize}
\item For fixed $\alpha$ and $r$, $P$ is minimized at $\theta_0$ such that $\sin \theta_0 = \frac{1}{\sqrt{1+\alpha^2 r^2}}$ 
and $\cos \theta_0 = \frac{\alpha r}{\sqrt{1+\alpha^2 r^2}}$. The minimal value of $P$ is then $\left| \sqrt{1+\alpha^2 r^2} - \alpha \right|^2$
\item For $|\theta-\theta_0| < \frac{1}{10}$, $\partial_\theta^2 P(\alpha,r,\theta) \gtrsim \alpha$, thus
$P(\lambda,r,\theta) \gtrsim \left| \sqrt{1+\alpha^2 r^2} - \alpha \right|^2 + \alpha (\theta-\theta_0)^2$.
\item For $|\theta-\theta_0| > \frac{1}{10}$, $P(\alpha,r,\theta) \gtrsim \alpha + \left| \sqrt{1+\alpha^2 r^2} - \alpha \right|^2$.
\end{itemize}
The above can therefore be bounded by
\begin{equation*}
\begin{split}
& \dots \lesssim \int_{|r|\ll1}   \int_{\frac{1}{4}<\alpha<4} \int_{|\theta-\theta_0| < \frac{1}{10}}
\frac{1}{\sqrt{\left| \sqrt{1+\alpha^2 r^2} - \alpha \right|^2 + \alpha (\theta-\theta_0)^2}}\,d\theta\,d\alpha\,dr  \\
& \qquad \qquad \qquad \qquad + \int_{|r|\ll1}  \int_{\frac{1}{4}<\alpha<4} \int_{|\theta-\theta_0| > \frac{1}{10}} 
\frac{1}{\sqrt{\alpha + (\sqrt{1+\alpha^2 r^2} - \alpha)^2}}\,d\theta\,d\alpha\,dr  \\
\end{split}
\end{equation*}
The second term on the above right-hand side is immediately seen to be bounded, thus we focus on the first one, which can be written as\begin{equation}\label{banana}
\begin{split}
& \int_{|r|\ll1}   \int_{\frac{1}{4}<\alpha<4}\int_{|\theta| < \frac{1}{10}}
\frac{1}{\sqrt{\left| \sqrt{1+\alpha^2 r^2} - \alpha \right|^2 + \alpha \theta^2}}\,d\theta\,d\alpha\,dr  \\
& \qquad \qquad \qquad \qquad 
 \lesssim \int_{|r|\ll1} \int_{\frac{1}{4}<\alpha<4} \langle \log \left| \sqrt{1 + \alpha^2 r^2} - \alpha \right| \rangle \,d\alpha\,dr,
\end{split}
\end{equation}
where we used the following simple bound for $0\leq u \leq 10$:
$$
\int_{|\theta|<1} \frac{1}{\sqrt{u + \theta^2}}\,d\theta=\int_{|\theta|\leq u^{-1/2}} \frac{1}{\sqrt{1 + \theta^2}}\,d\theta\lesssim \langle \log u \rangle.
$$
Finally we note that the integral \eqref{banana} is finite by the change of variable $\alpha \mapsto f(\alpha)= \alpha-\sqrt{1 + \alpha^2 r^2}$ which satisfies $f^\prime(\alpha)\sim 1$ in the considered range of $\alpha$ and $r$.
\end{proof}

The final lemma in this subsection will be useful to prove the well-posedness of \eqref{star} in the spaces $X^\sigma$ that were used in the approximation results of Theorems \ref{approx thm1} and \ref{result on T^2}.
\begin{lemma}\label{X^sigma bound of TT}
The following tame bound holds: for any functions $f$, $g$ and $h$,  and $\sigma \geq 0$ 
\begin{equation}
\label{TT in X^sigma}
\|T(f,g,h)\|_{X^\sigma}\lesssim \|f\|_{X^\sigma}\|g\|_{L^2}\|h\|_{L^2}+\|f\|_{L^4}\|g\|_{L^{4/3}}\|h\|_{X^\sigma}
\end{equation}
\end{lemma}

\begin{proof} Without loss of generality, we assume that $f,g,$ and $h$ are non-negative. Using the orthogonality relation
$$
|\xi|^2+|\xi+\lambda z +z^\perp|^2=|\xi+\lambda z|^2+|\xi+z^\perp|^2
$$
we arrive at the conclusion that for any $\xi, z\in \R^2$ and $\lambda\in [-1,1]$,
$$
\langle \xi\rangle^\sigma\lesssim_\sigma \langle \xi+\lambda z\rangle^{\sigma}+\langle \xi+z^\perp\rangle^\sigma.
$$
As a result, we get that
\begin{align*}
\langle \xi \rangle^\sigma \TT(f,g,h)\lesssim_\sigma \TT(\langle \,\cdot\, \rangle^\sigma f, g, h) +\TT(f,g,\langle\, \cdot\, \rangle^\sigma h)
\end{align*}
and hence we arrive at 
\begin{align*}
\|\TT(f,g,h)\|_{X^\sigma}\lesssim_\sigma& \|f\|_{X^\sigma}\sup_{\xi\in\R^2}\int_{-1}^1\int_{\R^2}g(\xi+\lambda z+z^\perp) h(\xi+z^\perp) dz\,d\lambda\\
&+\|h\|_{X^\sigma}\sup_{\xi \in \R^2}\int_{-1}^1 \int_{\R^2} f(\xi+\lambda z) g(\xi+\lambda z +z^\perp)dz\,d\lambda\\
\lesssim& \|f\|_{X^\sigma}\|g\|_{L^2}\|h\|_{L^2}+\|h\|_{X^\sigma}\int_{-1}^1 \underbrace{\|f(\xi+\lambda z)\|_{L_z^4}}_{=|\lambda|^{-1/2}\|f\|_{L^4}} \underbrace{\|g(\xi+\lambda z +z^\perp)\|_{L_z^{4/3}}}_{\lesssim \|g\|_{L^{4/3}}} d\lambda\\
\lesssim& \|f\|_{X^\sigma}\|g\|_{L^2}\|h\|_{L^2}+\|h\|_{X^\sigma}\|f\|_{L^4}\|g\|_{L^{4/3}},
\end{align*}
from which \eqref{TT in X^sigma} follows. 
\end{proof}

\subsection{Local and global well-posedness for (CR)}

We consider the Cauchy problem for (CR):
\begin{theorem} 
\label{th:7.4}
\begin{itemize}
\item[(i)] Local well posedness: For $X$ being any of the  spaces given in Proposition \ref{baldeagle}, the 
Cauchy problem {\rm (CR)} is locally well-posed in $X$.
That is, for any $f_0$ in $X$, there exists a time $T>0$, and  a solution in $\mathcal{C}^\infty ([0,T],X)$,
which is unique in $L^\infty ([0,T],X)$, and depends continuously on $f_0$ in this topology. 

\item[(ii)] Global well-posedness: if $f_0 \in L^2$, the local solution can be prolonged into a global one.
More precisely: there exists a unique solution in $\mathcal{C}^\infty ([0,\infty),L^2)$. For any $T$, it is unique
in $L^\infty ([0,T],L^2)$, and depends continuously on $f_0$ in this topology. 

Furthermore, the mass, kinetic energy, and first moment of $f$ remain constant in time if they were finite
at time 0.

\item[(iii)] Propagation of regularity: assume $f_0 \in L^2$, and let $f$ be the solution given in (ii).
If in addition $f_0 \in H^\sigma$ (respectively $L^{2,\sigma}$) for $\sigma  \geq 0$ then $f \in 
\mathcal{C}^\infty ([0,\infty),H^\sigma)$ (respectively $\mathcal{C}^\infty ([0,\infty),
L^{2,\sigma})$. 

\item[(iv)] The Cauchy problem {\rm (CR)} is locally well-posed in the space $X^\sigma$ for any $\sigma > 2$. If $g_0\in H^1$, the solution can be extended globally in time in $X^\sigma$. Moreover, if $\nabla g(0) \in X^\sigma$, then $g, \nabla g \in C_{\operatorname{loc}}(\R; X^\sigma)$.
\end{itemize}
\end{theorem}

\begin{remark} The restriction $\sigma>2$ for local well-posedness in item (iv) could easily be improved to $\sigma \geq 1$, but global well-posedness beyond $\sigma>2$ seems very difficult to prove. \end{remark}

\begin{proof} (i) Write (CR) as a fixed point problem:
$$
f(t) = f_0 + i \int_0^t \mathcal{T}(f,f,f)(s) \,ds.
$$
Proposition~\ref{baldeagle} immediately gives the a priori estimates
$$
\left\{ \begin{array}{l} \left\| f \right\|_{L^\infty([0,T],X)} \lesssim \|f_0\|_X + T \|f\|_{L^\infty([0,T],X)}^3 \\
\left\| f-g \right\|_{L^\infty([0,T],X)} \lesssim \|f_0-g_0\|_X + T \|f-g\|_{L^\infty([0,T],X)}
\left( \|f\|_{L^\infty([0,T],X)}^2 + \|g\|_{L^\infty([0,T],X)}^2 \right), \end{array} \right.
$$
for $f$ and $g$ two solutions of (CR), from which one deduces the local well-posedness statement
by a standard fixed point argument. Notice that $T$ can be chosen depending only on $\|f_0\|_X$ (namely $\sim \|f_0\|_X^{-2}$).

\medskip

(ii) Combining the above local well-posedness with the conservation laws (Corollary~\ref{pelican}) gives 
the global well-posedness (recall that the local well-posedness time $T$ only depends on $\|u_0\|_2$).

\medskip

(iii) is classical, so we only say a few words. Using the boundedness of $\TT$ from $L^2\times L^2 \times L^2\to L^2$, one notices that the equation is locally well-posed in $H^s$ with the time of existence depending on the conserved quantity $\|g(t)\|_{L^2}$. Consequently, one can iterate the local existence statement indefinitely to get global existence. The result for weighted spaces $L^{2,\sigma}$ follows from the invariance under Fourier transform.

(iv) Local well-posedness in $X^\sigma$ is a direct consequence of Picard iteration and Lemma \ref{X^sigma bound of TT}. This gives a time of existence bounded below by $C\|g_0\|_{X^\sigma}^{-2}$. Global existence follows from the same lemma, which gives due to the embeddings $H^1\hookrightarrow L^4$ and $L^{2,1}\hookrightarrow L^{4/3}$ and the conservation of the $H^1$ and $L^{2,1}$ norms that
$$
\|g(t)\|_{X^\sigma}\leq \|g_0\|_{X^\sigma}+\|g_0\|_{H^1}\|\langle x\rangle g_0\|_{L^2} \int_0^t \|g(s)\|_{X^\sigma} ds.
$$
This yields an a priori bound on $\|g(t)\|_{X^\sigma}$ that allows extending the solution globally. Finally, if $\nabla g_0 \in X^\sigma$, then proceeding as above and using that $\nabla g$ satisfies
\begin{equation}\label{nablag eqn}
\nabla g(t,\xi)=\nabla g_0(\xi) +i\int_0^t\left(\TT(\nabla g, g,g)+\TT(g,\nabla g, g)+\TT(g,g,\nabla g)\right) ds,
\end{equation}
one can obtain local-in-time existence of $\nabla g \in C([0,T]; X^\sigma)$. To extend this to a global statement, one has to show that $\|\nabla g\|_{X^\sigma}$ cannot blow up. This follows once we obtain an a priori bound for $\|\nabla g\|_{X^\sigma}$ that prohibits its blowup. Using Lemma \ref{X^sigma bound of TT} and \eqref{nablag eqn}, we have
\begin{align*}
\|\nabla g(t)\|_{X^\sigma}\lesssim& \|\nabla g_0\|_{X^\sigma}+\int_0^t \left(\|\nabla g(s)\|_{X^\sigma}\|g(s)\|_{L^2}^2+\|g(s)\|_{X^\sigma}\|\nabla g\|_{L^4}\| g(s)\|_{L^{4/3}}\right.\\
 &\left. +\|g(s)\|_{X^\sigma}\|\nabla g\|_{L^2}\|g\|_{L^2} +\|g(s)\|_{X^\sigma}\|\nabla g(s)\|_{L^{4/3}}\|g\|_{L^4}+\|\nabla g(s)\|_{X^\sigma}\|g\|_{L^4}\|g\|_{L^{4/3}}\right) ds\\
 \lesssim& \|\nabla g_0\|_{X^\sigma}+\int_0^t \|\nabla g\|_{X^\sigma}\|g(s)\|_{X^\sigma}^2\,ds
\end{align*}
Since the solution $g$ is bounded in $X^\sigma$ on any compact time interval of $\R$, we get the needed a priori bound on $\|\nabla g\|_{X^\sigma}$ to finish the proof.
\end{proof}

\subsection{Weak continuity of $\mathcal{T}$ on $L^2$}

Recall that
$$
\mathcal{T}(f,g,h)(\xi) = \int_{-1}^1 \int_{\R^2} f(\xi +x) \overline{g(\xi+x+\lambda x^\perp )} h(\xi + \lambda x^\perp) \,dx\,d\lambda.
$$
We introduce the following notation: $\mathcal{T}_{\substack{x \in I \\ \lambda \in J}} (f,g,h)$ is defined as $\mathcal{T}$ above, but with the integration domain changed
from $(x,\lambda) \in \mathbb{R}^2 \times [-1,1]$ to $(x,\lambda) \in I \times (J \cap [-1,1])$. For instance:
$$
\mathcal{T}_{\substack{|x| \sim 1 \\ |\lambda| < 1/4}} (f,g,h)(\xi) = \int_{|\lambda|<1/4} \int_{|x|\sim 1} 
f(\xi +x) \overline{g(\xi+x+\lambda x^\perp )} h(\xi + \lambda x^\perp) \,dx\,d\lambda.
$$
We start by proving more precise estimates on such localized versions of $\mathcal{T}$. 
Space and frequency localization operators will be needed; we start by defining them. Pick $\psi$ a smooth function equal to $1$ on $B(0,1)$,
$0$ on $B(0,2)^c$, and set
\begin{equation*}
\begin{split}
& Q_{<M} f \overset{def}{=} \psi \left(\frac{x}{M}\right) f \quad,\quad Q_{>M}\overset{def}{=} f - P_{<M} f \\
& P_{<M} \overset{def}{=} \mathcal{F}^{-1} \psi \left(\frac{\xi}{M}\right) \widehat{f}(\xi)
\quad \mbox{and} \quad P_{>M}\overset{def}{=} f - Q_{<M} f
\end{split}
\end{equation*}

\begin{proposition}
\label{loriot}
\begin{itemize}
\item[(i)] If $f$, $g$, $h$ belong to $L^2$,
$$
\left\| \mathcal{T}_{|\lambda|<\epsilon}(f,g,h) \right\|_{L^2(\mathbb{R}^2)} \lesssim \epsilon \|f\|_{L^2(\mathbb{R}^2)}\|g\|_{L^2(\mathbb{R}^2)} \|h\|_{L^2(\mathbb{R}^2)}
$$
\item[(ii)] For $\epsilon>0$, there exists a constant $C(\epsilon)$ such that if $f$, $g$, $h$ belong to $L^2$,
$$
\left\| \mathcal{T}_{\substack{|x| \geq M \\ |\lambda| > \epsilon}}(f,g,h) \right\|_{L^\infty(\mathbb{R}^2)} \leq C(\epsilon) M^{-1} \|f\|_{L^2(\mathbb{R}^2)} \|g\|_{L^2(\mathbb{R}^2)} \|h\|_{L^2(\mathbb{R}^2)}.
$$
\item[(iii)] For $\epsilon>0$, $R\ll M$, there exists a constant $C(\epsilon)$ such that if $f$, $g$, $h$ belong to $L^2$,
\begin{equation*}
\begin{split}
\left\| \mathcal{T}_{|\lambda|>\epsilon} (Q_{>M} f,g,h) \right\|_{L^\infty(B(0,R))} +\left\| \mathcal{T}_{|\lambda|>\epsilon} (f,Q_{>M} g,h) \right\|_{L^\infty(B(0,R))} \leq C(\epsilon)  M^{-1} \|f\|_{L^2(\mathbb{R}^2)} \|g\|_{L^2(\mathbb{R}^2)} \|h\|_{L^2(\mathbb{R}^2)} \\
\end{split}
\end{equation*}
\end{itemize}
\end{proposition}

\begin{proof} \underline{Proof of (i).} The proof of (i) follows exactly that of (i) in Proposition~\ref{baldeagle}, we do not repeat it here. 

\bigskip

\noindent
\underline{Proof of (ii).} Observe that
$\mathcal{T}_{\substack{|x| \sim 1 \\ |\lambda| > \epsilon}}(f,g,h)(0)$ can be written
$$
\mathcal{T}_{\substack{|x| \sim 1 \\ |\lambda| > \epsilon}}(f,g,h)(0) = \int_{\substack{|x| \sim 1 \\ 1> |\lambda| > \epsilon}}
f(\pi_1(x,\lambda)) \overline{g}(\pi_2(x,\lambda)) h(\pi_3(x,\lambda))\,dx\,d\lambda
$$
with $\pi_1(x,\lambda) = x$, $\pi_2(x,\lambda) =x + \lambda x^\perp$, and $\pi_3(x,\lambda) =  \lambda x^\perp$. This corresponds to the type of operators analyzed in Bennett, Carbery and Wright~\cite{BCW}. The non-degeneracy condition in Theorem 2 of this article can be checked, implying the nonlinear Loomis-Whitney inequality
$$
\left| \mathcal{T}_{\substack{|x| \sim 1 \\ |\lambda| > \epsilon}}(f,g,h)(0) \right| \leq C(\epsilon) \|f\|_{L^2} \|g\|_{L^2}  \|h\|_{L^2} .
$$
Since $\mathcal{T}_{\substack{|x| \sim 1 \\ |\lambda| > \epsilon}}(f,g,h)(z) = \mathcal{T}_{\substack{|x| \sim 1 \\ |\lambda| > \epsilon}}(f(z+\cdot),g(z+\cdot),h(z+\cdot))(0)$,
we deduce that
$$
\left\| \mathcal{T}_{\substack{|x| \sim 1 \\ |\lambda| > \epsilon}}(f,g,h)\right\|_{L^\infty} \leq C(\epsilon) \|f\|_{L^2}  \|g\|_{L^2}  \|h\|_{L^2} .
$$
By scaling,
$$
\left\| \mathcal{T}_{\substack{|x| \sim 2^k \\ |\lambda| > \epsilon}}(f,g,h)\right\|_{L^\infty} \leq C(\epsilon) 2^{-k} \|f\|_{L^2}  \|g\|_{L^2}  \|h\|_{L^2} .
$$
This gives the desired result since
$$
\left\| \mathcal{T}_{\substack{|x| > M \\ |\lambda| > \epsilon}}(f,g,h)\right\|_{L^\infty}
\leq \sum_{2^k \gtrsim M} \left\| \mathcal{T}_{\substack{|x| \sim 2^k \\ |\lambda| > \epsilon}}(f,g,h)\right\|_{L^\infty}
\leq C(\epsilon) M^{-1} \|f\|_{L^2}  \|g\|_{L^2}  \|h\|_{L^2} .
$$

\bigskip

\noindent
\underline{Proof of (iii).} We only prove that the first term on the left-hand side satisfies the inequality; the second being similar. Thus we wish to estimate
$$
\mathcal{T}_{|\lambda|>\epsilon} (Q_{>M} f,g,h)(z) =  \int_{1>|\lambda|>\epsilon} \int \left[ Q_{>M} f \right] (x+z)\overline{g(x + \lambda x^\perp + z)}  h(\lambda x^\perp + z) \,dx\,d\lambda, 
$$
where $|z|<R$. Since $Q_{>M} f$ is supported outside of a ball of radius $\sim M$ and center $0$, we see that only $x$ such that $|x+z| \gtrsim M$ contribute to the above
integral. Since $M \gg R$, this implies $|x| \gtrsim M$. But then the estimate follows from (ii).
\end{proof}

These estimates lead to the following theorem.

\begin{theorem}
\label{hibou}
Suppose that $(f^n)$, $(g^n)$ and $(h^n)$ are three sequences such that
$$
f^n \rightharpoonup f \quad\,\quad g^n \rightharpoonup g \quad\,\quad h^n \rightharpoonup h \quad\,\quad \mbox{in $L^2$}.
$$
(we denote as is customary $f^n \rightarrow f$ for the strong convergence in $L^2$, and $f^n \rightharpoonup f$ for the weak convergence in $L^2$). Then
$$
\mathcal{T}(f^n,g^n,h^n) \rightharpoonup \mathcal{T}(f,g,h) \quad\,\quad \mbox{in $L^2$}.
$$
\end{theorem}

\begin{remark}
It is tempting to mention here the div-curl lemma~\cite{Murat}~\cite{Tartar}: if $(u^n)$, $(v^n)$ are sequences of functions from $\mathbb{R}^3$ to $\mathbb{R}^3$,
in $L^2$, which are respectively divergence and curl free, and converge weakly to $u$ and $v$, then their scalar product $u^n \cdot v^n$ converges to $u \cdot v$ in the 
sense of distribution. It was later realized by Coifman, Lions, Meyer, Semmes~\cite{CLMS} that this particular structure also has implications in harmonic analysis. Namely, if $u$, $v$ belong in $L^2$, and are respectively divergence and curl free, then their scalar product $u \cdot v$ not only belongs to $L^1$, 
but even to the Hardy space $\mathcal{H}^1$.

One might wonder whether a similar phenomenon occurs for $\mathcal{T}$: does this operator map $L^2 \times L^2 \times L^2$ to a slightly smaller space than $L^2$?
\end{remark}

\begin{proof} First fix sequences $f^n$, $g^n$ and $h^n$ converging weakly to $f$, $g$, $h$ in $L^2$.
By the uniform boundedness principle, the functions $f^n$, $g^n$, $h^n$ enjoy uniform bounds in $L^2$.
Our aim will be to show that, for $\phi$ in the Schwartz class $\mathcal{S}$,
$$
\langle \mathcal{T}(f^n,g^n,h^n)\,,\,\phi \rangle \rightarrow \langle \mathcal{T}(f,g,h)\,,\,\phi \rangle.
$$
This will give the desired result by boundedness of $\mathcal{T}$ on $L^2$, and density of $\mathcal{S}$ in $L^2$. Thus we fix from now on $\phi \in \mathcal{S}$; all implicit constants are allows to depend on $\phi$. 

We will be using three parameters, $\epsilon$, $M$ and $R$, whose precise value will be fixed at the end of the proof, but which satisfy $\epsilon \ll 1$, $M\gg R\gg1$.

\bigskip

\noindent
\underline{Step 1: localization in space.} 
Writing $f = Q_{<M} f^n + Q_{>M} f^n$, and similarly for $g$ and $h$, we obtain that $\mathcal{T}(f,g,h)(x)$ 
can be written as the sum
\begin{subequations}
\begin{align}
\label{chardonneret1}
\langle \mathcal{T}(f^n,g^n,h^n)\,,\,\phi \rangle = & \langle \mathcal{T}(Q_{>M} f^n,g^n,h^n)  \,,\,\phi \rangle \\
& \label{chardonneret2} \quad + \langle \mathcal{T}(Q_{<M} f^n,g^n,Q_{>M}h^n) \,,\,\phi \rangle \\
& \label{chardonneret3} \quad + \langle \mathcal{T}(Q_{<M} f^n,Q_{>10 M}g^n,Q_{<M}h^n) \,,\,\phi \rangle \\
& \label{chardonneret4} \quad + \langle \mathcal{T}(Q_{<M} f^n,Q_{< 10 M}g^n,Q_{<M}h^n)  \,,\,\phi \rangle. 
\end{align}
\end{subequations}
To estimate~(\ref{chardonneret1}), we first split this term, and then use Proposition~\ref{loriot} as well as the fast decay of $\phi$:
\begin{equation}
\begin{split}
\left| (\ref{chardonneret1}) \right| & \leq \left| \langle \mathcal{T}_{|\lambda|<\epsilon} (Q_{>M} f^n, g^n, h^n) \,,\,\phi \rangle \right| 
 + \left|\langle \mathcal{T}_{|\lambda|>\epsilon}(Q_{>M} f^n,g^n,h^n)  \,,\,Q_{<R}\phi \rangle \right| \\
& \qquad \qquad + \left|\langle \mathcal{T}_{|\lambda|>\epsilon}(Q_{>M} f^n,g^n,h^n)  \,,\,Q_{>R}\phi \rangle \right| \\
& \lesssim \epsilon \|f^n\|_2 \|g^n\|_2 \|h^n\|_2 \|\phi\|_2 +  C(\epsilon) M^{-1} \|f^n\|_2 \|g^n\|_2 \|h^n\|_2 \|\phi\|_1 + \|f^n\|_2 \|g^n\|_2 \|h^n\|_2 R^{-10}\\
& \lesssim \epsilon +  C(\epsilon) M^{-1} + R^{-10}.
\end{split}
\end{equation}
The next term,~(\ref{chardonneret2}), can be treated identically to give
\begin{equation}
\left| (\ref{chardonneret2}) \right| \lesssim \epsilon +  C(\epsilon) M^{-1} + R^{-10}.
\end{equation}
Finally, taking advantage of the localization properties of $\mathcal{T}$, (\ref{chardonneret3}) can be bounded by
$$
\left| (\ref{chardonneret3}) \right| = \left| \langle \mathcal{T}(Q_{<M} f^n,Q_{>10 M}g^n,Q_{<M}h^n) \,,\,Q_{>R} \phi \rangle \right|
\lesssim \|f^n\|_2 \|g^n\|_2 \|h^n\|_2 R^{-10} \lesssim R^{-10}.
$$
We are left with~(\ref{chardonneret4}) which we will decompose further, in frequency this time.

\bigskip

\noindent
\underline{Step 2: localization in frequency.} By Plancherel's theorem and Lemma~\ref{albatros},
$$
4\pi^2 (\ref{chardonneret4}) = \langle \mathcal{F} \mathcal{T}(Q_{<M} f^n,Q_{< 10 M}g^n,Q_{<M}h^n)  \,,\,\widehat{\phi} \rangle
=  \langle \mathcal{T} ( P_{<M} \widehat{f^n},P_{< 10 M} \widehat{g^n},P_{<M} \widehat{h^n}) \,,\,\widehat{\phi} \rangle.
$$
Performing the same decomposition as in Step 1, this is equal to
\begin{subequations}
\begin{align}
\label{bruant1}
(\ref{chardonneret4}) = & \langle \mathcal{T}(Q_{>M} P_{<M} \widehat{f^n},P_{< 10 M} \widehat{g^n},P_{<M} \widehat{h^n})  \,,\,\widehat{\phi} \rangle \\
& \label{bruant2} \quad + \langle \mathcal{T}(Q_{<M} P_{<M} \widehat{f^n},P_{< 10 M} \widehat{g^n},Q_{>M}P_{<M} \widehat{h^n}) \,,\,\widehat{\phi} \rangle \\
& \label{bruant3} \quad + \langle \mathcal{T}(Q_{<M} P_{<M} \widehat{f^n},Q_{>10 M} P_{< 10 M} \widehat{g^n},Q_{<M}P_{<M} \widehat{h^n}) \,,\,\widehat{\phi} \rangle \\
& \label{bruant4} \quad + \langle \mathcal{T}(Q_{<M} P_{<M} \widehat{f^n},Q_{< 10 M}P_{< 10 M} \widehat{g^n},Q_{<M}P_{<M} \widehat{h^n})  \,,\,\widehat{\phi} \rangle,
\end{align}
\end{subequations}
and just like in Step 1 we can bound
$$
|(\ref{bruant1})| + |(\ref{bruant2})| + |(\ref{bruant3})| \lesssim \epsilon +  C(\epsilon) M^{-1} + R^{-10}.
$$

\bigskip

\noindent
\underline{Step 3: conclusion.} Fix $\eta>0$, we want to show that, for $n$ big enough,
$$
\left| \langle \mathcal{T}(f^n,g^n,h^n)\,,\,\phi \rangle - \langle \mathcal{T}(f,g,h)\,,\,\phi \rangle \right| < \eta.
$$
First split the above as
\begin{subequations}
\begin{align}
& 4 \pi^2 \left| \langle \mathcal{T}(f^n,g^n,h^n)\,,\,\phi \rangle - \langle \mathcal{T}(f,g,h)\,,\,\phi \rangle \right| \\
& \qquad = \left| \langle \mathcal{T}(\widehat{f^n},\widehat{g^n},\widehat{h^n})\,,\,\widehat{\phi} \rangle - \langle \mathcal{T}(\widehat{f},\widehat{g},\widehat{h})\,,\,\widehat{\phi} \rangle \right| \\
& \label{petrel1} \qquad \leq \left| \langle \mathcal{T}(\widehat{f^n},\widehat{g^n},\widehat{h^n})\,,\,\widehat{\phi} \rangle -
 \langle \mathcal{T}(Q_{<M} P_{<M} \widehat{f^n},Q_{< 10 M}P_{< 10 M} \widehat{g^n},Q_{<M}P_{<M} \widehat{h^n}) \,,\,\widehat{\phi} \rangle \right| \\
& \label{petrel2} \qquad \quad \begin{array}{l} 
                  + \left|\langle \mathcal{T}(Q_{<M} P_{<M} \widehat{f^n},Q_{< 10 M}P_{< 10 M} \widehat{g^n},Q_{<M}P_{<M} \widehat{h^n})\,,\,\widehat{\phi} \rangle\right. \\
\qquad \qquad \qquad \qquad - \left. \mathcal{T}(Q_{<M} P_{<M} \widehat{f},Q_{< 10 M}P_{< 10 M} \widehat{g},Q_{<M}P_{<M} \widehat{h})\,,\, \widehat{\phi} \rangle \right|
                 \end{array} \\
& \label{petrel3} \qquad \quad + \left|\langle \mathcal{T}(Q_{<M} P_{<M} \widehat{f},Q_{< 10 M}P_{< 10 M} \widehat{g},Q_{<M}P_{<M} \widehat{h})\,,\, \widehat{\phi} \rangle 
- \langle \mathcal{T}(\widehat{f},\widehat{g},\widehat{h}) \,,\,\widehat{\phi} \rangle \right|
\end{align}
\end{subequations}
Gathering the estimates of Step 1 and Step 2, we see that
$$
\left| (\ref{petrel1}) \right| + | \eqref{petrel3}|
\lesssim \epsilon +  C(\epsilon) M^{-1} + R^{-10}.
$$
The weak convergence of $f^n$, $g^n$ and $h^n$ to, respectively, $f$, $g$ and $h$ in $L^2$ implies that, for $M$ fixed
$$
Q_{<M} P_{<M} \widehat{f^n} \overset{L^2}{\rightarrow} Q_{<M} P_{<M} \widehat{f} \;\;,\;\;
 Q_{< 10 M}P_{< 10 M} \widehat{g^n} \overset{L^2}{\rightarrow}  Q_{< 10 M}P_{< 10 M} \widehat{g}  \;\;,\;\;
 Q_{<M}P_{<M} \widehat{h^n} \overset{L^2}{\rightarrow} Q_{<M}P_{<M} \widehat{h}
$$
as $n\rightarrow \infty$. Since $\mathcal{T}$ is bounded on $L^2$, we deduce that
$$
\left| (\ref{petrel2}) \right| \overset{n \rightarrow \infty}{\longrightarrow} 0.
$$
To conclude, it suffices to fix $\epsilon$, $M$, $R$ such that $\left| (\ref{petrel1}) \right| + \left|(\ref{petrel3}) \right| < \frac{1}{2} \eta$.
Then
$$
\operatorname{limsup}_{n \rightarrow \infty} \left| \langle \mathcal{T}(f^n,g^n,h^n)\,,\,\phi \rangle - \langle \mathcal{T}(f,g,h)\,,\,\phi \rangle \right| < \frac{1}{2} \eta.
$$
\end{proof}

\subsection{Weak compactness of solutions in $L^\infty_t L^2_x$} The preceding theorem on weak continuity of $\mathcal{T}$ on $L^2$ implies easily the following compactness
result.
\begin{corollary}
Assume $(f^n)$ is a sequence of solutions of {\rm (CR)} in $L^\infty([0,T],L^2)$. Then there exists a subsequence, which we still denote $(f^n)$,
and a solution $f \in L^\infty([0,T],L^2)$ of {\rm (CR)} such that
$$
f^n \overset{n \rightarrow \infty}{\longrightarrow} f \qquad \mbox{weak-* in $L^\infty ([0,T],L^2)$}.
$$
\end{corollary}
\begin{proof}
First notice that $(f^n)$ is uniformly bounded in $\operatorname{Lip}([0,T],L^2)$. This means that we may define $(f^n(t))$ for all $t\in [0,T]$ and for each such $t$, there exists a weak limit $f(t) \in L_t^\infty ([0,T];L^2)$ such $f^n(t) \rightharpoonup f(t)$ weakly in $L^2(\R^2)$. One can easily verify that $f^n(t) \rightharpoonup f(t)$ weak-* in $L_t^\infty([0,T]; L^2)$ and that $f(t) \in \operatorname{Lip}([0,T],L^2)$.

It suffices to show now that $f$ solves (*). Set
$$
f^n_j(t) = f^n \left( \frac{k}{2^j} \right) \quad \mbox{and} \quad f_j(t)=f \left( \frac{k}{2^j} \right) \qquad \mbox{where $k$ minimizes $\left| t - \frac{k}{2^j} \right|$}.
$$
We start from
$$
\partial_t f_n = \mathcal{T}(f_n,f_n,f_n)
$$
and want to pass to the limit $n \rightarrow \infty$ in the sense of distributions. For the linear part, it is automatic: $\partial_t f_n \rightarrow \partial_t f$.
Now take $\phi$ in the Schwartz class $\mathcal{S}([0,T] \times \mathbb{R}^2)$. Then
\begin{equation*}
\begin{split}
\left| \langle \mathcal{T}(f^n,f^n,f^n)\,,\,\phi \rangle - \langle \mathcal{T}(f,f,f)\,,\,\phi \rangle \right|
& \leq \left| \langle \mathcal{T}(f^n,f^n,f^n)\,,\,\phi \rangle - \langle \mathcal{T}(f^n_j,f^n_j,f^n_j)\,,\,\phi \rangle \right| \\
& \quad + \left| \langle \mathcal{T}(f^n_j,f^n_j,f^n_j)\,,\,\phi \rangle - \langle \mathcal{T}(f_j,f_j,f_j)\,,\,\phi \rangle \right| \\
& \quad + \left| \langle \mathcal{T}(f_j,f_j,f_j)\,,\,\phi \rangle - \langle \mathcal{T}(f,f,f)\,,\,\phi \rangle \right| \\
& \leq \left| \langle \mathcal{T}(f^n_j,f^n_j,f^n_j)\,,\,\phi \rangle - \langle \mathcal{T}(f_j,f_j,f_j)\,,\,\phi \rangle \right| + O(2^{-j}),
\end{split}
\end{equation*}
where we used in the last line the uniform bound on $(f^n)$ and $f$ in $\operatorname{Lip}([0,T]; L^2)$. From this last inequality, by weak convergence of 
the $f^n_j$ to $f_j$ in $L^2$, and the weak continuity of $\mathcal{T}$ in $L^2$, we deduce that
$$
\operatorname{limsup}_{n \rightarrow \infty} \left| \langle \mathcal{T}(f^n,f^n,f^n)\,,\,\phi \rangle - \langle \mathcal{T}(f,f,f)\,,\,\phi \rangle \right| \lesssim 2^{-j} 
$$
for any $j$; but this gives the desired result.
\end{proof}

\section{Variational properties}\label{sectionvariational}

We examine here more precisely how the Hamiltonian $\mathcal{H}$ and the $L^2$ mass are related. Our main result is that Gaussians maximize the Hamiltonian for prescribed $L^2$ norm. This implies dynamical stability of the Gaussians.

The results in this section would follow immediately from the relation $\mathcal{H}(f) = \|e^{it\Delta} f\|_4^4$, and known results on the $L^4$ Strichartz norm, in particular~\cite{Fos, HZ, BBCH}. Still, we chose to ignore this relation and give full proofs of Proposition~\ref{duck} and Theorem~\ref{duckling}; they turn out to be very short and provide a slightly different point of view.


\subsection{Positivity of the Hamiltonian}

\begin{proposition}
The Hamiltonian can be written
\begin{equation}
\label{duck}
\mathcal{H}(f) = \frac{1}{16}
\int_{\mathbb{S}^1} \int_\mathbb{R} \int_\mathbb{R} \left| 
\int_{\mathbb{R}} f ( u\omega + s \omega^\perp ) \overline{f(t\omega + s \omega^\perp)} \,ds \right|^2
du\,dt\,d\omega.
\end{equation}
In particular, it is non-negative, and zero only if $f=0$.
\end{proposition}

\medskip

\begin{proof}
The proof of the formula only consists of a change of integration variables in the formula~(\ref{penguin}). Parametrizing in this formula
$z$ by $v \omega$, with $(v,\omega) \in \mathbb{R} \times \mathbb{S}^1$ (at the expense of an extra factor of 2 in the demominator), and $\xi$ by $t \omega + s \omega^\perp$,
with $(s,t) \in \mathbb{R}^2$, it becomes
\begin{equation*}
\begin{split}
\mathcal{H}(f) 
& = \frac{1}{16} \int_{\mathbb{S}^1} \int_\mathbb{R} \int_\mathbb{R} \int_\mathbb{R} \int_\mathbb{R} 
f\left( (t+v) \omega + s \omega^\perp \right) f \left( t \omega + (s + \lambda v)  \omega^\perp \right) \\
& \qquad \qquad \qquad \qquad \qquad \qquad \qquad
\overline{f \left( (t+v) \omega + (s+\lambda v) \omega^\perp \right)}
\, \overline{f\left(t\omega + s \omega^\perp \right)} |v| \,d\lambda\,ds\,dt\,dv \,d\omega.
\end{split}
\end{equation*}
Setting now $\lambda' = s + \lambda v$, $u = t + v$ gives
\begin{equation*}
\begin{split}
\mathcal{H}(f) 
& = \frac{1}{16} \int_{\mathbb{S}^1} \int_\mathbb{R} \int_\mathbb{R} \int_\mathbb{R} \int_\mathbb{R} 
f\left( u \omega + s \omega^\perp \right) f \left( t \omega + \lambda'  \omega^\perp \right) \\
& \qquad \qquad \qquad \qquad \qquad \qquad \qquad
\overline{f \left( u \omega + \lambda' \omega^\perp \right)}
\, \overline{f\left(t\omega + s \omega^\perp \right)} \,ds\,d\lambda'\,dt\,du \,d\omega,
\end{split}
\end{equation*}
which is the desired result. It is then a small exercise to see that $\mathcal{H}$ vanishes only for the zero function.
\end{proof}

\subsection{Maxima of the Hamiltonian}

By Proposition~\ref{baldeagle}, $\mathcal{H}$ is bounded on $L^2$. In other words,
$$
0 \leq \mathcal{H}(f) \lesssim \|f\|_2^4.
$$
Our next aim here is to understand for which functions the above inequality is saturated, and what is the best constant.

\begin{theorem}
\label{duckling}
For fixed mass $\|f\|_2$, the Hamiltonian $\mathcal{H}(f)$ is maximized if $f$ equals, up to the 
symmetries of the equation, $\|f\|_2 G$, where $G$ is the $L^2-$normalized Gaussian $G(x) = \frac{1}{\sqrt \pi} e^{-\frac{|x|^2}{2}}$. Furthermore, all the maximizers are of the form $\|f\|_2 G$, up to the symmetries of the equation. Finally, the optimal constant in the bound for $\mathcal{H}$ is given by
\begin{equation}\label{optimal cst H}
\mathcal{H}(f) \leq \frac{\pi}{8} \|f\|_2^4.
\end{equation}
\end{theorem}

\begin{remark}
As noticed earlier, the operator $\mathcal{T}$, or the Hamiltonian $\mathcal{H}$, correspond to a nonlinear generalization of the class studied by Brascamp and Lieb.
Whether Brascamp-Lieb inequalities are saturated by Gaussians has been studied by many authors, we mention in particular~\cite{Lieb,CLL,BCCT}.Our proof relies on the heat flow, following an approach initiated in~\cite{CLL,BCCT}.
\end{remark}

\begin{proof}
\underline{Step 1: if $g \in L^1$, convergence in $L^1$ of $\sqrt{t \left[ e^{\frac{t}{4}\Delta} g \right] (\sqrt{t} x)}$ to $\left( \int g \right)^{1/2} G(x)$}.
Recalling that the kernel of $e^{t\Delta}$ is $\frac{e^{-\frac{|x|^2}{4t}}}{4\pi t}$, it is easy to see that this statement is true if $g$ is in addition in 
$\mathcal{C}_0^\infty$. The general case $g \in L^1$ follows by a density argument.

\bigskip

\noindent
\underline{Step 2: if $f \in L^2$, $f \geq 0$, $\mathcal{H}\left( \sqrt{e^{\tau \Delta} f} \right)$ is increasing with $\tau$.}
Notice that for $f\geq 0$, $ \sqrt{e^{\tau \Delta} f}$ has a constant $L^2-$norm. It suffices to treat the case where $f$ is smooth with rapid decay, the general case follows
by an approximation argument using the boundedness of $\mathcal{H}$ on $L^2$.
Furthermore, since $e^{\tau \Delta} f (x)>0$ for $\tau>0$, we assume from
now on that $f$ is positive, smooth, and decays rapidly, which justifies all the manipulations which follow.
Recall the formula~(\ref{duck}) giving $\mathcal{H}(f)$. It can be written as
$$
\mathcal{H}(f) = \int_{\mathbb{S}^1} \mathcal{H}_\omega (f) d\omega \quad \mbox{with} \quad 
\mathcal{H}_\omega (f) = \int_\mathbb{R} \int_\mathbb{R} \left| 
\int_s f ( u\omega + s \omega^\perp ) \overline{f(t\omega + s \omega^\perp)} \,ds \right|^2
du\,dt.
$$
We will show that for any $\omega$, $\frac{d}{d\tau} \mathcal{H}_\omega(\sqrt{e^{\tau \Delta} f}) \geq 0$.
Due to the semi-group property of $e^{\tau \Delta}$, it suffices to prove it for $\tau =0$, which we proceed to do.
For convenience, we adopt the Cartesian coordinates given by $\omega$, and denote simply
$f(u,s) = f ( u\omega + s \omega^\perp )$. Then
\begin{equation*}
\begin{split}
& \frac{d}{d\tau} \mathcal{H}_\omega \left( \sqrt{e^{\tau \Delta} f} \right) (\tau=0) \\
& \qquad \qquad = 2 \int \int \int \int \sqrt{f} (u,\sigma) \sqrt{f}(t,\sigma) \sqrt{f}(t,s) \left[ \partial_s^2 + \partial_u^2 \right]
f(u,s) \frac{1}{\sqrt{f}(u,s)} \,ds\,d\sigma \,du\,dt \\
& \qquad \qquad \overset{def}{=} I + II.
\end{split}
\end{equation*}
Since $I$ and $II$ are symmetrical, it suffices to show that $II \geq 0$. But integrating by parts 
once in $u$, it can be written
\begin{equation}
\label{elephant}
\begin{split}
II & = \int \int \int \int \left[ - \frac{ \sqrt{f}(t,\sigma) \sqrt{f}(t,s) }{\sqrt{f}(u,\sigma)
\sqrt{f}(u,s) } \partial_u f(u,\sigma) \partial_u f (u,s) \right. \\
& \qquad \qquad \qquad \qquad \left. + \frac{\sqrt{f}(t,\sigma) \sqrt{f}(t,s) 
\sqrt{f}(u,\sigma)}{f(u,s)\sqrt{f}(u,s)} \left| \partial_u f(u,s) \right|^2 \right] \,du\,dt\,d\sigma\,ds \\
&= \int\int\int \int \frac{1}{2} \left( \partial_u f(u,s) \frac{f^{1/2}(u,\sigma)}{f^{1/2}(u,s)}
- \partial_u f(u,\sigma) \frac{f^{1/2}(u,s)}{f^{1/2}(u,\sigma)} \right)^2\frac{ \sqrt{f}(t,\sigma) \sqrt{f}(t,s) }{\sqrt{f}(u,\sigma)
\sqrt{f}(u,s) }\,
du\,dt\,d\sigma\,ds \\
& \geq 0,
\end{split}
\end{equation}
which concludes the proof.

\bigskip

\noindent\
\underline{Step 3: Gaussians are maximizers} Let $f$ in $L^2$. Then
\begin{equation}
\begin{split}
\mathcal{H}(f) & \leq \mathcal{H}(|f|) \\
& \leq \mathcal{H}\left( \sqrt{ e^{\frac{t}{4}\Delta} |f|^2 } \right) \qquad \mbox{for $t\geq 0$, by Step 2} \\
& = \mathcal{H}\left( \sqrt{t \left[ e^{\frac{t}{4}\Delta} |f|^2 \right] (\sqrt{t} \, \cdot \, ) } \right) \qquad \mbox{by scaling invariance (Proposition~\ref{robin})} \\
& \overset{t \rightarrow \infty}{\longrightarrow} \|f\|^4_2 \;\mathcal{H}(G) \qquad \mbox{by Step 1 and continuity of $\mathcal{H}$ on $L^2$ (Proposition~\ref{baldeagle})},
\end{split} 
\end{equation}
which gives the desired inequality.

\bigskip

\noindent
\underline{Step 4: up to the symmetries, Gaussians are the unique non-negative maximizers.} 
Assume that $g$ a maximizer with $g \geq 0$. We will show that $g$ is a Gaussian, up to the symmetries of the equation. Applying the heat flow, we might assume without loss of generality
that $g$ is positive and smooth. Then, if it is a maximizer, $\frac{d}{dt} \mathcal{H}(\sqrt{ e^{t \Delta} g } ) = 0$, which, by~(\ref{elephant}), implies (with the notations of Step 2) that
$$
\frac{\partial_u g(u,s)}{g(u,s)} = \frac{\partial_u g(u,\sigma)}{g(u,\sigma)} \quad \mbox{for any}\; u,s,\sigma.
$$
Setting $h(u,s) = \log g(u,s)$, this becomes
$$
\partial_u h(u,s) = \partial_u h(u,\sigma), \quad \mbox{for any}\; u,s,\sigma.
$$
We can rewrite the above as $\partial_u h(u,s) = \partial_u h(u,s + s_0)$ for any $u$, $s$, $s_0$. Taking the Fourier transform of this identity in $u$, $s$ and denoting $\xi$, $\eta$ for their
respective Fourier variables, we obtain
$$
\xi \widehat{h}(\xi,\eta) = \xi \widehat{h}(\xi,\eta) e^{i\eta s_0} \quad \mbox{for any}\;\xi,\eta,s_0.
$$
From this identity one deduces first that $\widehat{h}$ is localized at $(0,0)$, and then that it is a linear combination of Dirac masses $\delta_0$, first derivatives of Dirac masses $\partial_i \delta_0$,
and $\Delta \delta_0$. Thus, $h$ is a second order polynomial with an isotropic second order part.
Coming back to $g$, and using that it belongs to $L^2$, it has to be a Gaussian up to symmetries.

\bigskip

\noindent
\underline{Step 5: up to the symmetries, Gaussians are the unique complex-valued maximizers.}  
Assume that $f$ is a maximizer. So is $|f|$, which, by the previous step, is a Gaussian. We can thus assume that $f = G e^{i\theta}$, where $\theta$ is real and smooth. By adding a constant factor to $\theta$, we may also assume that $\mathcal H(f)=\mathcal H(G)>0$. Using the original formulation of $\mathcal H(f)$ in \eqref{Hamil2}, it is easy to see that we must then have that for any $\xi , z\in \R^2$ and $\lambda\in [-1,1]$:
$$
\theta(\xi+z^\perp)+\theta(\xi+\lambda z)=\theta(\xi)+\theta(\xi+z^\perp+\lambda z).
$$
Taking the derivative with respect to $\lambda$ we get that for any $\xi, z\in \R^2, \lambda \in [-1,1]$:
$$
z\cdot \nabla \theta(\xi+\lambda z)=z\cdot\nabla \theta(\xi+z^\perp+\lambda z).
$$
which implies that 
\begin{equation*}
z\cdot \nabla \theta(\eta)=z\cdot \nabla \theta(\eta+z^\perp)\quad \mbox{ for any }\eta, z \in \R^2.
\end{equation*}
As a result, for any $\omega \in \mathbb{S}^1$ and with the notation $\theta(u,s)=\theta(u\omega+s\omega^\perp)$, it holds that
$$
\partial_u \theta(u,s) = \partial_u \theta(u,\sigma), \quad \mbox{for any}\; u,s,\sigma.
$$
This identity already appeared in the previous step. Arguing as there, we obtain that $\theta$ is a quadratic polynomial with isotropic second order part, which gives the desired result. 

\end{proof}

\subsection{Compactness of maximizing sequences and $L^2$ stability of Gaussians}

Let $\mathcal{G}$ be the manifold obtained by letting the symmetries of $\eqref{star}$ act on $G$:
$$
\mathcal{G} = \left\{ \gamma G\left( \frac{x-x_0}{|\gamma|} \right) e^{iv_0 \cdot x + i\beta |x|^2}, \; (\gamma, x_0,v_0,\beta) \in \mathbb{C} \times \mathbb{R}^2 \times \mathbb{R}^2 \times \mathbb{R} \right\}.
$$

\begin{theorem} 
\label{bull}
Let $(f^n)$ be a sequence in $L^2$ such that, for some $Q > 0$, 
$$
\|f^n\|_2 \rightarrow Q \quad \mbox{and} \quad \mathcal{H}(f^n) \rightarrow \frac{\pi}{8} Q^4 \qquad \mbox{as $n \rightarrow \infty$}.
$$
Then
$$
\operatorname{dist}_{L^2}(f^n,Q \mathcal{G}) \longrightarrow 0 \qquad \mbox{as $n \rightarrow \infty$}.
$$
\end{theorem}
\begin{proof} The proof follows by combining Theorem~\ref{duckling} with the profile decomposition in~\cite{MV} via standard concentration compactness arguments. We omit the details.\end{proof}

This theorem has an easy consequence: dynamic stability of Gaussians by the flow of (CR).

\begin{proposition}
Assume that $(f^n_0)$ is a sequence in $L^2$ converging to $\mathcal{G}$. Denote $f_n(t)$ for the solution of \eqref{star} with data $f^n_0$. Then for any time $t>0$,
$$
\operatorname{dist}_{L^2}(f^n(t) ,Q \mathcal{G}) \longrightarrow 0 \qquad \mbox{as $n \rightarrow \infty$}.
$$
\end{proposition}

\appendix

\section{The case of dimension 1}

It is easy and instructive to describe in dimension 1 the weakly nonlinear ($\epsilon \rightarrow 0$), big box ($L \rightarrow \infty$) limit of (NLS).
We sketch it here and start by considering
$$
- i \partial_t u + \Delta u = \epsilon^2 |u|^2 u \qquad \mbox{set on $\mathbb{T}_L$}.
$$
First expand $u$ in Fourier series: $u = \frac{1}{L} \sum_{K \in \mathbb{Z}/L} a_K e^{2\pi iKx}$, and set $b_K(t)= a_K(L^2 t) e^{-i  4\pi^2 t L^2 K^2}$. It solves
$$
- i \partial_t b_K = \epsilon^2 \sum_{K_1 - K_2 + K_3 = K} b_{K_1} \bar{b}_{K_2} b_{K_3} e^{i 4 \pi^2 t L^2 \Omega } \qquad \mbox{with $\Omega = -K^2 + K_1^2 - K_2^2 + K_3^2$}.
$$
In the above sum, keep only the resonant terms, for which $\Omega$ vanishes; the other ones have less impact on the dynamics due to the weak nonlinearity hypothesis. This gives
$$
- i \partial_t b_K = - \epsilon^2 |b_K|^2 b_K + 2 \epsilon^2 \underbrace{\left(\sum_{J \in \mathbb{Z}/L} |b_J|^2 \right)}_{L \|u\|_{L^2(\mathbb{T}_L)}^2} b_K. 
$$
Noticing that the  $L^2$ norm of the solution is preserved for all times, and changing the dependent variables one last time to $c_K = b_K e^{i 2 t \epsilon^2 L \|u\|_2^2}$, we obtain the equation
$$
- i \partial_t c_K = - \epsilon^2 |c_K|^2 c_K.
$$
In the above equation, $K$ ranges over $\mathbb{Z}/L$. As $L \rightarrow \infty$, it is clear that the above converges to the continuous equation
$$
- i \partial_t g(t,\xi) = - \epsilon^2 |g(t,\xi)|^2 g(t,\xi), \qquad \mbox{where $(t,\xi) \in \mathbb{R}^2$}.
$$
Rescaling time gives
$$
i \partial_t g(t,\xi) = |g(t,\xi)|^2 g(t,\xi).
$$
The corresponding Cauchy problem for an initial data $g(0,\xi) = g_0(\xi)$ is easily solved:
\begin{equation}
\label{eq:wtdim1}
g(t,\xi) = g_0(\xi) e^{it |g_0(\xi)|^2}.
\end{equation}
One notices that in this 1D setting, the continuous resonant dynamics are not as interesting as the 2D case. This is due to the absence of non-trivial cubic resonances in dimension 1. Indeed, an application of a normal forms transformation indicates that the nonlinear interactions responsible for any change in the dynamics of ``action" variables $|\widehat u(K)|^2$ are higher order (namely quintic). On the other hand, equation \eqref{eq:wtdim1} suggest that the corresponding ``angle" variables tend to become (generically) decorrelated as time increases.


\subsection*{Acknowledgements} The authors would like to thank Benoit Pausader for a discussion leading to the interpretation in Section \ref{Interpretation}.


\begin{thebibliography}{00}

\bibitem{Arnold} V. Arnold, \textit{Mathematical methods of classical mechanics.} Translated from the Russian by K. Vogtmann and A. Weinstein. 
Second edition. Graduate Texts in Mathematics, 60. Springer-Verlag, New York, 1989.

\bibitem{BBCH} J. Bennett, N. Bez, A. Carbery, D. Hundertmark, \textit{Heat-flow monotonicity of Strichartz norms.} Anal. PDE 2 (2009), no. 2, 147--158.

\bibitem{BCW} J. Bennett, A. Carbery, J. Wright, 
\textit{A non-linear generalisation of the Loomis-Whitney inequality and applications.}
Math. Res. Lett. 12 (2005), no. 4, 443--457. 

\bibitem{BCCT} J. Bennett, A. Carbery, M. Christ, T. Tao, \textit{The Brascamp-Lieb inequalities: finiteness, structure and extremals.}
Geom. Funct. Anal. 17 (2008), no. 5, 1343--1415.

\bibitem{Bourgain} J. Bourgain, \textit{Fourier transform restriction phenomena for certain lattice subsets and applications to nonlinear evolution equations.}
I. Schr\"odinger equations. Geom. Funct. Anal. 3 (1993), no. 2, 107--156.

\bibitem{Bourgain2} J. Bourgain, \textit{Invariant measures for the 2D-defocusing nonlinear Schr\"odinger equation}, Comm. Math. Phys. 176 (1996), no. 2, 421--445. 

\bibitem{BourgainQuasi} J. Bourgain, \textit{Quasi-periodic solutions of Hamiltonian perturbations of 2D linear Schr\"odinger equations}, Ann. of Math. (2) 148 (1998), no. 2, 363--439. 

\bibitem{Bourgain3} J. Bourgain, \textit{On Strichartz's inequalities and the nonlinear Schr\"odinger equation on irrational tori}, Mathematical aspects of nonlinear dispersive equations, 1--20, Ann. of Math. Stud., 163, Princeton Univ. Press, Princeton, NJ, 2007.

\bibitem{BL} H. Brascamp, E. Lieb, \textit{Best constants in Young's inequality, its converse, and its generalization to more than three functions.} Advances in Math. 20 (1976), no. 2, 151--173.

\bibitem{CMMT} D.~Cai, A.~J.~Majda, D.~W.~McLaughlin, E.~G.~Tabak, \emph{Dispersive wave turbulence in one dimension}. Advances in nonlinear mathematics and science. Phys. D 152/153 (2001), 551--572. 

\bibitem{CLL} E. Carlen, E. Lieb, M. Loss, \textit{A sharp analog of Young's inequality on SN and related entropy inequalities.} J. Geom. Anal. 14 (2004), no. 3, 487--520.

\bibitem{CF} R. Carles, E. Faou, \textit{Energy cascades for NLS on the torus.} Discrete Contin. Dyn. Syst. 32 (2012), no. 6, 2063--2077. 

\bibitem{CW} T. Cazenave and F. Weissler, \textit{The Cauchy problem for the nonlinear Schr\"odinger Equation in $H^s$}, Nonlinear Analysis, 14:807--836, 1990.

\bibitem{Cazenave} T. Cazenave, \textit{Semilinear Schr\"odinger equations}, Courant Lecture Notes in Mathematics, 10.  
American Mathematical Society, Providence, RI, 2003.

\bibitem{CLMS} R. Coifman, P.L. Lions, Y. Meyer, S. Semmes, 
\textit{Compensated compactness and Hardy spaces}, J. Math. Pures et Appl. 72 (1992), pp. 247--286.

\bibitem{CKSTT} J. Colliander, M. Keel, G. Staffilani, H. Takaoka, T. Tao, 
\textit{Transfer of energy to high frequencies in the cubic defocusing nonlinear Schr\"odinger equation.} Invent. Math. 181 (2010), no. 1, 39--113. 

\bibitem{CNP} C. Connaughton, S. Nazarenko, A. Pushkarev, \emph{Discreteness and quasiresonances in weak turbulence of capillary waves}, Physical Reviews E, 63, 046306, (2001).

\bibitem{DLN} P.~Denissenko, S.~Lukaschuk, S.~Nazarenko, \emph{Gravity wave turbulence in a laboratory flume}. Physical review letters 99.1 (2007): 014501.

\bibitem{Dodson1} B. Dodson, \textit{Global well-posedness and scattering for the defocusing, $L^{2}$-critical, nonlinear Schr\"odinger equation when $d = 2$}, preprint arXiv:1006.1375 [math.AP].

\bibitem{Dodson2} B. Dodson, \textit{Global well-posedness and scattering for the mass critical nonlinear Schr\"odinger equation with mass below the mass of the ground state}, preprint arXiv:1104.1114 [math.AP].

\bibitem{EliaKuk} L.~H.~Eliasson, S.~Kuksin, \textit{KAM for the nonlinear Schr\"odinger equation}, Ann. of Math. (2) 172 (2010), no. 1, 371--435. 

\bibitem{FLG} E. Faou, L. Gauckler, C. Lubich, \textit{Sobolev stability of plane wave solutions to the cubic nonlinear Schr\"odinger equation on a torus}. 
Comm. PDE, to appear. 

\bibitem{Fos} D. Foschi, \textit{Maximizers for the Strichartz inequality.} J. Eur. Math. Soc. (JEMS) 9 (2007), no. 4, 739--774.

\bibitem{Geng}
J. Geng, X. Xu and J. You, \textit{An infinite dimensional KAM theorem and its application to the two dimensional cubic Schršdinger equation.}  Adv. Math. 226 (2011), no. 6, 5361Ð5402. 


\bibitem{GG2} P. Gerard, S. Grellier, \textit{The cubic Szeg\"o equation.} Ann. Sci. Ec. Norm. Sup\'er. (4) 43 (2010), no. 5, 761--810.

\bibitem{GG1} P. G\'erard, S. Grellier, \textit{Invariant tori for the cubic Szeg\"o equation.} Invent. Math. 187 (2012), no. 3, 707--754.

\bibitem{GG3} P. G\'erard, S. Grellier, \textit{Effective integrable dynamics for a certain nonlinear wave equation}, Anal. PDE 5 (2012), no. 5, 1139--1155.

\bibitem{Guardia}
M. Guardia and V. Kaloshin, \textit{Growth of Sobolev norms in the cubic defocusing nonlinear Schr\"odinger equation.}
arxiv.org:1205.5188


\bibitem{Hani} Z.~Hani, \emph{Long-time instability and unbounded Sobolev orbits for some periodic nonlinear Schr\"odinger equations}. To appear in ARMA. Preprint arXiv:1210.7509 [math.AP].

\bibitem{HaPau} Z.~Hani, B. Pausader, \emph{On scattering for the quintic defocusing nonlinear Schr\"odinger equation on $\R \times \T^2$}, to appear in CPAM, arXiv:1205.6132.

\bibitem{HW} Hardy, G. H.; Wright, E. M., \emph{An introduction to the theory of numbers}, Sixth edition. Oxford University Press, Oxford, 2008. xxii+621 pp. ISBN: 978-0-19-921986-5.

\bibitem{Hassel1} K. Hasselmann,\emph{On the non-linear energy transfer in a gravity-wave spectrum. I. General theory.} J. Fluid Mech. 12 (1962) 481--500. 

\bibitem{Hassel2} K. Hasselmann,\emph{On the non-linear energy transfer in a gravity wave spectrum. II. Conservation theorems; wave-particle analogy; irreversibility.} J. Fluid Mech. 15 (1963) 273--281. 

\bibitem{HZ} D. Hundertmark, V. Zharnitsky, \textit{On sharp Strichartz inequalities in low dimensions.} Int. Math. Res. Not. 2006.

\bibitem{Jarnick} V. Jarnick, \textit{\"Uber die Gitterpunkte auf konvexen Curven}, Math. Z. 24 (1926), 500--518.

\bibitem{Kartashova0} E. Kartashova, \emph{Wave resonances in systems with discrete spectra}, AMS Transl. (2), 182 (1998).

\bibitem{Kartashova3} E. Kartashova, \emph{Exact and quasiresonances in discrete water wave turbulence}. Physical review letters 98.21 (2007): 214502.

\bibitem{KNR} E.~Kartashova, S.~Nazarenko, O.~Rudenko, \emph{Resonant interactions of nonlinear water waves in a finite basin}. Physical Review E 78.1 (2008): 016304.

\bibitem{Kartashova} E. Kartashova, \textit{Discrete wave turbulence}, Europhysics Letters 87.4 (2009): 44001.

\bibitem{Kartashova2} E.~Kartashova, \emph{Nonlinear Resonance Analysis.} Cambridge University Press, 1 (2010).

\bibitem{KTV} R.~Killip, T.~Tao, M.~Visan, \textit{The cubic nonlinear Schr\"odinger equation in two dimensions with radial data}, J. Eur. Math. Soc. (JEMS) 11 (2009), no. 6, 1203--1258. 

\bibitem{Kishi} N. Kishimoto, \emph{Remark on the periodic mass critical nonlinear Schr\"odinger equation}, preprint arXiv:1203.6375 [math.AP].

\bibitem{Lieb} E. Lieb, \textit{Gaussian kernels have only Gaussian maximizers.} Invent. Math. 102 (1990), no. 1, 179--208. 

\bibitem{LS} J. Lukkarinen, H. Spohn, \textit{Weakly nonlinear wave equations with random initial data}, Invent. Math. 183 (2011), 79-188.

\bibitem{LN} V.~S.~L'vov, S. Nazarenko, \emph{Discrete and mesoscopic regimes of finite-size wave turbulence}, Physical Review E 82, 056322 (2010).

\bibitem{MMT} A. Majda, D. McLaughlin, E. Tabak, \textit{A one-dimensional model for dispersive wave turbulence.} J. Nonlinear Sci. 7 (1997), no. 1, 9--44.

\bibitem{MV} F. Merle, L. Vega, \textit{Compactness at blow-up time for L2 solutions of the critical nonlinear
Schr\"odinger equation in 2D}, Int. Math. Res. Not. 8 (1998), 399--425.

\bibitem{Murat} F. Murat, \textit{Compacit\'e par compensation}, Ann. Scuola Norm. Sup. Pisa Cl. Sci., 5(4), 1978.

\bibitem{Nazarenko} S. Nazarenko, \textit{Wave turbulence}, Lecture notes in Physics 825, Springer, Heidelberg, 2011.

\bibitem{Nazarenko2} S.~Nazarenko,\emph{Sandpile behaviour in discrete water-wave turbulence}. Journal of Statistical Mechanics: Theory and Experiment 2006.02 (2006): L02002.

\bibitem{Peierls} R. Peierls, \emph{The kinetic theory of thermal conduction in crystals},  Annalen der Physik 3 (8): 1055--1101 (1929).

\bibitem{P1} O. Pocovnicu, \textit{Traveling waves for the cubic Szeg\"o equation on the real line}, Analysis \& PDE, 4-3 (2011), 379--404.

\bibitem{P2} O. Pocovnicu, \textit{Explicit formula for the solutions of the Szeg\"o equation on the real line and applications}, Discrete Contin. Dynam. Syst. A 31 (2011) no. 3, 607--649.

\bibitem{Procesi} M.~Procesi, C.~Procesi, \textit{A KAM algorithm for the resonant non--linear Schr\"odinger equation}. Preprint arXiv:1211.4242 [math.AP].

\bibitem{Spohn} H. Spohn, \textit{The phonon Boltzmann equation, properties and links to weakly anharmonic lattice dynamics}, J. Stat. Physics 124 (2--4), 1041--1104 (2006).

\bibitem{Strichartz} R. Strichartz, \textit{Restrictions of Fourier transforms to quadratic surfaces and decay of solutions of wave equations.} Duke Math. J. 44 (1977), no. 3, 705--714. 

\bibitem{TaoBook} T. Tao, \textit{Nonlinear dispersive equations: Local and global analysis}, CBMS Regional Conference Series in Mathematics, 106. Published for the Conference Board of the Mathematical Sciences, Washington, DC; by the American Mathematical Society, Providence, RI, 2006. xvi+373 pp. ISBN: 0-8218-4143-2.
 
\bibitem{TaoPL} T. Tao, \textit{Poincar\'e's legacies, pages from year two of a mathematical blog. Part I}, American Mathematical Society, Providence, RI, 2009. x+293 pp. ISBN: 978-0-8218-4883-8.

\bibitem{TY} M. Tanaka, N. Yokoyama, \emph{Effects of discretization of the spectrum in water-wave turbulence}, Fluid Dynamics Research 34 (2004) 199--216.
 
\bibitem{Tartar} L. Tartar, \textit{Compensated compactness and applications to partial differential equations.} Heriot-Watt Sympos., vol. IV, Pitman, New York, 1979.

\bibitem{Tomas} P. Tomas, \textit{A restriction theorem for the Fourier transform.} Bull. Amer. Math. Soc. 81 (1975), 477--478.

\bibitem{Zakh1} V.~E.~Zakharov, N.~N.~Filonenko, \emph{Energy spectrum for the stochastic oscillations of a fluid surface}. Kohl. Acad. Nauk SSSR 170, 1292--1295 (1966) (translation: Sov. Phys. Dokl. 11, 881--884 (1967)).

\bibitem{Zakh2} V.~E.~Zakharov, N.~N.~Filonenko, \emph{Weak turbulence of capillary waves}, Journal of Applied Mechanics and Technical Physics, Volume 8, Issue 5, pp.37--40 (1967).

\bibitem{ZLF} V.~E.~Zakharov, V. L'vov, G. Falkovich, \textit{Kolmogorov spectra of turbulence 1. Wave turbulence.} Springer, Berlin, 1992.

\bibitem{Zakh3} V.~E.~Zakharov, F.~Dias, A.~Pushkarev, \emph{One-dimensional wave turbulence}. Physics Reports 398.1 (2004): 1--65.

\bibitem{ZKP} V.~E.~Zakharov, A.~O.~Korotkevich, A.~N.~Pushkarev, A.~I.~Dyachenko, \emph{Mesoscopic wave turbulence}, 82, 8, 487--491 (2005).

\end{thebibliography}
\end{document}